\documentclass[a4paper, reqno, twoside,12p]{amsart}
\usepackage{etex}
\usepackage{amsfonts,amssymb,graphics,amsthm,amsmath}
\usepackage{latexsym}
\usepackage[2emode]{psfrag}
\usepackage{mathrsfs}
\usepackage[all]{xy}
\usepackage{enumerate}
\usepackage[latin1]{inputenc}
\usepackage{lscape}
\usepackage{a4wide}
\usepackage[bookmarks=false]{hyperref}

\usepackage{txfonts}

\usepackage[all]{pstricks}
\usepackage{pst-text,pst-node,pst-coil}
\usepackage{tikz}

\usepackage{cancel}
\usepackage[normalem]{ulem}
\usetikzlibrary{arrows,shapes,backgrounds}

\usepgflibrary{arrows}
\usetikzlibrary{topaths}
\usetikzlibrary{er}

\newtheorem{theorem}{\sc Theorem}[section]
\newtheorem{proposition}[theorem]{\sc Proposition}

\newtheorem{lemma}[theorem]{\sc Lemma}
\newtheorem{corollary}[theorem]{\sc Corollary}
\theoremstyle{definition}
\newtheorem{definition}[theorem]{\sc Definition}

\newtheorem{example}[theorem]{\sc Example}

\theoremstyle{remark}
\newtheorem{remark}[theorem]{\sc Remark}

\newcommand{\tensor}[1]{\otimes_{\scriptscriptstyle{#1}}}
\newcommand{\cotensor}[1]{\, \square_{\scriptscriptstyle{#1}} \,}

\newcommand{\Sf}[1]{\mathsf{#1}}
\newcommand{\fk}[1]{\mathfrak{#1}}

\newcommand{\rmod}[1]{\Sf{Mod}_{#1}}

\renewcommand{\hom}[3]{\mathrm{Hom}_{\Sscript{#1}}\left(#2,\,#3\right)}
\newcommand{\cohom}[3]{\mathrm{Hom}^{\Sscript{#1}}\left(#2,\,#3\right)}

\newcommand{\td}[1]{\widetilde{#1}}

\newcommand{\bara}[1]{\overline{#1}}

\newcommand{\End}[2]{\mathrm{End}_{\Sscript{#1}}(#2)}
\newcommand{\CoEnd}[2]{\mathrm{End}^{\Sscript{#1}}(#2)}

\newcommand{\rcomod}[1]{ \mathsf{Comod}{}_{#1}}
\newcommand{\frcomod}[1]{ \mathsf{comod}{}_{#1}}
\newcommand{\can}[1]{\mathsf{can}_{#1}}

\newcommand{\LR}[1]{\left\{\underset{}{} #1 \right\}}

\newcommand{\B}[1]{\boldsymbol{#1}}
 
 \newcommand{\id}{\mathrm{Id}}

\newcommand{\Spec}[1]{{\rm Spec}(#1)}

\newcommand{\aA}{\mathscr{A}}
\newcommand{\bB}{\mathscr{B}}

\newcommand{\gG}{\mathscr{G}}
\newcommand{\hH}{\mathscr{H}}

\newcommand{\kK}{\mathscr{K}}
\newcommand{\lL}{\mathscr{L}}

\newcommand{\oO}{\mathscr{O}}
\newcommand{\pP}{\mathscr{P}}

\newcommand{\rR}{\mathscr{R}}
\newcommand{\sS}{\mathscr{S}}

\newcommand{\uU}{\mathscr{U}}

\newcommand{\cA}{{\mathcal A}}

\newcommand{\cG}{{\mathcal G}}
\newcommand{\cH}{{\mathcal H}}

\newcommand{\cJ}{{\mathcal J}}
\newcommand{\cK}{{\mathcal K}}

\newcommand{\cM}{{\mathcal M}}

\newcommand{\cR}{{\mathcal R}}
\newcommand{\cS}{{\mathcal S}}

\newcommand{\Sscript}[1]{\scriptscriptstyle{#1}}
\newcommand{\due}[3]{{}_{{#2 }} {#1}_{{ #3}}\,}

\newcommand{\Algk}{{\rm Alg}_{\Sscript{\Bbbk}}}
\newcommand{\Algl}{{\rm Alg}_{\Sscript{L}}}
\newcommand{\hHo}{\hH_{\Sscript{0}}}
\newcommand{\hHa}{\hH_{\Sscript{1}}}
\newcommand{\AL}{A_{\Sscript{L}}}
\newcommand{\HL}{\cH_{\Sscript{L}}}
\newcommand{\hHL}{\hH_{\Sscript{L}}}
\newcommand{\gGL}{\gG_{\Sscript{L}}}
\newcommand{\oOL}{\oO_{\Sscript{L}}}
\newcommand{\hHq}{\hH_{\Sscript{q}}}
\newcommand{\Hq}{\cH_{\Sscript{q}}}
\newcommand{\Hp}{\cH_{\Sscript{p}}}
\newcommand{\sL}{{\sf{s}}_{\Sscript{L}}}
\newcommand{\tL}{{\sf{t}}_{\Sscript{L}}}
\newcommand{\eL}{\varepsilon_{\Sscript{L}}}
\newcommand{\Lq}{L_{\Sscript{q}}}

\begin{document}
\allowdisplaybreaks

\title[On geometrically transitive Hopf algebroids]{On geometrically transitive Hopf algebroids}

\author{Laiachi El Kaoutit}
\address{Universidad de Granada. Departamento de \'{A}lgebra  y IEMath-Granada. Facultad de Educaci\'{o}n, Econon\'ia y Tecnolog\'ia de Ceuta. Cortadura del Valle, s/n. E-51001 Ceuta, Spain}
\email{kaoutit@ugr.es}
\urladdr{http://www.ugr.es/~kaoutit/}

\date{\today}
\thanks{Research supported by the Spanish Ministerio de Econom\'{i}a y Competitividad  and FEDER, grants MTM2016-77033-P and MTM2013-41992-P}

\begin{abstract}
This paper contributes to the characterization of a certain class of commutative Hopf algebroids. It is shown that a commutative flat Hopf algebroid with a non zero base ring and a nonempty character groupoid is geometrically transitive if and only if any base change morphism is a weak equivalence (in particular, if any extension of the base ring is Landweber exact), if and only if any trivial bundle is a principal  bi-bundle, and if and only if any two objects are fpqc locally isomorphic. As a consequence,  any  two isotropy Hopf algebras of a geometrically transitive Hopf algebroid (as above) are weakly equivalent. Furthermore, the character groupoid is transitive and any two isotropy Hopf algebras are conjugated. Several other characterizations of these Hopf algebroids in relation to transitive groupoids are also given.
\end{abstract}
\vspace{-1cm}

\maketitle
\vspace{-1cm}
\begin{small}
\tableofcontents
\end{small}

\pagestyle{headings}

\vspace{-1.4cm}
\section{Introduction}
\subsection{Motivation and overview}\label{ssec:1}   
A commutative Hopf algebroid can be thought as an  affine groupoid scheme, that is,  a groupoid scheme \cite[D\'efinition page 299]{DemGab:GATIGAGGC} in which the schemes defining objects and morphisms are affine schemes. In other words, this is a representable presheaf of groupoids in the category of affine schemes, or a  prestack of groupoids whose "stackification" leads  to a stack in the \emph{fpqc} (fid\`element plate et quasi-compacte)  topology. 
For instance, an action of an affine group scheme  on an affine scheme leads to an affine groupoid scheme which gives rise (by passage to the coordinate rings) to a commutative Hopf algebroid, known as \emph{split Hopf algebroid} (\cite[Appendix A.1]{Ravenel:1986}, see also \cite{Landweber:1973}). More examples  of commutative Hopf algebroids can be performed using Set-theoretically  constructions in groupoids.

Hopf algebroids in relation with groupoids are fundamental objects in both algebraic topology and algebraic geometry. They appear in the study of stable homotopy theory \cite{Ravenel:1986,HovStr:CALEHT, Naumann:07,Powell:2008} (see also the references therein), and prove to be very useful in  studying quotients of preschemes, prestacks of groupoids over affine schemes as well as (commutative) Tannakian categories \cite{Gabriel:LNM151, Deligne:1990, Bruguieres:1994, Alonso/all:2014, Laumon/Moret-Bailly:2002}.  

As in the case of affine group schemes \cite{DemGab:GATIGAGGC},   
several constructions and results on groupoids  have  a certain geometric meaning in presheaves of groupoids and then a possible  algebraic interpretation at the level of  Hopf algebroids.  In this way,  Hopf algebroids are better understood  when looking to  classical results in groupoids, or  by mimicking well-known results on classical geometric groupoids, e.g.~ topological or Lie groupoids.  
The main motivation  of this paper fits into these lines of research. Specifically,  we  study a class of commutative Hopf algebroids called \emph{geometrically transitive} (see below),  by means of transitive  groupoids and their properties, obtaining in this way  several new characterizations of these Hopf algebroids. Besides,  much of the properties of transitive groupoids hereby developed and used in the study of Hopf algebroids   can be also seen as a contribution to the theory of groupoids.

The notion of transitivity varies depending on the context. In groupoid theory,  a  (small) groupoid  is said to be \emph{transitive} when the cartesian  product of the source and the target is a surjective map. A Lie groupoid is called \emph{locally trivial} (or a \emph{transitive Lie groupoid}), when this map is a surjective submersion \cite{Mackenzie:2005, Cartier:2008}. For  groupoid schemes, the meaning of the abstract notion of transitivity was introduced by Deligne in \cite[page 114]{Deligne:1990}. More precisely,   a groupoid scheme  is \emph{transitive} in the fpqc topology sense if the morphism constituted by the fibred product of the source and the target  is a cover in this topology.  
In  \cite[D\'efinition page 5850]{Bruguieres:1994}, Brugui\`eres introduced a  class of  Hopf algebroids referred to as \emph{geometrically transitive}, where he showed (\cite[Th\'eor\`eme 8.2 page 5858]{Bruguieres:1994}) that in the commutative case (the case which we are interested in) these are Hopf algebroids whose associated affine groupoid schemes are transitive in the fpqc sense. It is also implicitly shown in \cite{Bruguieres:1994} that a commutative Hopf algebroid is geometrically transitive  if and only if the unit map (i.e., the tensor product of the source and the target) is a faithfully flat extension.  This, in fact, can be thought of as a proper definition of geometrically transitive (GT for short) commutative  Hopf algebroids. Nevertheless, we will use here the original definition of \cite{Bruguieres:1994} (see Definition \ref{def:geo-tran} below) and show using elementary methods that the faithful flatness of the unit characterizes in fact GT  Hopf algebroids with nonempty character groupoids (see Theorem \ref{thm:AI} below).

Transitive groupoids are also characterized by the fact that any two objects are isomorphic, or equivalently: a groupoid with only one connected component, or  \emph{connected groupoid} in the terminology of \cite{Higgins:1971}.  From the geometric point of view, that is,  for presheaves of groupoids, this means that any two objects are locally isomorphic in the fpqc topology, see \cite[Proposition 3.3]{Deligne:1990}. At the level of Hopf algebroids, this property can be directly  expressed in terms of faithfully flat extensions (see Definition \ref{def:1} below), which in turn characterizes  GT Hopf algebroids with nonempty character groupoids, as we will see in the sequel by using elementary (algebraic) arguments.  Our methods, in fact, lead us also to recover other results on equivalences of categories stated in \cite[\S 3.5]{Deligne:1990} (see the paragraph after Theorem \ref{thm:AI}).

From their very definition one can then see that GT  Hopf algebroids can be understood as a natural algebro-geometric substitute of transitive groupoids. Under this point of view, it is reasonable to expect that most of the properties or characterizations of transitive groupoids could have an analogous counterpart at the level of GT Hopf algebroids.  However, apart from the basic definition,  
there are still several characterizations of transitive groupoids which, up to our knowledge,  are not known for GT  Hopf algebroids. In the following, we describe the two most interesting  of these characterizations.

A perhaps well-known result (see Proposition \ref{prop:grpd} for details) says that a groupoid $\gG$: $\xymatrix@C=35pt{G_{\Sscript{1}}\ar@<0.80ex>@{->}|-{\scriptscriptstyle{\sf{s}}}[r] \ar@<-0.80ex>@{->}|-{\scriptscriptstyle{\sf{t}}}[r] & \ar@{->}|-{ \scriptscriptstyle{\iota}}[l]G_{\Sscript{0}}}$
is  transitive if and only if for any map $\varsigma: X \to G_{\Sscript{0}}$ the induced morphism of groupoids $\gG^{\Sscript{\varsigma}} \to \gG$ is a weak equivalence (i.e., an essentially surjective and fully faithful functor), where $\gG^{\Sscript{\varsigma}}$ is the \emph{induced groupoid} 
whose set of objects is  $X$ and its set of arrows is  the fibred product $ X \,  \due \times {\Sscript{\varsigma}} {\, \Sscript{\Sf{t}}} \, G_{\Sscript{1}} \,  \due \times {\Sscript{\Sf{s}}} {\, \Sscript{\varsigma}}  \, X $, that is, $\gG^{\Sscript{\varsigma}}$ is the pull-back groupoid of $\gG$ along $\varsigma$ (see \cite[\S 2.3]{Mackenzie:2005} and \cite{Higgins:1971} for a dual construction). 
Another more interesting and perhaps not yet known characterization of the transitivity is by means of \emph{principal groupoid-bisets}; for the precise definition see Definitions  \ref{def:Gset},   \ref{def:biset}, and \ref{def:pbset}.   This notion  is in fact an abstract formulation of the notion of \emph{principal bi-bundles} in the context of topological and Lie groupoids \cite{Moedijk/Mrcun:2005, Jelenc:2013}, or that of \emph{bi-torsors} in sheaf theory \cite{DemGab:GATIGAGGC, Giraud:1971},  which is  of course based on the natural generalization of the notion of group-bisets \cite{Bouc:2010} to the context of groupoids. The aforementioned characterization can be expressed as follows: a groupoid $\gG$ is transitive if and only if  for any map $\varsigma: X \to G_{\Sscript{0}}$ the pull-back groupoid-biset $ G_{\Sscript{1}} \,\due \times {\Sscript{\Sf{s}}} {\, \Sscript{\varsigma}}  \, X $ is a principal $(\gG,\gG^{\Sscript{\varsigma}})$-biset, see again Proposition \ref{prop:grpd}.  

The  main aim of this paper is to investigate GT  Hopf algebroids by means of transitive groupoids. Our aim is in part to  see how the previous characterizations of  transitive groupoids can be transferred,  by means of weak equivalences and principal groupoids-bisets, to the commutative Hopf algebroid framework.

\subsection{Description of the main results}\label{ssec:raroes} Let $\Bbbk$ be the ground field. The term algebra in the following stands for a commutative $\Bbbk$-algebra, as usual the unadorned tensor product $\tensor{}$ denotes the tensor product over $\Bbbk$. 

Our main result is summarized in the following theorem, which includes Theorem \ref{thm:A} below:

{\renewcommand{\thetheorem}{{\bf A}}
\begin{theorem}\label{thm:AI}
Let $(A,\cH)$ be a commutative flat Hopf algebroid over $\Bbbk$ with $A\neq 0$ and denote by $\hH$ its associated presheaf of groupoids. Assume that $\hHo(\Bbbk) \neq \emptyset$.  Then the following are equivalent:
\begin{enumerate}[(i)]
\item $\etaup=\Sf{s}\tensor{}\Sf{t}: A\tensor{}A \to \cH$ is a  faithfully flat extension;
\item Any two objects of $\mathscr{H}$ are fpqc locally isomorphic (see Definition \ref{def:1});
\item For any extension $\phi: A \to B$, the  extension $\alpha: A \to \cH_{\Sscript{\Sf{t}}}\tensor{A}{}_{\Sscript{\phi}}B$, $a \mapsto \Sf{s}(a)\tensor{A}1_{\Sscript{B}}$ is  faithfully flat;
\item $(A,\cH)$ is geometrically transitive (see Definition \ref{def:geo-tran});
\item For any extension $\phi: A \to B$, the canonical morphism 
${\B \phi}:(A,\cH) \to (B,\cH_{\Sscript{\phi}})$ of Hopf algebroids is a weak equivalence, that is, the 
induced functor $\phi_{*}: \rcomod{\cH} \to \rcomod{\cH_{\Sscript{\phi}}}$  is an equivalence of  symmetric monoidal $\Bbbk$-linear categories of comodules, where $\cH_{\Sscript{\phi}}=B\tensor{A}\cH\tensor{A}B$;
\item For any extension $\phi: A \to B$,  the trivial principal left $(\cH, \cH_{\Sscript{\phi}})$-bundle $\cH\tensor{A}B$ is a principal bi-bundle (see subsection \ref{ssec:W}).
\end{enumerate} 
\end{theorem}
}

The flatness is to ensure that the categories of comodules of the involved Hopf algebroids are Grothendieck ones with exact forgetful functor to the category of modules over the base algebra. As for the hypothesis $A\neq 0$, one of the conditions  in Definition  \ref{def:geo-tran} below says that the endomorphisms ring of $A$ viewed as a comodule should coincides with $\Bbbk$, so it is reasonable to ask $A$ to be a non zero object.  

The examples of Hopf algebroids we have in mind and satisfy the assumptions of Theorem \ref{thm:AI}, are the ones  which can be obtained from Tannaka's reconstruction process, by using  $\Bbbk$-linear representations of a (small) groupoid $\gG$ and taking $A=Map(G_{\Sscript{0}}, \Bbbk)$, the set of all maps from $G_{\Sscript{0}}$ to $\Bbbk$, as the base algebra. 

The proof of Theorem \ref{thm:AI} is done by showing the implications $(i) \Rightarrow (ii) \Rightarrow (iii) \Rightarrow (iv) \Rightarrow (i)$, and using the equivalences $(iii)\Leftrightarrow (v) \Leftrightarrow(vi)$ from \cite[Proposition 5.1]{Kaoutit/Kowalzig:14}. 
The assumption $\hHo(\Bbbk) \neq \emptyset$ which ensures that the character groupoid $\hH(\Bbbk)$ is nonempty,  is needed to prove the implication $(iii) \Rightarrow (iv)$. Although, the implication $(ii) \Rightarrow (iii)$ can be shown  under the weaker assumption $\hHo(L) \neq \emptyset$ for some field extension $L$ of $\Bbbk$. Moreover, the equivalent conditions of Theorem \ref{thm:AI} are stable under change of scalars for the class of Hopf algebroids  with nonempty character groupoids (i.e.,  with $\hHo(\Bbbk) \neq \emptyset$). Indeed, if  $L$ is a field extension of $\Bbbk$ and $(A, \cH)$ is a Hopf algebroid which verifies the assumptions of Theorem \ref{thm:AI} and satisfies one of  these conditions say $(i)$, then the Hopf algebroid  $(\AL, \HL)=(A\tensor{\Bbbk}L, \cH\tensor{\Bbbk}L)$ satisfies this condition as well. However, if  $\hHo(\Bbbk) = \emptyset$ and there is a field extension $L$ such that $\hHo(L) \neq \emptyset$, then it is not clear to us how to show the implication $(iii) \Rightarrow (iv)$,   see Remarks \ref{rem:Lkk} and \ref{rem:Lk} for more comments on the change of scalars.

As we have mention before, saying that $(A,\cH)$ is GT Hopf algebroid is equivalent to say that the $\Bbbk$-groupoid  $\hHa$ acts transitively on $\hHo$ in the fpqc sense. Thus, by Theorem  \ref{thm:AI}, this  can be now easily deduced by comparing condition $(i)$ with the Definition of \cite[\S 1.6]{Deligne:1990}.  On the other hand, there is also a notion of transitivity for Hopf algebroids \cite[D\'efinition page 5850]{Bruguieres:1994} and as was shown in \cite[Proposition 7.3 page 5851]{Bruguieres:1994}, a Hopf algebroid $(A,\cH)$ is geometrically transitive if and only if $(\AL,\HL)$ is transitive for any field extension $L$ of $\Bbbk$. Perhaps this justifies the terminology, although it is not clear, at least to us,  how to express this transitivity at the level of the associated presheaves of groupoids with respect to a certain topology,  see  Remark  \ref{rem:Transitif} for more details. 

Condition $(v)$ in Theorem \ref{thm:AI} implies in particular that any algebra extension $B$ is \emph{Landweber exact} over $A$ in the sense of \cite[Definition 2.1]{HovStr:CALEHT} and shows  that certain GT Hopf algebroids do not have a non trivial hereditary torsion theory in the sense of \cite[Theorem A]{HovStr:CALEHT}.  On the other hand, following the notation of \cite[\S 3.5]{Deligne:1990}, we know that the category of comodules $\rcomod{\cH}$ is canonically equivalent, as  a symmetric monoidal $\Bbbk$-linear category, to the category of linear representations ${\rm Rep}(\hHo:\hHa)$ of the associated affine presheaf of groupoids. So, up to these canonical equivalences, condition $(v)$ gives the "affine version" of the equivalence of categories stated in \cite[(3.5.1) page 130]{Deligne:1990}. Furthermore, Theorem \ref{thm:AI} shows that for affine $\Bbbk$-schemes  with the induced fpqc topology, the equivalence of categories  stated in \emph{op.cit.},  is not only a necessary condition to have a transitive action (for the class of affine $\Bbbk$-groupoids $\hH$ with $\hHo(\Bbbk) \neq \emptyset$), but also a sufficient one. Thus we obtain a new characterization of these transitive affine $\Bbbk$-groupoids.

The fact that transitive groupoids are characterized by the  conjugacy of  their isotropy groups, and the analogue of this characterization in the Hopf algebroid context  also  attracted our attention. More precisely, given a Hopf algebroid $(A,\cH)$ and denote by $\hH(C)$ the fiber of $\hH$ at a  commutative algebra $C$, that is, the groupoid with set of objects $\hHo(C)=\Algk(A, C)$ and set of arrows $\hHa(C)=\Algk(\cH,C)$ (see equation \eqref{Eq:miacosa} below). Assume as above that  the character groupoid $\hH(\Bbbk)$ in nonempty (i.e., $\hH_{\Sscript{0}}(\Bbbk) \neq \emptyset$), then for any object $x: A \to \Bbbk$, there is a presheaf  of sets which assigns to each  algebra $C$, the isotropy group of the object $x^*(1_{\Sscript{C}}) \in \hH_{\Sscript{0}}(C)$, where $1_{\Sscript{C}}: \Bbbk \to C$ denotes the unit of $C$.
It turns out that this is an affine group scheme represented by the Hopf algebra  $(\Bbbk_{\Sscript{x}}, \cH_{\Sscript{x}})$ which is the base change of $(A,\cH)$ by the algebra map $x$ (here $\Bbbk_{\Sscript{x}}$ denotes $\Bbbk$ viewed as an $A$-algebra via $x$, and $\cH_{\Sscript{x}}=\Bbbk_{\Sscript{x}}\tensor{A}\cH\tensor{A}\Bbbk_{\Sscript{x}}$).  
The  pair $(\Bbbk_{\Sscript{x}}, \cH_{\Sscript{x}})$ is refereed to as the \emph{isotropy Hopf algebra} of $(A,\cH)$  at (the point) $x$. Now, recall from \cite{HovStr:CALEHT} that  two flat Hopf algebroids $(A, \cH)$  and $(B,\cK)$ are weakly equivalent  if there is a diagram of weak equivalences 
$$
\xymatrix@R=7pt{  &  (C, \cJ) & \\ (A, \cH) \ar@{->}[ru] & & \ar@{->}[lu] (B,\cK),}
$$ 
see Subsection \ref{ssec:W} for more details. 
The fact that two isotropy groups of  a transitive groupoid are isomorphic is  translated to the fact that  two isotropy Hopf algebras of a GT Hopf algebroid are weakly equivalent; of course, Hopf algebras are considered here as Hopf algebroids where source and target coincide.

This result is part of the subsequent corollary of Theorem \ref{thm:AI} which contains both  Proposition \ref{prop:weak-isotropy} and Corollary \ref{coro:DFR}; the last statement in part \emph{(2)} below follows from Propositions  \ref{prop:GT}\emph{(GT13)} and \ref{prop:dualizable}.

{\renewcommand{\thetheorem}{{\bf A}}
\begin{corollary}  
Let $(A,\cH)$ be a flat Hopf algebroid as in Theorem \ref{thm:AI} with $\hH_{\Sscript{0}}(\Bbbk) \neq \emptyset$. Assume that $(A,\cH)$ is  geometrically transitive. Then 
\begin{enumerate}[(1)]
\item Any two isotropy Hopf algebras are weakly equivalent.
\item Any dualizable (right) $\cH$-comodule is locally free of constant rank. Moreover, any right $\cH$-comodule is an inductive limit of dualizable right $\cH$-comodules.
\end{enumerate}
\end{corollary}
}

The notion of  the conjugation relation between two isotropy Hopf algebras  is not automatic. This relation can be formulated by using 2-isomorphisms in the  2-category of flat Hopf algebroids. More specifically, using the notations and the assumptions of Theorem \ref{thm:AI}, for a given Hopf algebroid $(A,\cH)$, two isotropy Hopf algebras $(\Bbbk_{\Sscript{x}}, \cH_{\Sscript{x}})$ and $(\Bbbk_{\Sscript{y}}, \cH_{\Sscript{y}})$, at the points $x,\,  y\, \in \, \hH_{\Sscript{0}}(\Bbbk)$ are said to be  \emph{conjugated} provided there is an isomorphism $\Sf{g}:   (\Bbbk_{\Sscript{x}}, \cH_{\Sscript{x}}) \to (\Bbbk_{\Sscript{y}}, \cH_{\Sscript{y}})$ of Hopf algebras such that the following diagram 
$$
\xymatrix@R=7pt{ (\Bbbk_{\Sscript{x}},\cH_{\Sscript{x}}) \ar@{->}^-{\Sf{g}}[rr] & & (\Bbbk_{\Sscript{y}},\cH_{\Sscript{y}}) \\ & \ar@{->}^-{\Sf{x}}[lu] (A,\cH) \ar@{->}_-{\Sf{y}}[ru] &}
$$
commutes up to a $2$-isomorphism, where $\Sf{x}$ and $\Sf{y}$ are the canonical morphisms attached, respectively,  to $x$ and $ y$.
The transitivity of the conjugation relation characterizes  in fact  the transitivity of the character groupoid. This result is also a corollary of Theorem \ref{thm:AI} and stated as Proposition \ref{prop:conjugation} below: 

{\renewcommand{\thetheorem}{{\bf B}}
\begin{corollary}\label{coro:B}  
Let $(A,\cH)$ be a flat Hopf algebroid as in Theorem \ref{thm:AI} with $\hH_{\Sscript{0}}(\Bbbk) \neq \emptyset$. Assume that $(A,\cH)$ is  geometrically transitive. Then the following  are equivalent 
\begin{enumerate}[(i)]
\item the character groupoid $\mathscr{H}(\Bbbk)$ is transitive;
\item for any two objects $x, y$ in $\mathscr{H}_{\Sscript{0}}(\Bbbk)$, the left $\cH$-comodule algebras $\cH\tensor{A}\Bbbk_{\Sscript{x}}$ and $\cH\tensor{A}\Bbbk_{\Sscript{y}}$ are isomorphic;
\item any two isotropy Hopf algebras are conjugated.
\end{enumerate}
Furthermore, under the stated assumptions, we have that condition $(i)$ is always fulfilled. 
\end{corollary}
}

Let  $(A,\cH)$  be a flat Hopf algebroid  over $\Bbbk$ with $A \neq 0$ and  $\hHo(\Bbbk) = \emptyset$. If  $(A,\cH)$ satisfies condition $(i)$ of Theorem \ref{thm:AI} and there exists a field extension $L$ of $\Bbbk$ such that $\hHo(L) \neq \emptyset$, then Corollary \ref{coro:B} can be applied to $(\AL, \HL)$  and implies that $\hH(L)$ is a transitive groupoid, see Proposition \ref{prop:Lk}. 

Transitive groupoids are related to principal group-bisets, in the sense that there is a (non canonical) correspondence between these two notions, see Subsection \ref{ssec:GPS}. This in fact is an abstract formulation of Ehresmann's well-known result dealing with the correspondence between transitive Lie groupoids and principal fibre bundles, as was expounded in  \cite{Cartier:2008}. 
The analogous correspondence at the level of Hopf algebroids is not always possible and some technical assumptions are required.  The  formulation of this  result is given as follows.

For any object $x \in \hH_{\Sscript{0}}(\Bbbk)$, consider the presheaf of sets which associates to each algebra $C$  the set $\pP_{\Sscript{x}}(C) :=\Sf{t}^{-1}\big( \{ x^*(1_{\Sscript{C}})\}\big)$, where $\Sf{t}$ is the target of the groupoid $\hH(C)$. In the terminology of \cite{Higgins:1971}, this is the \emph{left star set} of the object $ x^*(1_{\Sscript{C}})$.
Denote by $\alpha_{\Sscript{x}}: A \to P_{\Sscript{x}}:=\cH\tensor{A}\Bbbk_{\Sscript{x}}$ the algebra map which sends $a \mapsto \Sf{s}(a)\tensor{}1$. It turns out that  the presheaf of sets $\mathscr{P}_{\Sscript{x}}$ is affine, and is up to a natural isomorphism represented by the  $(\cH,\cH_{\Sscript{x}})$-bicomodule algebra $P_{\Sscript{x}}$.
The subsequent is a  corollary of Theorem \ref{thm:AI}, and formulates the desired result. It is  a combination of Lemma \ref{lema:P} and Proposition \ref{prop:GTPB} below.

{\renewcommand{\thetheorem}{{\bf C}}
\begin{corollary} Let $(A,\cH)$ be a flat Hopf algebroid as in Theorem \ref{thm:AI} with $\hH_{\Sscript{0}}(\Bbbk) \neq \emptyset$.   If $(A,\cH)$ is a GT Hopf algebroid, then for any object $x \in \hH_{\Sscript{0}}(\Bbbk)$, the comodule algebra $(P_{\Sscript{x}}, \alpha_{\Sscript{x}})$ is a right  principal $\cH_{\Sscript{x}}$-bundle (i.e., a Hopf-Galois extension).

Conversely, let $(P,\alpha)$ be a right principal $B$-bundle over a Hopf $\Bbbk$-algebra $B$ with extension $\alpha: A \to P$. Denote by $\upsilonup: \cH:=(P\tensor{}P)^{\Sscript{coinv_{B}}} \to P\tensor{}P$ the canonical injection, where $P\tensor{}P$ is a right $B$-comodule algebra via the diagonal coaction  (here $R^{\Sscript{coinv_B}}$ denotes the subalgebra of  coinvariant elements of a right $B$-comodule algebra $R$). Assume that $\upsilonup$ is a faithfully flat extension and that $$\cH\tensor{A}\cH = \Big((P\tensor{}P)\tensor{A}(P\tensor{}P) \Big)^{\Sscript{coinv_B}},$$
where $(P\tensor{}P)\tensor{A}(P\tensor{}P)$ is endowed with a canonical right $B$-comodule algebra structure. 
Then the pair of algebras $(A,\cH)$ admits a unique  structure of a GT Hopf algebroid such that $(\alpha,\upsilonup):(A,\cH) \to (P,P\tensor{}P)$ is a morphism of GT Hopf algebroids.
\end{corollary}
}

\section{Abstract groupoids: General notions and basic properties}\label{sec:Grpd}
This section contains  the results  about  groupoids,
which we want to transfer to the context of Hopf algebroids in the forthcoming sections.  For sake of completeness we include some of their proofs.

\subsection{Notations, basic notions and examples}\label{ssec:basic}
A \emph{groupoid (or abstract groupoid)} is a small category where each morphism is an isomorphism. That is, a pair of two sets $\mathscr{G}:=(G_{\Sscript{1}}, G_{\Sscript{0}})$ with diagram 
$\xymatrix@C=35pt{G_{\Sscript{1}}\ar@<0.70ex>@{->}|-{\scriptscriptstyle{\sf{s}}}[r] \ar@<-0.70ex>@{->}|-{\scriptscriptstyle{\sf{t}}}[r] & \ar@{->}|-{ \scriptscriptstyle{\iota}}[l]G_{\Sscript{0}}}$,
where $\Sf{s}$ and $\Sf{t}$ are resp.~ the source and the target of a given arrow, and $\iota$ assigns to each object its identity arrow; together with an associative and unital multiplication  $G_{\Sscript{2}}:= G_{\Sscript{1}}\, \due  \times {\Sscript{\Sf{s}}} {\, \Sscript{\Sf{t}}} \, G_{\Sscript{1}} \to G_{\Sscript{1}}$  as well as a map $G_{\Sscript{1}} \to G_{\Sscript{1}}$ which associates to each arrow its inverse. 

Given a groupoid $\gG$, consider an object $x \in G_{\Sscript{0}}$, \emph{the isotropy group of $\gG$ at $x$}, is the group of loops:
\begin{equation}\label{Eq:isotropy}
\gG^{\Sscript{x}}:=\Big\{ g \in G_{\Sscript{1}}|\, \Sf{s}(g)=\Sf{t}(g)=x \Big\}.
\end{equation}

Notice that the disjoint union ${\biguplus}_{\Sscript{x\, \in \, G_0}} \gG^{\Sscript{x}}$ of all isotropy groups form  the set of arrows of a subgroupoid  of $\gG$ whose source equal to its target, namely, the projection ${\biguplus}_{\Sscript{x\, \in \, G_0}} \gG^{\Sscript{x}} \to G_{\Sscript{0}}$.

A \emph{morphism of groupoids} $\phi: \mathscr{H} \to \mathscr{G}$ is a functor between the underlying categories. That is, $\phi=(\phi_{\Sscript{0}}, \phi_{\Sscript{1)}})$, where $\phi_{\Sscript{0}}: H_{\Sscript{0}} \to G_{\Sscript{0}}$ and $\phi_{\Sscript{1}}:  H_{\Sscript{1}} \to G_{\Sscript{1}}$ satisfying the pertinent compatibility conditions: 
$$
\phi_{\Sscript{1}} \circ \iota \,=\, \iota \circ \phi_{ \Sscript{0}},\;\; \phi_{\Sscript{0}} \circ \Sf{s} \,=\, \Sf{s} \circ \phi_{ \Sscript{1}},\;\; \phi_{\Sscript{0}} \circ \Sf{t} \,=\, \Sf{t} \circ \phi_{ \Sscript{1}},\;\; \phi_{\Sscript{1}}(fg)= \phi_{\Sscript{1}}(f) \phi_{\Sscript{1}}(g),
$$
whenever the multiplication $fg$ in $H_{\Sscript{1}}$ is  permitted.

Obviously any such a morphism  induces  morphisms between the isotropy groups: $\phi^{\Sscript{u}}: \hH^{\Sscript{u}}  \to \gG^{\Sscript{\phi_0(u)}}$, for every $u \in H_{\Sscript{0}}$. Naturally,  groupoids,  morphism of groupoids, and natural transformations form a $2$-category $\Sf{Grpds}$.
Next we describe some typical examples of  groupoids and their morphisms.

\begin{example}\label{exam:pb}
Let $G$ be a group and  fix a set $M$. Denote by $BG_{\Sscript{M}}$ the category whose objects are (left) $G$-torsors of the form $(P,G,M)$ and morphisms are $G$-morphisms. In the terminology of Definition  \ref{def:pbset}  below,  an object $(P,G,M)$ in $BG_{\Sscript{M}}$ is a  \emph{principal  left $G$-set} $P$. Precisely, this is a left $G$-set $P$ with projection  $\pi:P \to M$ to the set of orbits $M$ and where the canonical map $G\times P \to P\, \due  \times {\Sscript{\pi}} {\, \Sscript{\pi}} \, P$, $(g,p) \mapsto (gp,p)$ is bijective. It is clear that any morphism in this category is an isomorphism, thus, $BG_{\Sscript{M}}$ is a  groupoid (probably not small). The groupoid $BG_{\Sscript{pt}}$ plays  a crucial role in the representation theory of the group $G$. Furthermore,  when $M$ varies in the category $\sf{Sets}$ of sets, we obtain the presheaf of groupoids $BG: \Sf{Sets}^{\Sscript{op}} \longrightarrow \Sf{Grpds}$, which is known as \emph{the classifying stack} of the group $G$.
\end{example}

\begin{example}\label{exam:X}
Assume that $\cR \subseteq X \times X$ is an equivalence relation on a set $X$.  One can construct a groupoid $\xymatrix@C=35pt{\cR \ar@<0.8ex>@{->}|-{\scriptscriptstyle{pr_2}}[r] \ar@<-0.8ex>@{->}|-{\scriptscriptstyle{pr_1}}[r] & \ar@{->}|-{ \scriptscriptstyle{\iota}}[l] X, }$  with structure maps as follows. The source and the target are $\Sf{s}=pr_{\Sscript{2}}$ and $\Sf{t}=pr_{\Sscript{1}}$, the second and the first projections, and the  map of identity arrows $\iota$ is the diagonal one. The multiplication and the inverse maps use, respectively, the transitivity and reflexivity of $\cR$ and are given by 
$$
(x,x') \, (x',x'')\,=\, (x,x''),\quad \text{and} \quad (x,x')^{-1}\,=\, (x',x).
$$

This is  an important class of groupoids  known as \emph{the groupoid of equivalence relation}, see \cite[Exemple 1.4, page 301]{DemGab:GATIGAGGC}.  A particular situation is  when $\cR=X \times X$, that is, the obtained groupoid is the so called \emph{the groupoid  of pairs} (called \emph{fine groupoid} in \cite{Brown:1987} and \emph{simplicial groupoid} in \cite{Higgins:1971}); or  when $\cR$ is defined by a certain fibred product $X\, \due \times {\Sscript{\nuup}} {\; \Sscript{\nuup}} \, X$ for a map $\nuup: X \to Y$.
\end{example}

\begin{example}\label{exam:action}
Any group $G$ can be considered as a groupoid by taking $G_{\Sscript{1}}=G$ and $G_{\Sscript{0}}=\{*\}$ (a set with one element). Now if $X$ is a right $G$-set with action $\rho: X\times G \to X$, then one can define the so called \emph{the action groupoid}: $G_{\Sscript{1}}=X \times G$ and $G_{\Sscript{0}}=X$, the source and the target are $\Sf{s}=\rho$ and $\Sf{t}=pr_{\Sscript{1}}$, the identity map sends $x \mapsto (e,x)=\iota_{\Sscript{x}}$, where $e$ is the identity element of $G$. The multiplication is given by  $(x,g) (x',g')=(x,gg')$, whenever $xg=x'$, and the inverse is defined by $(x,g)^{-1}=(xg,g^{-1})$. Clearly the pair of maps $(pr_{\Sscript{2}}, *): (G_{\Sscript{1}}, G_{\Sscript{0}}) \to (G,\{*\})$ defines a morphism of groupoids. 
\end{example}

\begin{example}\label{exam:induced}
Let $\gG=(G_{\Sscript{1}}, G_{\Sscript{0}})$ be a groupoid and $\varsigma:X \to G_{\Sscript{0}}$ a map. Consider the following  pair of sets:
$$
G^{\Sscript{\varsigma}}{}_{\Sscript{1}}:= X \,  \due \times {\Sscript{\varsigma}} {\, \Sscript{\Sf{t}}} \, G_{\Sscript{1}} \;  \due \times {\Sscript{\Sf{s}}} {\, \Sscript{\varsigma}}  \, X= \Big\{ (x,g,x') \in X\times G_{\Sscript{1}}\times X| \;\; \varsigma(x)=\Sf{t}(g), \varsigma(x')=\Sf{s}(g)  \Big\}, \quad G^{\Sscript{\varsigma}}{}_{\Sscript{0}}:=X.
$$
Then $\gG^{\Sscript{\varsigma}}{}=(G^{\Sscript{\varsigma}}{}_{\Sscript{1}}, G^{\Sscript{\varsigma}}{}_{\Sscript{0}})$ is a groupoid, with structure maps: $\Sf{s}= pr_{\Sscript{3}}$, $\Sf{t}= pr_{\Sscript{1}}$, $\iota_{\Sscript{x}}=(\varsigma(x), \iota_{\Sscript{\varsigma(x)}}, \varsigma(x))$, $x \in X$. The multiplication is defined by $(x,g,y) (x',g',y')= ( x,gg',y')$, whenever $y=x'$, and the inverse is given by $(x,g,y)^{-1}=(y,g^{-1},x)$. 
The groupoid $\gG^{\Sscript{\varsigma}}$ is known as \emph{the induced groupoid of $\gG$ by the map $\varsigma$}, (or \emph{ the pull-back groupoid of $\gG$ along $\varsigma$}, see   \cite{Higgins:1971} for dual notion).  Clearly, there  is a canonical morphism $\phi^{\Sscript{\varsigma}}:=(pr_{\Sscript{2}}, \varsigma): \gG^{\Sscript{\varsigma}} \to \gG$ of groupoids. 
\end{example}

Any morphism $\phi: \mathscr{H} \to \mathscr{G}$ of groupoids factors through the canonical morphism $\gG^{\Sscript{\phi_0}} \to \gG$, that is we have the following  (strict) commutative diagram 
$$
\xymatrix@R=7pt{\hH \ar@{->}^-{\phi}[rr] \ar@{.>}_-{\phi'}[rd]  & & \gG \\ & \gG^{\Sscript{\phi_0}} \ar@{->}^-{}[ru] & }
$$
of groupoids, where $\phi'_{\Sscript{0}}=id_{\Sscript{\hH_{0}}}$ and 
$$
\phi'_{\Sscript{1}}: H_{\Sscript{1}} \longrightarrow G^{\Sscript{\phi_0}}{}_{\Sscript{1}}, \quad \Big( h \longmapsto \big(\Sf{t}(h), \phi_{\Sscript{1}}(h), \Sf{s}(h)\big)\Big).  
$$

A particular and important example of an induced groupoid is the case when $\gG$ is a groupoid with one object, that is, a group.  In this case, to any group $G$ and a set $X$, one can associated the groupoid $(X\times G \times X,X)$ as the induced groupoid of $(G,\{ * \})$ by the map $*: X \to \{ * \}$.

Recall that a groupoid $\gG=(G_{\Sscript{1}}, G_{\Sscript{0}})$ is said to be \emph{transitive} if the map $(\Sf{s},\Sf{t}): G_{\Sscript{1}} \to G_{\Sscript{0}} \times G_{\Sscript{0}}$ is surjective.

\begin{example}\label{exam.transitive}
The groupoid of pairs is clearly transitive, as well as any  induced groupoid of the form $(X\times G \times X, X)$. On other hand, if a group $G$ acts transitively on  a set $X$, then the associated action groupoid is by construction transitive. 
\end{example}

Let $\gG$ be a transitive groupoid. Then if $x, y \in G_{\Sscript{0}}$, there is a non-canonical isomorphism of groups $\gG^{\Sscript{x}} \cong \gG^{\Sscript{y}}$ given by conjugation: Let 
$g \in G_{\Sscript{1}}$ with $x= \Sf{s}(g)$ and $\Sf{t}(g)=y$, then
$$
\gG^{\Sscript{x}} \longrightarrow \gG^{\Sscript{y}}, \Big( h \longmapsto g h g^{-1}\Big) 
$$
is an isomorphism of groups. 
This fact is essential in showing that any transitive groupoid is  isomorphic, in a non-canonical way, to an induced groupoid of the form $(X\times G \times X, X)$. Indeed, given a transitive groupoid $\gG$, fix an object $x \in G_{\Sscript{0}}$ with isotropy group $\gG^{\Sscript{x}}$ and chose a family of arrows $\{f_{\Sscript{y}}\}_{\Sscript{ y \,\in\, G_0}}$ such that $f_{\Sscript {y}} \in \Sf{t}^{-1}(\{x\})$ and $\Sf{s}(f_{\Sscript{y}})=y$, for $y \neq x$ while $f_{\Sscript{x}}=\iota(x)$, for $y=x$.  In this way the morphism 
$$
\phi^{\Sscript{x}}: \gG \overset{\cong}{\longrightarrow} (G_{\Sscript{0}}\times \gG^{\Sscript{x}}\times G_{\Sscript{0}}, G_{\Sscript{0}}),\quad \Big[ (g,z) \longmapsto \Big(\Big(\Sf{s}(g), f_{\Sscript{\Sf{t}(g)}}\, g \, f_{\Sscript{\Sf{s}(g)}}^{-1}, \Sf{t}(g)\Big), z \Big) \Big]
$$
establishes an isomorphism of groupoids.

\subsection{Groupoids actions,  equivariant maps  and the orbits sets}\label{ssec:Grpd1} The following definition is a natural generalization to the context of groupoids, of the usual notion of group-set. It is an abstract formulation of that given in \cite[Definition 1.6.1]{Mackenzie:2005} for Lie groupoids, and essentially the same definition based on the Sets-bundles notion given in  \cite[Definition 1.11]{Renault:1980}.
\begin{definition}\label{def:Gset}
Given a groupoid $\mathscr{G}$, a \emph{right} $\mathscr{G}$-\emph{set} is a triple $(X,\varsigma, \rho)$ where  $X$ is a set and $\varsigma:X \to G_{\Sscript{0}}$ and $\rho: X\, \due \times {\Sscript{\varsigma}} {\, \Sscript{\Sf{t}}} \, G_{\Sscript{1}} \to X$ (shortly written as  $\rho(x,g) := xg$)  are the \emph{structure} and \emph{action} maps respectively. These maps obey  the following conditions
\begin{enumerate}
\item $\Sf{s}(g)=\varsigma(xg)$, for any $x \in X$ and $g \in G_{\Sscript{1}}$ with $\varsigma(x)=\Sf{t}(g)$.
\item $x \iota_{\varsigma(x)}= x$, for every $x \in X$.
\item $ (xg)h= x(gh)$, for every $x \in X$, $g,h \in G_{\Sscript{1}}$ with $\varsigma(x)=\Sf{t}(g)$ and $\Sf{t}(h)=\Sf{s}(g)$.
\end{enumerate}
\end{definition}
In order to simplify the notation we denote a right $\gG$-set by a pair $(X,\varsigma)$, omitting the action $\rho$. A \emph{left groupoid action} is analogously defined by interchanging the source with the target and similar notations might  be adopted. Obviously, any groupoid  $\gG$ acts  over itself on both sides by using the regular action, i.e.,  the multiplication $G_{\Sscript{1}} \, \due \times {\Sscript{\Sf{s}}} {\, \Sscript{\Sf{t}}} \, G_{\Sscript{1}} \to G_{\Sscript{1}}$. Thus, we have that  $(G_{\Sscript{1}}, \Sf{s})$ is a right $\gG$-set and $(G_{\Sscript{1}}, \Sf{t})$ is a left $\gG$-set.

Let  $(X,\varsigma)$ be a right $\gG$-set, and consider the pair of sets $X \rJoin \gG :=\Big( X\, \due \times {\Sscript{\varsigma}} {\, \Sscript{\Sf{t}}} \, G_{\Sscript{1}} , X \Big)$ as a groupoid with structure maps $\Sf{s}=\rho$, $\Sf{t}=pr_{\Sscript{1}}$, $\iota_{\Sscript{x}}=(x, \iota_{\varsigma(x)})$. The multiplication and the inverse maps are defined by $(x,g)(x',g')=(x,gg')$ and $(x,g)^{-1}=(xg,g^{-1})$. The groupoid $X\rJoin \gG$ is known as the \emph{right translation groupoid of $X$ by $\gG$}. 

For sake of completeness let us recall the notion of equivariant maps. A \emph{morphism of  right $\gG$-sets} (or \emph{$\gG$-equivariant map})  $F: (X,\varsigma) \to (X',\varsigma')$ is a map $F:X \to X'$ such that the diagrams 
\begin{equation}
\begin{gathered}
\xymatrix@R=7pt{ & X \ar@{->}_-{\Sscript{\varsigma}}[ld]  \ar@{->}^-{F}[dd] & \\ G_{\Sscript{0}}& & \\ & X' \ar@{->}^-{\Sscript{\varsigma'}}[lu]  & } \qquad  \qquad \xymatrix@R=7pt{X\, \due \times {\Sscript{\varsigma}} {\, \Sscript{\Sf{t}}} \,  G_{\Sscript{1}} \ar@{->}^-{}[rr]  \ar@{->}_-{\Sscript{F\, \times \, id}}[dd] & & X  \ar@{->}^-{\Sscript{F}}[dd] \\  & & \\ X'\, \due \times {\Sscript{\varsigma'}} {\, \Sscript{\Sf{t}}} \,  G_{\Sscript{1}}  \ar@{->}^-{}[rr] & & X'  } 
\end{gathered}
\end{equation}
commute. Clearly any such a $\gG$-equivariant map induces a morphism of groupoids $\Sf{F}: X \rJoin \gG \to X' \rJoin \gG$.

Next we recall the notion of the orbit set attached to a right groupoid-set.  This notion is a generalization of the orbit set in the context of group-sets.  Here we use the (right) translation groupoid to introduce this set. 
First we recall the notion of the orbit set of a given groupoid.  \emph{The orbit set of a groupoid} $\gG$ is the  quotient set of $G_{\Sscript{0}}$  by the following equivalence relation: 
$$ 
x \sim y \, \iff \, \exists \, g \in G_{\Sscript{1}} \;\text{ such that }\; {\sf{s}}(g)=x \;\text{ and }\; {\sf{t}}(g)=y.
$$
In others words, this is the set of all connected components of $\gG$. 

Given  a right $\gG$-set  $(X,\varsigma)$, the \emph{orbit set}  $X/\gG$ of $(X,\varsigma)$ is the orbit set of the translation groupoid  $X \rJoin \gG$. If $\gG=(X\times G, X)$ is an  action groupoid as in Example \ref{exam:action}, then obviously the orbit set of this groupoid coincides with the classical set  $X/G$ of orbits.

\subsection{Principal groupoid-bisets and the two sided translation groupoid}\label{ssec:biset}
We give in this subsection an exhaustive survey on principal groupoids bisets and the formal constructions of their bicategories. 

Let $\gG$ and $\hH$ be two groupoids and $(X,\varsigma, \vartheta)$ a triple consisting of a set $X$ and two maps $\varsigma : X \to G_{\Sscript{0}}$, $\vartheta: X \to H_{\Sscript{0}}$.  The following definitions are abstract formulations of those given in \cite{Jelenc:2013, Moedijk/Mrcun:2005} for topological  and Lie groupoids.
\begin{definition}\label{def:biset}
An \emph{$(\hH,\gG)$-biset}  is a triple $(X,\varsigma, \vartheta)$ where $(X,\varsigma)$ is endowed with a  right $\gG$-action $\rho:X\, \due \times {\Sscript{\varsigma}} {\, \Sscript{\Sf{t}}} \,  G_{\Sscript{1}} \to X$  and  $(X,\vartheta)$ with a left $\hH$-action $ \lambda: H_{\Sscript{1}}\, \due \times {\Sscript{\Sf{s}}} {\, \Sscript{\vartheta}} \,  X \to X$ 
such that
\begin{enumerate} 
\item For any $x \in X$, $h \in H_{\Sscript{1}}$, $g \in G_{\Sscript{1}}$ with $\vartheta(x)=\Sf{s}(h)$ and $\varsigma(x)=\Sf{t}(g)$, we have
$$ \vartheta(xg) =\vartheta(x)\; \text{ and }\; \varsigma(hx)=\varsigma(x).$$
\item For any $ x \in X$, $h \in H_{\Sscript{1}}$ and $ g \in G_{\Sscript{1}}$ with  $\varsigma(x)=\Sf{t}(g)$, $\vartheta(x)=\Sf{s}(h)$, we have 
$h(xg)\,=\, (hx)g$.
\end{enumerate}
\end{definition}

\emph{The two sided translation groupoid} associated to a given $(\hH, \gG)$-biset $(X,\varsigma, \vartheta)$ is defined to be the groupoid $\hH \lJoin X \rJoin \gG$ whose set of objects is $X$ and set of arrows is 
$$
H_{\Sscript{1}}\, \due \times {\Sscript{\Sf{s}}} {\, \Sscript{\vartheta}} \, X \, \due \times {\Sscript{\varsigma}}{\, \Sscript{\Sf{s}}} \, G_{\Sscript{1}}\,=\, \Big\{ (h,x,g) \, \in \,  H_{\Sscript{1}}\times X \times G_{\Sscript{1}}| \,\, \Sf{s}(h)= \vartheta(x),\, \Sf{s}(g)=\varsigma(x) \Big\}.
$$
The structure maps are:  
$$
\Sf{s}(h,x,g)=x,\quad \Sf{t}(h,x,g)=hxg^{-1}\;\;  \text{ and }\; \iota_{\Sscript{x}}=(\iota_{\Sscript{\vartheta(x)}}, x,  \iota_{\Sscript{\varsigma(x)}}). 
$$
The multiplication and the inverse are given by: 
$$
(h,x,g) (h',x',g')\,=\,(hh',x',gg'),\quad  (h,x,g)^{-1}=(h^{-1}, hxg^{-1}, g^{-1}).
$$

Associated to a given $(\hH, \gG)$-biset $(X,\varsigma,\vartheta)$, there are two canonical morphisms of groupoids: 
\begin{eqnarray}
\Sigma: \hH \lJoin X \rJoin \gG \longrightarrow \gG, & & \Big( (h,x,g),  y \big) \longmapsto \big( g,\varsigma(y) \Big),         \label{Eq:t}  \\
\Theta: \hH \lJoin X \rJoin \gG \longrightarrow \hH, & &   \Big( (h,x,g),  y \big) \longmapsto \big( h,\vartheta(y) \Big). \label{Eq:s}
\end{eqnarray}

\begin{definition}\label{def:pbset}
Let  $(X,\varsigma,\vartheta)$ be  an $(\hH,\gG)$-biset. We say that $(X,\varsigma,\vartheta)$ is a \emph{left principal $(\hH,\gG)$-biset} if it satisfies the following conditions:
\begin{enumerate}[(P-1)]
\item $\varsigma:  X \to G_{\Sscript{0}}$ is surjective;
\item the canonical map 
\begin{equation}\label{Eq:can}
\nabla: H_{\Sscript{1}}\, \due \times {\Sscript{\Sf{s}}} {\, \Sscript{\vartheta}} \, X \longrightarrow X\,  \due \times {\Sscript{\varsigma}} {\, \Sscript{\varsigma}} \, X  , \quad \Big( (h,x) \longmapsto (hx,x)\Big) 
\end{equation}
is bijective. 
\end{enumerate}
\end{definition}
By condition (P-2) we consider the map $\delta: = pr_{\Sscript{1}} \circ \nabla^{-1}:  X\,  \due \times {\Sscript{\varsigma}} {\, \Sscript{\varsigma}} \, X \to H_{\Sscript{1}}$. This map clearly satisfies:  
\begin{eqnarray}
 \Sf{s}\big(\delta(x,y)\big) &=& \vartheta(y) \label{Eq:d1} \\
\delta(x,y)y&=& x,\quad \text{ for any} \,  x,y \in X\, \text{ with }\, \varsigma(x)=\varsigma(y); \label{Eq:d2} \\ 
\delta(hx,x) &=& h, \quad \text{for } h \in H_{\Sscript{1}},\, x \in X\,\, \text{ with } \Sf{s}(h)=\vartheta(x). \label{Eq:d3}
\end{eqnarray}
Equation \eqref{Eq:d3}, shows that the action is in fact free, that is, $h x=x$ only when $h=\iota_{\Sscript{\vartheta(x)}}$. The subsequent lemma is also immediate from this definition. 
\begin{lemma}\label{lema:orbit}
Let $(X,\varsigma,\vartheta)$ be a left principal  $(\hH,\gG)$-biset. Then the map $\varsigma$ induces a bijection between the orbit set $X/\hH$ and the set of objects $G_{\Sscript{0}}$.
\end{lemma}

Analogously one defines \emph{right principal $(\hH,\gG)$-biset}. \emph{A principal  $(\hH,\gG)$-biset} is  both left and right principal biset.  
For instance, $(G_{\Sscript{1}}, \Sf{t}, \Sf{s})$ is a left and right principal $(\gG,\gG)$-biset, known as the \emph{unit principal biset},  which we denote by $\uU(\gG)$. 
More examples of left principal bisets can be performed, as in the geometric case, by pulling back other left principal bisets. Precisely, assume we are given $(X,\varsigma,\vartheta)$ a left principal $(\hH,\gG)$-biset, and let $\psi: \kK \to \gG$ be a morphism of groupoids. Consider the set $Y:= X\,\due \times {\Sscript{\vartheta}}{\, \Sscript{\psi_0}} \, K_{\Sscript{0}}$ together with maps $ pr_{\Sscript{2}}  :Y \to K_{\Sscript{0}}$ and $\td{\varsigma}:=  \varsigma \circ  pr_{\Sscript{1}}: Y \to H_{\Sscript{0}}$. Then the triple $(Y,\td{\varsigma},pr_{\Sscript{2}} )$ is an $(\hH, \kK)$-biset with actions 
\begin{eqnarray}
\lambda: H_{\Sscript{1}} \, \due \times {\Sscript{\Sf{s}}} {\, \Sscript{\varsigma}} \, Y \longrightarrow Y, & & \big(h, (x,u)\big) \longmapsto \big( hx,u \big)  \label{Eq:laction} \\ 
 \rho: Y  \, \due \times {\Sscript{\td{\varsigma}}} {\, \Sscript{\Sf{t}}} \, K_{\Sscript{1}} \longrightarrow Y, & & \big((x,u), f\big) \longmapsto \big( x\psi_{\Sscript{1}}(f),\Sf{s}(f) \big),  \label{Eq:raction}
\end{eqnarray}
which is  actually a left principal $(\hH,\kK)$-biset, and known as the \emph{pull-back principal biset of $(X,\varsigma,\vartheta)$}; we denote it by $\psi^{\Sscript{*}}\big((X,\varsigma,\vartheta)\big)$. A left principal biset is called a \emph{trivial left principal biset} if it is the pull-back of the unit left principal biset, that is, of the form $\psi^{\Sscript{*}}(\uU(\gG))$ for some morphism of groupoids $\psi: \kK \to \gG$. 

Next we expound  the bicategorical constructions beyond the notion of principal groupoids-bisets.  A \emph{morphism of left principal $(\hH, \gG)$-bisets} $F: (X,\varsigma, \vartheta) \to (X', \varsigma', \vartheta')$  is a map $F: X \to X'$ which is simultaneously  $\gG$-equivariant and $\hH$-equivariant, that is,   the following diagrams
\begin{equation}\label{Eq:mpb}
\begin{gathered}
\xymatrix@R=7pt{ & X \ar@{->}_-{\Sscript{\varsigma}}[ld] \ar@{->}^-{\Sscript{\vartheta}}[rd]  \ar@{->}^-{F}[dd] & \\ G_{\Sscript{0}}& & H_{\Sscript{0}} \\ & X' \ar@{->}^-{\Sscript{\varsigma'}}[lu] \ar@{->}_-{\Sscript{\vartheta'}}[ru] & } \qquad \xymatrix@R=7pt{X\, \due \times {\Sscript{\varsigma}} {\, \Sscript{\Sf{t}}} \,  G_{\Sscript{1}} \ar@{->}^-{}[rr]  \ar@{->}_-{\Sscript{F\, \times  \, id}}[dd] & & X  \ar@{->}^-{\Sscript{F}}[dd] \\  & & \\ X'\, \due \times {\Sscript{\varsigma'}} {\, \Sscript{\Sf{t}}} \,  G_{\Sscript{1}}  \ar@{->}^-{}[rr] & & X'  } \qquad \xymatrix@R=7pt{ H_{\Sscript{1}} \, \due \times {\Sscript{\Sf{s}}} {\, \Sscript{\vartheta}} \, X \ar@{->}^-{}[rr]  \ar@{->}_-{\Sscript{id \,  \times  \, F}}[dd] & & X  \ar@{->}^-{\Sscript{F}}[dd] \\  & & \\ H_{\Sscript{1}} \, \due \times {\Sscript{\Sf{s}}} {\, \Sscript{\vartheta'}} \,  X'  \ar@{->}^-{}[rr] & & X'.  }
\end{gathered}
\end{equation}
commute. \emph{An isomorphism of left principal bisets} is a morphism whose underlying map is bijective.   As in the geometric case we have:

\begin{proposition}\label{prop:pbg}
Given two groupoids $\gG$ and $\hH$. Then any morphism between left principal $(\hH, \gG)$-bisets is an isomorphism. 
\end{proposition}
\begin{proof}
Let  $F: (X,\varsigma, \vartheta) \to (X', \varsigma', \vartheta')$ be a morphism of left principal $(\hH, \gG)$-bisets. We  first show that  $F$ is injective. So take $x, y \in X$ such that $F(x )=F(y)$, whence $\varsigma(x)=\varsigma(y)$. By Lemma \ref{lema:orbit}, we know that there exists $h \in H_{\Sscript{1}}$ with  $\Sf{s}(h)=\vartheta(x)$ such that $hx=y$. Therefore, we have $F(hx)=F(y)=hF(x)=F(x)$ and so $h=\iota_{\Sscript{\vartheta(x)}}$, since the left action is free. This shows that $x=y$. The surjectivity of $F$ is derived as follows. Take an arbitrary element $x' \in X'$ and consider its image $\varsigma'(x') \in G_{\Sscript{0}}$. Since $\varsigma$ is surjective, there exists $x \in X$ such that $\varsigma(x) =\varsigma'F(x)=\varsigma'(x')$. This means that $F(x)$ and $x'$ are in the same orbit, so there exists $h' \in H_{\Sscript{1}}$ (with $\Sf{s}(h')=\vartheta(x)$) such that $h'F(x)=F(h'x)=x'$, which shows that $F$ is surjective.   
\end{proof}

\begin{remark}\label{remark:Entredosmares}
By Proposition \ref{prop:pbg}, the category of left principal bisets $\mathsf{PB}^{\Sscript{l}}(\hH,\gG)$ is actually a groupoid (not necessarly  a small category). On the other hand, notice that if $(X,\varsigma, \vartheta)$ is a left principal $(\hH,\gG)$-biset, then its opposite $(X^{\Sscript{o}}, \vartheta, \varsigma)$ is a right principal $(\gG,\hH)$-biset, where the underlying set still the same set $X$ while the actions were switched by using the inverse maps of both groupoids. This in fact establishes an isomorphism of categories between $\mathsf{PB}^{\Sscript{l}}(\hH,\gG)$ and the category of right principal bisets  $\mathsf{PB}^{\Sscript{r}}(\gG,\hH)$.
\end{remark}

\begin{remark}\label{remark:TanFrida}
Given $(X,\varsigma, \vartheta)$ an $(\hH,\gG)$-biset and $(X',\varsigma', \vartheta')$  a $(\gG,\kK)$-biset.   One can  endow the fibre product $X\, \due \times {\Sscript{\varsigma}} {\,\Sscript{\vartheta'}} \, X'$ within a structure of an $(\hH, \kK)$-biset. Furthermore, $\gG$ also acts on this set by the action $(x,x').g=(xg,g^{-1}x')$, for $g \in G_{\Sscript{1}}$, $(x,x') \in X\, \due \times {\Sscript{\varsigma}} {\,\Sscript{\vartheta '}} \, X'$ with $\Sf{t}(g)= \varsigma(x)=\vartheta'(x')$. Denote by  $X\tensor{\gG}X':= \big(X\, \due \times {\Sscript{\varsigma}} {\,\Sscript{\vartheta'}} \, X'\big)/\gG $ its orbit set, then clearly this set inherits a structure of $(\hH, \kK)$-biset. This is \emph{the tensor product of bisets}, also known as  \emph{le produit contract\'e} \cite[D\'efinition 1.3.1 page 114]{Giraud:1971},  \cite[Chap.III, \S 4, 3.1]{DemGab:GATIGAGGC}. 
It turns out that, if $(X,\varsigma, \vartheta)$ is a left principal biset and $(X',\varsigma', \vartheta')$ is a left principal biset, then $X\tensor{\gG}X'$ is a left principal $(\hH,\kK)$-biset.  Moreover, one can show that the tensor product (over different groupoids) is associative, up to a natural isomorphism. This defines the bicategory $\mathsf{PB}^{\Sscript{l}}$ of left principal bisets. Analougly, we have the bicategories $\mathsf{PB}^{\Sscript{r}}$ and  $\mathsf{PB}^{\Sscript{b}}$ (of  principal bisets).

For a single $0$-cell, i.e., a groupoid $\gG$,  the category $\mathsf{PB}^{\Sscript{b}}(\gG,\gG)$ turns to be a \emph{bigroup} (or \emph{a categorical group}).  
Moreover, in analogy with the group case, one can construct with the help of Proposition \ref{prop:pbg} and by using morphisms between left translation groupoids,  a presheaf $\bB\gG: \Sf{Sets}^{\Sscript{op}} \longrightarrow 2\text{-}\Sf{Grpds}$ to the category of $2$-groupoids known as \emph{the classifying 2-stack} of the groupoid $\gG$ (compare with Example \ref{exam:pb}).
\end{remark} 

\subsection{Principal groupoids-biset versus weak equivalences}\label{ssec:WE} 
 A morphism of groupoids $\phi: \mathscr{H} \to \mathscr{G}$ is said to be a \emph{weak equivalence} if it satisfies the following two conditions:
\begin{enumerate}[(WE-1)]
\item The composition map $\xymatrix{ G_{\Sscript{1}} \, \due \times {\Sscript{\Sf{s}}} {\, \Sscript{\phi_{\Sscript{0}}}} \, H_{\Sscript{0}} \ar@{->}^-{\Sscript{pr_1}}[r] & G_{\Sscript{1}}  \ar@{->}^-{\Sscript{\Sf{t}}}[r] & G_{\Sscript{0}} }$ is surjective.\label{Cond1}

\item The following diagram is cartesian 
$$
\xymatrix@R=7pt{ H_{\Sscript{1}} \ar@{->}_-{(\Sf{s},\Sf{t})}[dd]  \ar@{->}^-{\phi_{\Sscript{1}}}[rr] & &  G_{\Sscript{1}} \ar@{->}^-{(\Sf{s},\Sf{t})}[dd]\\ & & \\ H_{\Sscript{0}} \times H_{\Sscript{0}} \ar@{->}^-{\phi_{\Sscript{0}} \times \phi_{\Sscript{0}}}[rr] & & G_{\Sscript{0}} \times G_{\Sscript{0}} }
$$ 
Equivalently  there is a bijection $\Gamma: H_{\Sscript{1}} \cong H_{\Sscript{0}} \,\, \due \times {\Sscript{\phi_0}} {\, \Sscript{\Sf{s}}} \, G_{\Sscript{1}}\,\, \due \times {\Sscript{\Sf{t}}} {\, \Sscript{\phi_0}}  \, H_{\Sscript{0}}$ such that $pr_{\Sscript{2}} \circ \Gamma=\phi_{\Sscript{0}}$  and $(pr_{\Sscript{1}},\,pr_{\Sscript{3}}) \circ \Gamma= (\Sf{s},\Sf{t})$. \label{Cond2}
\end{enumerate}

In categorical terms, condition (WE-1) says that $\phi$ is an \emph{essentially surjective} functor:  Each object of $\mathscr{G}$ is isomorphic to the  image by $\phi$ of an object in $\mathscr{H}$. The second condition, means that $\phi$ is \emph{fully faithful}: If $u, v$ are two objects in $\mathscr{H}$ then $\phi$ defines a bijection between the sets of arrows $\mathscr{H}(u, v)$ and  $\mathscr{G}\big(\phi_{\Sscript{0}}(u), \phi_{\Sscript{0}}(v)\big)$. Both properties classically characterize  functors which define  equivalences of categories. 

Two groupoids $\gG$ and $\hH$ are said to be \emph{weakly equivalent} when there exists a third groupoid $\kK$ with a diagram (i.e.,  a \emph{span}) of weak equivalences:
$$
\xymatrix@R=7pt{ & \ar@{->}_-{}[ld] \kK \ar@{->}^-{}[rd] &  \\  \gG  & &   \hH.}
$$

For sake of completeness, next  we give a result which relate the notion of principal biset with that of weak equivalence.

\begin{proposition}\label{prop:pb}
Let $\gG$ and $\hH$ be two groupoids. Assume that there is $(X,\varsigma, \vartheta)$ a  principal $(\hH, \gG)$-biset. Then the canonical morphisms of groupoids 
$$
\xymatrix@R=7pt{ & \ar@{->}_-{\Sscript{\Theta}}[ld] \hH \lJoin X \rJoin \gG \ar@{->}^-{\Sscript{\Sigma}}[rd] &  \\  \hH  & &   \gG}
$$
are weak equivalences, where $\Theta$, $\Sigma$ are as in \eqref{Eq:t} and \eqref{Eq:s}. In particular, $\gG$ and $\hH$ are weakly equivalent. 
\end{proposition}
\begin{proof}
We only show that if $(X,\varsigma,\vartheta)$ is a left principal $(\hH,\gG)$-biset, then the canonical morphism 
$$
\Sigma:  \hH \lJoin  X \rJoin \gG \longrightarrow \gG,\quad \Big( \big((h,x,g),x\big)  \longmapsto (g, \varsigma(x) \Big)
$$
is a weak equivalence. The proof  of the fact that  $\Theta$ is a weak equivalence follows similarly from the assumption that $(X,\varsigma,\vartheta)$ is right principal $(\hH, \gG)$-biset. 

Condition (WE-1) for $\Sigma$ is clear, since $\varsigma$ is surjective by condition (P-1) of Definition \ref{def:pbset}.  Consider the map  
$$
\xymatrix@R=0pt{ 
\Gamma: H_{\Sscript{1}} \, \due \times {\Sscript{\Sf{s}}} {\; \Sscript{\vartheta}}\, X\, \due  \times {\Sscript{\varsigma}} {\,\Sscript{\Sf{s}}}\, G_{\Sscript{1}} \ar@{->}^-{}[rr] & &  X \, \due \times {\Sscript{\varsigma}} {\; \Sscript{\Sf{s}}}\, G_{\Sscript{1}}\, \due  \times {\Sscript{\Sf{t}}} {\,\Sscript{\varsigma}}\, X \\  \big(h,x,g \big) \ar@{|->}^-{}[rr]& & \big( x,g,hxg^{-1}\big).  }
$$
Using the map $\delta: X\, \due  \times {\Sscript{\varsigma}} {\; \Sscript{\varsigma}}\, X \to H_{\Sscript{1}}$ resulting from condition (P-2) on $(X,\varsigma, \vartheta)$ and which satisfies equations \eqref{Eq:d1}-\eqref{Eq:d3}, we define the inverse of $\Gamma$ to be the  map: 
$$
\xymatrix@R=0pt{ 
\Gamma^{-1}: X \, \due \times {\Sscript{\varsigma}} {\; \Sscript{\Sf{s}}}\, G_{\Sscript{1}}\, \due  \times {\Sscript{\Sf{t}}} {\,\Sscript{\varsigma}}\, X  \ar@{->}^-{}[rr] & &  
H_{\Sscript{1}} \, \due \times {\Sscript{\Sf{s}}} {\; \Sscript{\vartheta}}\, X\, \due  \times {\Sscript{\varsigma}} {\,\Sscript{\Sf{s}}}\, G_{\Sscript{1}} \\  \big(x,g,y \big) \ar@{|->}^-{}[rr]& & \big( \delta(y,xg^{-1}), x,g\big), }
$$
which gives condition (WE-2) for $\Sigma$.
\end{proof}

\begin{remark}\label{remak:tusojosverdes}
As we have seen in Remark \ref{remark:Entredosmares}, the opposite of left principal $(\hH,\gG)$-biset is a right principal $(\gG,\hH)$-biset. Thus the opposite of principal biset is also a principal biset. In this way,  Proposition \ref{prop:pb} says that the ``equivalence relation" between groupoids defined by '\emph{being connected by a principal biset'}  is contained in the equivalence relation defined by '\emph{being weakly equivalent}'. An interesting question is then to check if both relations are the same. Precisely, one can  ask whether two weakly equivalent  groupoids  $\hH$ and $\gG$ are connected by a certain principal $(\hH,\gG)$-biset. The complete answer was recently given in \cite[Theorem 2.9]{Kaoutit/Kowalzig:14} (see Remark \ref{remark:serastu} below for these equivalence relations in Hopf algebroids context).
\end{remark}

\subsection{Transitive groupoids are characterized by weak equivalences}\label{ssec:TGrpd}
This subsection is the main motivation for the forthcoming sections. Here we show perhaps a well known result that characterizes transitive groupoids by means of weak equivalences and principal groupoids-bisets.
\begin{proposition}\label{prop:grpd}
Let $\mathscr{G}$ be a groupoid. Then the following are equivalent:
\begin{enumerate}[(i)]
\item For every map $\varsigma: X \to G_{\Sscript{0}}$, the induced morphism of groupoids $\phi^{\Sscript{\varsigma}}: \mathscr{G}^{\Sscript{\varsigma}} \to \mathscr{G}$ is a weak equivalence;
\item $\mathscr {G}$ is a transitive groupoid;
\item For every map $\varsigma: X \to G_{\Sscript{0}}$, the pull-back biset $\phi^{\Sscript{\varsigma^{\,*}}}(\uU(\gG))$ is a principal $(\gG,\gG^{\Sscript{\varsigma}})$-biset.
\end{enumerate}
\end{proposition}
\begin{proof}
$(i)  \Rightarrow (ii)$. Is immediate. 

$(ii) \Rightarrow (iii)$. By definition $\phi^{\Sscript{\varsigma^{\,*}}}(\uU(\gG))$ is a left principal $(\gG,\gG^{\Sscript{\varsigma}})$-biset. We need then to check that, under condition $(ii)$, it is also right principal $(\gG,\gG^{\Sscript{\varsigma}})$-biset. This biset is given  by $\phi^{\Sscript{\varsigma^{\,*}}}(\uU(\gG))=\big(G_{\Sscript{1}} \, \due \times {\Sscript{\Sf{s}}} {\, \Sscript{\varsigma}} \, X, \td{\Sf{t}},  pr_{\Sscript{2}} \big)$, where $\td{\Sf{t}}:= \Sf{t} \circ pr_{\Sscript{1}}:  G_{\Sscript{1}} \, \due \times {\Sscript{\Sf{s}}} {\, \Sscript{\varsigma}} \, X \to G_{\Sscript{1}} \to G_{\Sscript{0}}$. 
The left and right actions are given as in equations \eqref{Eq:laction} and \eqref{Eq:raction} by
$$
g\rightharpoonup  (f,x) \,=\, (gf,x)\;\; \text{ and }\;\; (f,x) \leftharpoonup  (y,h,x)\,=\, (fh,y),
$$
for any $(f,x) \in G_{\Sscript{1}} \, \due \times {\Sscript{\Sf{s}}} {\, \Sscript{\varsigma}} \, X$, $g \in G_{\Sscript{1}}$ with $\Sf{s}(g)=\Sf{t}(f)$, and $h \in G_{\Sscript{1}}$ with $\Sf{s}(h)=\varsigma(y)$, $\Sf{t}(h)=\varsigma(x)$ and $\Sf{s}(f)=\Sf{t}(h)$.  Both conditions (1)-(2) in Definition \ref{def:biset} are then clearly satisfied. 
The right canonical map is defined by 
$$
\nabla': \Big( G_{\Sscript{1}} \, \due \times {\Sscript{\Sf{s}}} {\, \Sscript{\varsigma}} \, X\Big) \, \due \times {\Sscript{pr_2}} {\,\Sscript{\Sf{t}}} \, G^{\Sscript{\varsigma}}{}_{\Sscript{1}}  \longrightarrow \Big( G_{\Sscript{1}} \, \due \times {\Sscript{\Sf{s}}} {\, \Sscript{\varsigma}} \, X\Big) \, \due \times {\Sscript{\td{\Sf{t}}}} {\, \Sscript{\td{\Sf{t}}}} \, \Big(G_{\Sscript{1}} \, \due \times {\Sscript{\Sf{s}}} {\, \Sscript{\varsigma}} \, X \Big),\;\, \Big[ \Big( \big( f,x\big), \big(y,h,x\big)\Big) \longmapsto \Big(\big(f,x \big),  \big(fh,y \big)\Big) \Big].
$$

The map $\td{\Sf{t}}$ is clearly surjective, since $\gG$ is transitive. This gives condition (P-1) of Definition \ref{def:pbset} for  $\phi^{\Sscript{\varsigma^{\,*}}}(\uU(\gG))$ as a  right principal biset.   Now, we need to check that $\nabla'$ is bijective, that is, condition (P-2) is fulfilled. However, the inverse of this map is easily shown to be the following map
$$
\nabla'{}^{-1}:   \Big( G_{\Sscript{1}} \, \due \times {\Sscript{\Sf{s}}} {\, \Sscript{\varsigma}} \, X\Big) \, \due \times {\Sscript{\td{\Sf{t}}}} {\, \Sscript{\td{\Sf{t}}}} \, \Big(G_{\Sscript{1}} \, \due \times {\Sscript{\Sf{s}}} {\, \Sscript{\varsigma}} \, X \Big)    \longrightarrow
\Big( G_{\Sscript{1}} \, \due \times {\Sscript{\Sf{s}}} {\, \Sscript{\varsigma}} \, X\Big) \, \due \times {\Sscript{pr_2}} {\,\Sscript{\Sf{t}}} \, G^{\Sscript{\varsigma}}{}_{\Sscript{1}} ,\;\, \Big[\Big( \big( f,x\big), \big(f',x'\big)\Big) \longmapsto \Big(\big(f,x \big),  \big(x',f^{-1}f',x \big)\Big)\Big].
$$

$(iii) \Rightarrow (i)$. If we assume that $\phi^{\Sscript{\varsigma^{\,*}}}(\uU(\gG))$ is a principal $(\gG,\gG^{\Sscript{\varsigma}})$-biset, then the map $\td{\Sf{t}}$ above should be surjective. Therefore, the map 
$$
\xymatrix{  G_{\Sscript{1}} \, \due \times {\Sscript{\Sf{s}}} {\, \Sscript{\varsigma}} \, X \ar@{->}^-{pr_{\Sscript{1}}}[r] & G_{\Sscript{1}} \ar@{->}^-{\Sf{t}}[r] & G_{\Sscript{0}} }
$$
is also surjective, which is condition (WE-1) for the morphism $\phi^{\Sscript{\varsigma}}$. Condition (WE-2) for this morphism is trivial, since by definition we know that $G^{\Sscript{\varsigma}}{}_{\Sscript{1}}= X\, \due \times {\Sscript{\varsigma}} {\,\Sscript{\Sf{t}}}\,   G_{\Sscript{1}} \; \due \times {\Sscript{\Sf{s}}} {\, \Sscript{\varsigma}} \, X $.
\end{proof}

\subsection{Correspondence between transitive groupoids and principal group-sets}\label{ssec:GPS}
A particular example of principal groupoids-bisets are, of course, principal group-sets. As we will see below transitive groupoids are characterized by these group-sets. Precisely, there is a (non canonical) correspondence between transitive groupoids and principal group-sets, as we will show in this subsection.

Let $\pi: P \to G_{\Sscript{0}}$ be a map and $G$ a group which acts on the left side of $P$. Recall that the triple $(P,G,\pi)$ is said to be a \emph{left principal $G$-set}, if the following conditions are satisfied:
\begin{enumerate}[({P'}1)]
\item $\pi$ is surjective; \label{GPS1}
\item $\pi(gp)=\pi(p)$, for every $p \in P$ and $g \in G$; \label{GPS2}
\item The canonical map $G \times P \longrightarrow P \, \due \times {\Sscript{\pi}}{\, \Sscript{\pi}} \, P $ sending $(g,p) \mapsto (gp,p)$ is bijective. \label{GPS3}
\end{enumerate}
Equivalently, the action is free and $G_{\Sscript{0}} $ is the orbit set of $P$. 
Comparing with Definition \ref{def:biset}, this means that the triple  $(P,*,\pi)$ with $*: P \to \{*\}$, is a principal left $(G,G_{\Sscript{0}})$-biset, where the group $G$ is considered as a groupoid with one object $\{*\}$ and $G_{\Sscript{0}}$ is considered as a groupoid whose underlying category is a discrete category (i.e., category with only identities arrows) with set of objects $G_{\Sscript{0}}$, and  acts trivially on $P$ along $\pi$.

In the previous situation, consider $P \times P$ as a left $G$-set by the diagonal action and denote by $G_{\Sscript{1}}:= (P \times P) / G$ its set of orbits. The pair $(G_{\Sscript{1}}, G_{\Sscript{0}})$ admits as follows a  structure of transitive groupoid. 
Indeed, let $(p,p') \in P\times P$ and denote by $[(p,p')] \in (P\times P)/G$ its equivalence class. The source and target are $\Sf{s}\big([(p,p')]\big)=\pi(p')$ and $\Sf{t}\big([(p,p')]\big)=\pi(p)$. The identity arrow of an object $x \in G_{\Sscript{0}}$ is given, using conditions (P'1)-(P'2), by the class $[(p,p)]$ where $\pi(p)=x$. Let $\fk{p}, \fk{q}$  be two equivalence classes in $G_{\Sscript{1}}$ such that $\Sf{s}(\fk{p})=\Sf{t}(\fk{q})$. Henceforth, if $(p,p')$ is a representative of $\fk{p}$,  then $\fk{q}$ can be represented by $(p',p'')$. The multiplication $\fk{p}  \fk{q}$ is then represented by $(p,p'')$. This  is a well defined multiplication since the action is free.  By conditions (P'1), (P'3), we have that $(G_{\Sscript{1}}, G_{\Sscript{0}})$ is a transitive  groupoid with a canonical morphism of groupoids 
$$
\xymatrix@R=22pt{  P \times P  \ar@{->>}[d] \ar@<0.6ex>@{->}|-{}[rr] \ar@<-0.55ex>@{->}|-{}[rr] & & \ar@{->}|-{}[ll] P \ar@{->>}[d] \\ G_{\Sscript{1}} := (P\times P) /G \ar@<0.6ex>@{->}|-{}[rr] \ar@<-0.6ex>@{->}|-{}[rr] & & \ar@{->}|-{}[ll] P/G:=G_{\Sscript{0}}. }
$$

Conversely, given a  transitive groupoid $\gG$, and fix an object $x \in G_{\Sscript{0}}$. Set $G:= \gG^{\Sscript{x}}$ the isotropy group of $x$ and let $P:=\Sf{t}^{-1}(\{x\})$ be the set of all arrows with target this $x$, i.e.,~ the left star set of $x$. Consider the left $G$-action $G \times P \to P$ derived from the multiplication of $\gG$. 
Since $\gG$ is transitive, the triple $(P,G,\pi)$ satisfies then the above conditions (P'1)-(P'3), which means that it is a left principal $G$-set.

\section{Hopf algebroids: comodules algebras, principal bundles, and weak equivalences}\label{sec:HAlgd}
This section contains the definitions of commutative Hopf algebroids and theirs bicomodules algebras. All definitions are given in the algebraic way.  Nevertheless, we will use a slightly superficial language of presheaves, sufficiently enough  to make clearer the connection with the contents of Section \ref{sec:Grpd}.  

Parallel to subsections  \ref{ssec:biset}  and \ref{ssec:WE}, we present a brief contents on principal bibundles between Hopf algebroids and their connection with weak equivalences. Dualizable objects in the category of (right) comodules are treated in the last subsection, where we also proof some useful lemmata.

\subsection{Preliminaries and basic notations}\label{sec:1}
We work over a  commutative base field $\Bbbk$. Unadorned tensor product $-\tensor{}-$ stands for the tensor product of $\Bbbk$-vector spaces $-\tensor{\Bbbk}-$. By $\Bbbk$-algebra, or algebra,  we understand commutative $\Bbbk$-algebras, unless otherwise specified.  The category of (right) $A$-modules over an algebra $A$, is denoted by $\rmod{A}$. The $\Bbbk$-vector space of all $A$-\emph{linear maps} between two (right) $A$-modules $M$ and $N$, is denoted by $\hom{A}{M}{N}$. When $N=A$ is the regular module, we denote $M^*:=\hom{A}{M}{A}$.

Given two algebras $R, S$, we denote by $S(R):={\rm Alg}_{\Sscript{\Bbbk}}(S,R)$ the set of all $\Bbbk$-algebra maps from $S$ to $R$.   In what follows, a \emph{presheaf} of sets (of groups, or of groupoids) stands for a functor from the category of algebras ${\rm Alg}_{\Sscript{\Bbbk}}$ to the category of sets (groups, or  groupoids).
Clearly, to any algebra $A$, there is an associated presheaf which sends $C \to A(C)={\rm Alg}_{\Sscript{\Bbbk}}(A,C)$, thus, the presheaf represented by $A$.

For two algebra maps $\sigma: A \to T$ and $\gamma: B \to T$  we denote by ${}_{\scriptscriptstyle{\sigma}}T_{\scriptscriptstyle{\gamma}}$ (respectively, ${}_{\scriptscriptstyle{\sigma}}T$ or $T_{\scriptscriptstyle{\gamma}}$, if one of the algebra maps is the identity) the underlying $(A,B)$-bimodule of $T$ (respectively, the underlying $A$-module of $T$) whose left $A$-action is induced by $\sigma$ while its right $B$-action is induced by $\gamma$, that is, 
$$
a\,.\, t \,=\, \sigma(a)t,\quad t\,.\, b\,= \, t \gamma(b), \quad \text{ for every }\; a \in A,\; b \in B,\; t \in T.
$$ 

Assume there is an algebra map  $x \in  A(R)$. The extension functor $(-)_{\Sscript{x}}: \rmod{A} \to \rmod{R}$  is the functor which sends any  $A$-module $M$ to the extended $R$-module $M_{\Sscript{x}}=M \tensor{A}R$. In order to distinguish between two extension functors, we use the notation $M_{\Sscript{x}}:=M\tensor{x}R$ and $M_{\Sscript{y}}:=M\tensor{y}S$, whenever  another algebra map $y \in A(S)$ is given.

In the sequel we will use the terminology \emph{coring} (or \emph{cog\'ebro\"ide} as in \cite{Deligne:1990, Bruguieres:1994}) for coalgebra with possibly different left-right structures on its underlying modules over the base ring. We refer to \cite{BrzWis:CAC} for basic notions and properties of these objects.

\subsection{The $2$-category of Hopf algebroids}\label{ssec:H}
Recall from, {\em e.g.}, \cite{Ravenel:1986} 
that a {\em commutative Hopf algebroid}, or a {\em Hopf algebroid over a field $\Bbbk$}, is a pair $(A,\cH)$ of two  commutative $\Bbbk$-algebras,  together with algebra maps
$$
\etaup: A\tensor{}A \to \cH, \quad \varepsilon: \cH \to A, \quad \Delta: {}_{\scriptscriptstyle{\Sf{s}}}\cH_{\scriptscriptstyle{\Sf{t}}} \to {}_{\scriptscriptstyle{\Sf{s}}}\cH_{\scriptscriptstyle{\Sf{t}}}\,\tensor{A}\,{}_{\scriptscriptstyle{\Sf{s}}}\cH_{\scriptscriptstyle{\Sf{t}}}, \quad \mathscr{S}: {}_{\scriptscriptstyle{\Sf{s}}}\cH_{\scriptscriptstyle{\Sf{t}}} \to {}_{\scriptscriptstyle{\Sf{t}}}\cH_{\scriptscriptstyle{\Sf{s}}}
$$
and a  structure  $({}_{\scriptscriptstyle{\Sf{s}}}\cH_{\scriptscriptstyle{\Sf{t}}}, \Delta, \varepsilon)$ of an $A$-coring with $\mathscr{S}$ an $A$-coring map to the opposite coring. Here the source and the target are the algebra maps $\Sf{s}: A\to \cH$ and $\Sf{t}: A \to\cH$ defined by $\Sf{s}(a)=\etaup(a\tensor{}1)$ and $\Sf{t}(a)=\etaup(1\tensor{}a)$, for every $a \in A$. 
The map $\mathscr{S}$ is  called the \emph{antipode} of $\cH$ subject  to  the following equalities:
\begin{equation}\label{Eq:antipode}
\mathscr{S}^2 = \id, \quad \quad   \Sf{t}(\varepsilon(u))\,=\, \mathscr{S}(u_{\Sscript{(1)}})u_{\Sscript{(2)}}, \quad \quad \Sf{s}(\varepsilon(u))\,=\, u_{\Sscript{(1)}}\mathscr{S}(u_{\Sscript{(2)}}),\quad \text{for every  }\, u  \, \in \cH,
\end{equation}
where we used Sweedler's notation: $\Delta(u)= u_{\Sscript{(1)}}\tensor{A}u_{\Sscript{(2)}}$ and summation is understood. The algebras $A$ and $\cH$ are  called, respectively,  \emph{the base algebra} and \emph{the total algebra}  of the Hopf algebroid $(A, \cH)$.

As commutative Hopf algebra leads to an affine group scheme, a Hopf algebroid  leads to an affine groupoid scheme (i.e., a presheaf of groupoids).  More precisely, given a Hopf algebroid $(A, \cH)$ and  an  algebra $C$, reversing the structure of $(A,\cH)$ we have, in a natural way, a groupoid structure 
\begin{equation}\label{Eq:miacosa}
\mathscr{H}(C):\xymatrix@C=35pt{\cH(C)\ar@<1ex>@{->}|-{\scriptscriptstyle{\sf{s}^*}}[r] \ar@<-1ex>@{->}|-{\scriptscriptstyle{\sf{t}^*}}[r] & \ar@{->}|-{ \scriptscriptstyle{\varepsilon^*}}[l]A(C).}
\end{equation}
This structure is explicitly given as follows: the source and the target of a given arrow $g \in \cH(C)$ are, respectively,  $\Sf{s}^{\Sscript{*}}(g)=g \circ  \Sf{s}$ and $\Sf{t}^{\Sscript{*}}(g)=g \circ \Sf{t}$, the inverse is $g^{-1}=g \circ  \mathscr{S}$.  Given another arrow $f \in \cH(C)$ with $\Sf{t}^{\Sscript{*}}(f)=\Sf{s}^{\Sscript{*}}(g)$, then the groupoid multiplication is defined by the following algebra map
$$
g f: \cH \longrightarrow C, \quad\Big(   u \longmapsto f(u_{\Sscript{(1)}})g(u_{\Sscript{(2)}}) \Big),
$$ summation always understood. 
The identity arrow of an object $x \in A(C)$ is $ \varepsilon^{\Sscript{*}}(x)=x \circ  \varepsilon$.

The functor $\hH$ is referred to as \emph{the associated presheaf of groupoids} of the Hopf algebroid $(A,\cH)$, and the groupoids of equation \eqref{Eq:miacosa} are called \emph{the fibres of} $\hH$.  Depending on the handled situation, we will employ different notations  for the fibres of $\hH$ at an algebra $C$:
$$
\hH(C)\,:=\,\big(\cH(C),A(C)\big)\,:=\, \big( \hH_{\Sscript{1}}(C), \hH_{\Sscript{0}}(C)\big).
$$

The presheaf of groupoids $\hH^{\Sscript{op}}$ is defined to be the presheaf whose fibre at $C$ is the opposite groupoid $\hH(C)^{\Sscript{op}}$ (i.e., the same groupoid with the source interchanged by the target).

Examples of Hopf algebroids can be then proportioned using well known  constructions in groupoids, as we have seen in subsection \ref{ssec:basic}.

\begin{example}\label{exam:HAlgd1}
Let $A$ be an algebra and set $\cH:=A\tensor{}A$. Then the pair $(A,\cH)$ is clearly a Hopf algebroid with structure $\Sf{s}: A \to \cH$, $a \mapsto a\tensor{}1$, $\Sf{t}: A \to \cH$, $a \mapsto 1\tensor{}a$; $\Delta(a\tensor{}a') = (a\tensor{}1) \tensor{A} (1\tensor{}a')$, $\varepsilon(a\tensor{}a') = aa'$,  $\sS(a\tensor{}a') = a'\tensor{}a$.  Clearly, the fibres of the associated presheaf of groupoids $\hH$ are groupoids of pairs, as in Example \ref{exam:X}. Thus, $\hH \cong \big(\aA \times \aA,\aA\big)$ an isomorphism of presheaves of groupoids, where $\aA$ is the presheaf of sets attached to the algebra $A$.
\end{example}

\begin{example}\label{exam:HAlgd2}
Let $(B,\Delta, \varepsilon, \cS)$ be a commutative Hopf $\Bbbk$-algebra and $A$ a commutative right $B$-comodule algebra with coaction $A \to A\tensor{}B$, $a \mapsto a_{\Sscript{(0)}} \tensor{} a_{\Sscript{(1)}}$. This means that $A$ is right $B$-comodule and the coaction is an algebra map, see \cite[\S, 4]{Montgomery:1993}.  
Let $\bB$ be the affine $\Bbbk$-group attached to $B$. To any algebra $C$, one associated in a natural way, the following action groupoid as in Example \ref{exam:action}:
$$
(\aA \times \bB)(C): \xymatrix@C=35pt{A(C)\times B(C) \ar@<1ex>@{->}|-{\scriptscriptstyle{}}[r] \ar@<-1ex>@{->}|-{\scriptscriptstyle{}}[r] & \ar@{->}|-{ \scriptscriptstyle{}}[l]A(C),} 
$$  
where the source is given by the action $(x,g) \mapsto xg$ sending $a \mapsto (xg)(a)=x(a_{\Sscript{(0)}})g(a_{\Sscript{(1)}})$, and the target is the first projection. Consider, on the other hand, the algebra  $\cH= A\tensor{}B$ with  algebra extension $ \eta: A\tensor{}A \to  \cH$, $a'\tensor{}a \mapsto a'a_{\Sscript{(0)}}\tensor{}a_{\Sscript{(1)}}$. Then $(A,\cH)$ has  a structure of Hopf algebroid, known as a \emph{split Hopf algebroid}: 
$$
\Delta(a\tensor{}b) = (a\tensor{}b_{\Sscript{(1)}}) \tensor{A} (1_{\Sscript{A}}\tensor{}b_{\Sscript{(2)}}), \;\;\varepsilon(a\tensor{}b)=a\varepsilon(b),\;\; \sS(a\tensor{}b)= a_{\Sscript{(0)}}\tensor{}  a_{\Sscript{(1)}}\cS(b).
$$
Obviously, the  associated  presheaf of groupoids $\hH^{\Sscript{op}}$ (where $\hH$ is   the one associated to $(A,\cH)$) is canonically isomorphic to the action groupoids $\aA \times \bB$. Thus, we have an isomorphism   $\hH^{\Sscript{op}} \cong \big(\aA \times \bB,\aA\big)$ of presheaves of groupoids.
\end{example}

\begin{example}\label{exam:HAlgd3}
Let $A$ be an algebra and consider the commutative polynomial Laurent ring over $A\tensor{}A$, that is,   $\cH=(A\tensor{}A)[X,X^{-1}]$ with the canonical injection $\eta: A\tensor{}A \to \cH$. The pair $(A,\cH)$ is a Hopf algebroid with structure maps
$$
\Delta\big( (a\tensor{}a') X \big) \,=\, \big( (a\tensor{A}1) X \big) \tensor{A} \big( (1\tensor{}a') X \big),\;\; \varepsilon\big( (a\tensor{}a') X \big) \,=\, aa',\;\; \sS\big( (a\tensor{}a') X \big) \,=\, (a'\tensor{}a) X^{-1}.
$$
The fibres of the associated presheaf $\hH$ are described using the induced groupoid by the multiplicative affine $\Bbbk$-group, in the sense of Example \ref{exam:induced}. Precisely, take an algebra $C$, then $\cH(C)$ is canonically bijective to the set $A(C)\times \cG_{\Sscript{m}}(C)\times A(C)$, where $\cG_{\Sscript{m}}$ is the multiplicative affine $\Bbbk$-group.  This in fact induces, in a natural way,  an isomorphisms of groupoids $(\cH(C), A(C)) \cong \big(A(C)\times \cG_{\Sscript{m}}(C)\times A(C),A(C)\big)$, where the later is the induced groupoid by the group $\cG_{\Sscript{m}}(C)$, as in Example \ref{exam:induced}. As presheaves of groupoids,  we have then an isomorphism $\hH \cong \big(\aA \times \cG_{\Sscript{m}} \times \aA,\aA\big)$, where as above $\aA$ is the presheaf attached to the algebra $A$. 

There is in fact a more general construction: Take any commutative Hopf $\Bbbk$-algebra $B$, then for any algebra $A$, the pair $(A,A\tensor{}B\tensor{}A)$ admits a canonical structure of Hopf algebroid whose  associated presheaf is also of the form $\big(\aA \times \bB\times \aA,\aA\big)$, where $\aA$ and $\bB$ are as before. 
\end{example}

\begin{example}[Change of scalars]\label{exam:Scalars}
Let $(A,\cH)$ be a  Hopf algebroid over $\Bbbk$ and consider $L$ a field extension of $\Bbbk$. Then the pair of algebras $(\AL,\HL): =(A\tensor{\Bbbk}L, \cH\tensor{\Bbbk}L)$ admits, in a canonical way, a structure of Hopf algebroid over the field $L$. The structure maps are denoted using the subscript $L$, i.e., $\sL, \tL, \eL, ..$. If we denote by $\hHL: \Algl \to \Sf{Grpds}$ the presheaf of groupoids associated to $(\AL, \HL)$, then the usual hom-tensor adjunction shows  that  $\hHL$ factors through the forgetful functor $\Algl \to \Algk$. That is we have a commutative diagram of functors:
$$
\xymatrix@C=50pt{ \Algl \ar@{->}@/_2pc/^-{\hHL}[rr] \ar@{->}^-{\oO}[r] & \Algk \ar@{->}^-{\hH}[r]  & \Sf{Grpds}. }
$$ 
\end{example}

The notion of character group in commutative Hopf algebras context is naturally extended to that of character groupoids in commutative Hopf algebroids:
\begin{definition}\label{def:characters}
Let $(A, \cH)$ be a Hopf algebroid over a field $\Bbbk$ and $\hH$ its associated presheaf of groupoids. The {\em character groupoid} of $(A,\cH)$ is the fibre groupoid $\mathscr{H}(\Bbbk)=(\cH(\Bbbk),A(\Bbbk))$ at the base field $\Bbbk$.  Notice that the character groupoid might be empty (i.e.,~could be a category without objects). 
\end{definition}

The following definition, which we will frequently used  in the sequel, can be found in \cite[page 129]{Deligne:1990}. It is noteworthy to mention that in our case (i.e., the case of affine $\Bbbk$-schemes),  the base presheaf $\hHo$ of the presheaf $\hH$ associated to a given Hopf algebroid $(A,\cH)$ with $A\neq 0$, is always non empty. That is, the condition $\hHo \neq \emptyset$ in \emph{op.cit.}, is satisfied since $\hHo$ is represented by $A$ and there is always a field extension $L$ of $\Bbbk$ such that $A(L) \neq \emptyset$ as $A \neq 0$ (i.e., it have maximal ideals).

\begin{definition}\label{def:1}
Let $(A,\cH)$ be a Hopf algebroid and $\hH$ its associated presheaf of groupoids.  Given an algebra $C$, consider the fibre groupoid $\mathscr{H}(C)$. Two objects $x, y \in A(C)$ are said to be \emph{locally isomorphic}, in the sense of the  fpqc topology (or \emph{fpqc locally isomorphic}),
if there exists a faithfully flat extension $p: C \to C'$ and an arrow $g \in \cH(C')$ such that 
$$
p \circ  x \,=\, g \circ  \Sf{s}, \quad \text{ and } \quad p \circ  y \,= \, g \circ  \Sf{t}.
$$ 
We say that any two objects of $\mathscr{H}$ are  \emph{fpqc locally isomorphic} (without specifying the algebra $C$), if for any algebra $C$ and any two objects $x, y \in A(C)$, $x$ and $y$ are fpqc  locally isomorphic. 
\end{definition}

\begin{remark}\label{rem:Lkk}
In case we have $\hH(C) = \emptyset $, for some algebra $C$,  the condition of Definition \ref{def:1} is conventionally assumed to be verified for this $C$.  On the other hand, it is not difficult to check that  if there exists a field extension $L$ of $\Bbbk$ such that two objects of $ \hHL$ are fpqc locally isomorphic, then any  two objects of $\hH$  are also fpqc locally isomorphic. The converse is not immediate and follows from Theorem \ref{thm:A} below. More precisely, if $\hHo(\Bbbk) \neq \emptyset$ and any two objects of $\hH$ are fpqc locally isomorphic, then any two objects  of $\hHL$ are also fpqc locally isomorphic for any field extension $L$ of $\Bbbk$.
\end{remark}

If the presheaf $\mathscr{H}$ is \emph{fibrewise transitive}, that is,  each of its fibres  $\mathscr{H}(C)$ is a transitive groupoid, then obviously any two objects of $\mathscr{H}$ are fpqc locally isomorphic. For instance, this is the case for the Hopf algebroid $(A,A\tensor{}A)$, since in this case each of the groupoid's $\hH(C)$ is the groupoid of pairs, see Example \ref{exam:HAlgd1}. The same holds true for the class of Hopf algebroids described in Example \ref{exam:HAlgd3}.

A \emph{morphism} $\B{\phi}: (A,\cH) \to (B,\cK)$ {\em of Hopf algebroids} 
consists of a pair $\B{\phi}=(\phi_{\scriptscriptstyle{0}},\phi_{\scriptscriptstyle{1}})$ 
of algebra maps  $\phi_{\scriptscriptstyle{0}}: A \to B$ and $\phi_{\scriptscriptstyle{1}}: \cH \to \cK$ that are compatible,  in a canonical way,  with the structure maps of both $\cH$ and $\cK$. 
That is, the equalities 
\begin{eqnarray}
\phi_{\scriptscriptstyle{1}} \circ \Sf{s} \,\,=\,\,  \Sf{s} \circ \phi_{\scriptscriptstyle{0}}, & &
\phi_{\scriptscriptstyle{1}} \circ \Sf{t} \,\,=\,\,  \Sf{t} \circ \phi_{\scriptscriptstyle{0}},   \\
\Delta \circ \phi_{\scriptscriptstyle{1}} \,\,=\,\,  \chi \circ (\phi_{\scriptscriptstyle{1}} \tensor{A}\phi_{\scriptscriptstyle{1}})  \circ \Delta, && \varepsilon \circ \phi_{\scriptscriptstyle{1}}\,\,=\,\, \phi_{\scriptscriptstyle{0}} \circ \varepsilon, \\ 
\mathscr{S} \circ \phi_{\scriptscriptstyle{1}} \,\,=\,\, \phi_{\scriptscriptstyle{1}} \circ \mathscr{S}, 
\end{eqnarray}
hold, where $\chi$ is the obvious map $\chi: \cK\tensor{A}\cK \to \cK\tensor{B}\cK$,   and  where no distinction between the structure maps of $\cH$, $\cK$ was made. 
Clearly, any morphism $\B{\phi}: (A,\cH) \to (B,\cK)$ of Hopf algebroids induces  (in the opposite way) a morphism  between the associated presheaves of groupoids, which is given over each fibre by $$(\phi_{\Sscript{0}}{}^* , \phi_{\Sscript{1}}{}^*): \mathscr{K}(C)=(\cK(C),B(C)) \longrightarrow \mathscr{H}(C)=(\cH(C),A(C)), \;\; \text{ sending }\;(g,x) \mapsto (g \circ \phi_{\Sscript{1}}, x \circ \phi_{\Sscript{0}}).$$

In this way, the construction in the following example  corresponds to the construction of the induced groupoid as expounded in Example \ref{exam:induced}.

\begin{example}[Base change]\label{exam:Base change}
Given a Hopf algebroid $(A,\cH)$ and an algebra map $\phi:A \to B$, then the pair of algebras
$$
(B,\cH_{\scriptscriptstyle{\phi}}):=(B,B\tensor{A}\cH\tensor{A}B)
$$ is a Hopf algebroid known as \emph{the base change Hopf algebroid} of $(A,\cH)$, and $(\phi, \phi_1): (A,\cH) \to (B,\cH_{\scriptscriptstyle{\phi}})$ is a morphism of Hopf algebroids, where $\phi_1: \cH \to \cH_{\Sscript{\phi}}$  sends $u \mapsto 1_{\Sscript{B}}\tensor{A}u \tensor{A}1_{\Sscript{B}}$. Moreover, as in the case of groupoids, see subsection \ref{ssec:basic}, any morphism $\B{\phi}: (A,\cH) \to (B,\cK)$ factors through  \emph{the base change morphism} $(A,\cH) \to (B,\cH_{\Sscript{\phi_{\Sscript{0}}}})$, by using the map $\cH_{\Sscript{\phi_{0}}} \to \cK$ sending $b\tensor{A}u\tensor{A}b' \mapsto bb'\phi_{\Sscript{1}}(u)$.

The associated presheaf of groupoids $\hH_{\Sscript{\phiup}}$ of the Hopf algebroid $(B, \cH_{\Sscript{\phi}})$ is fiberwise computed as the induced groupoid (see Example \ref{exam:induced}) of $\hH$ along the map $\phiup: \bB \to \aA$ where $\bB$ and $\aA$ are the presheveas of sets associated to $B$ and $A$, respectively. 
\end{example}

The aforementioned relation with the induced groupoids comes out as  follows. Take an algebra $C$ and consider the associated groupoid $\hH(C)$. Then the map $\phi^{\Sscript{*}}: B(C) \to A(C)$ leads, as in Example \ref{exam:induced}, to the induced groupoid  $\hH(C)^{\Sscript{\phi(C)}}$. This in fact determines a presheaf of groupoids $C \to \hH(C)^{\Sscript{\phi^{\Sscript{*}}}}$ which can be easily shown to be represented by the pair of algebras $(B,\cH_{\Sscript{\phi}})$.
  
We finish this subsection by recalling the construction of the $2$-category of flat Hopf algebroids. \\
A Hopf algebroid $(A,\cH)$ is said to be \emph{flat}, when ${}_{\Sscript{\Sf{s}}}\cH$ (or $\cH_{\Sscript{\Sf{t}}}$) is a flat $A$-module. Notice, that in this case $\Sf{s}$ as well as $\Sf{t}$ are faithfully flat extensions.
As was mentioned before,   groupoids, functors, and natural transformations form a $2$-category.   Analogously (flat) Hopf algebroids  over the ground field $\Bbbk$ form a $2$-category, as was observed in \cite[\S3.1]{Naumann:07}.   Precisely, $0$-cells are Hopf algebroids, or even flat ones, $1$-cells are morphisms of Hopf algebroids, and for 
two $1$-cells $(\phi_{\Sscript{0}},\phi_{\Sscript{1}}), (\psi_{\Sscript{0}},\psi_{\Sscript{1}}): (A, \cH) \to (B, \cK)$, a $2$-cell $\fk{c}:   (\phi_{\Sscript{0}},\phi_{\Sscript{1}}) \to (\psi_{\Sscript{0}},\psi_{\Sscript{1}})$ is defined to be an algebra map $\fk{c}: \cH \to B$ 
that makes the diagrams
\begin{equation}\label{Eq:2cells}
\begin{gathered}
\xymatrix@R=20pt{ \cH \ar@{->}^-{\fk{c}}[r] & B \\ A \ar@{->}_-{\phi_{\scriptscriptstyle{0}}}[ru] \ar@{->}^-{\sf{s}}[u] & }\qquad  \xymatrix@R=20pt{ \cH \ar@{->}^-{\fk{c}}[r] & B \\ A \ar@{->}_-{\psi_{\scriptscriptstyle{0}}}[ru] \ar@{->}^-{\sf{t}}[u] & } \qquad \xymatrix@R=20pt{ \cH \ar@{->}^-{\Delta}[rr] \ar@{->}_-{\Delta}[d] &  & \cH\tensor{A}\cH \ar@{->}|-{m_{\Sscript{\cK}}\,\circ\, \big(\phi_{\scriptscriptstyle{1}}\tensor{A}(\sf{t}\circ\fk{c})\big)}[d]  \\ \cH\tensor{A}\cH \ar@{->}_-{m_{\Sscript{\cK}}\, \circ\,\big((\sf{s}\circ {\fk{c})}\,\tensor{A}\psi_{\scriptscriptstyle{1}}\big)}[rr]  & & \cK }
\end{gathered}
\end{equation}
commutative, where $m_{\Sscript{\cK}}$ denotes the multiplication of $\cK$.
The identity $2$-cell for $(\phi_{\Sscript{0}},\phi_{\Sscript{1}})$ is given by $1_{\scriptscriptstyle{\phi}}:=  \phi_{\Sscript{0}} \circ \varepsilon$.
The tensor product (or vertical composition) of $2$-cells  
is given as
$$
\xymatrix{  \fk{c}' \circ \fk{c}:  (\phi_{\Sscript{0}},\phi_{\Sscript{1}}) \ar@{->}^-{\fk{c}}[r] & (\psi_{\Sscript{0}},\psi_{\Sscript{1}}) \ar@{->}^-{\fk{c}'}[r] & (\xi_{{0}},\xi_{{1}}),}
$$ where 
\begin{equation}
\label{Eq:verticalComp}
\fk{c}' \B{\circ} \fk{c}: \cH  \to B, \quad u \mapsto \fk{c}(u_{{(1)}}) \fk{c}'(u_{{(2)}}).
\end{equation}

\subsection{Comodules, bicomodules (algebras), and  presheaves of orbit sets}\label{ssec:C}
A \emph{right $\cH$-comodule} is a pair $(M, {\varrho})$ consisting of right $A$-module $M$ and right $A$-linear map (referred to as the \emph{coaction}) ${\varrho}: M \to M  \tensor{A} {}_{\Sscript{\Sf{s}}}\cH$, $m \mapsto m_{\Sscript{(0)}} \tensor{A} m_{\Sscript{(1)}}$ (summation understood) satisfying the usual counitary and coassociativity properties.   Morphisms between right $\cH$-comodules (or right $\cH$-\emph{colinear map}) are $A$-linear maps compatible with both coactions. The \emph{category of all right $\cH$-comodules} is denoted by $\rcomod{\cH}$. This is a symmetric monoidal $\Bbbk$-linear category with identity object $A$ endowed with the coaction $\Sf{t}:A \to \cH_{\Sscript{\Sf{t}}} \cong A\tensor{A}{}_{\Sscript{\Sf{s}}}\cH_{\Sscript{\Sf{t}}}$. 

The tensor product in $\rcomod{\cH}$ is defined via the so called \emph{the diagonal coaction}. Precisely, given $(M, \varrho)$ and $(N, \varrho)$ two (right)  $\cH$-comodules. Then the tensor product $M\tensor{A}N$ is endowed with the following (right) $\cH$-coaction:
\begin{equation}\label{Eq:tensorproduct}
\varrho_{\Sscript{M\tensor{A}N}}: M \tensor{A}N \longrightarrow (M\tensor{A}N)\tensor{A}\cH, \quad \Big( m\tensor{A}n \longmapsto (m_{\Sscript{(0)}}\tensor{A}n_{\Sscript{(0)}}) \tensor{A} m_{\Sscript{(1)}}n_{\Sscript{(1)}} \Big). 
\end{equation} 

The vector space of all $\cH$-colinear maps between two comodules $(M,\varrho)$ and $(N,\varrho)$ will be denoted by $\cohom{\cH}{M}{N}$, and  the endomorphism ring   by $\CoEnd{\cH}{M}$.

Inductive limit and cokernels do exist in $\rcomod{\cH}$, and can be computed in $A$-modules. Furthermore, it is well known that the underlying module ${}_{\Sscript{\Sf{s}}}\cH$ is flat if and only if $\rcomod{\cH}$ is a Grothendieck category and the forgetful functor $\mathscr{U}_{\Sscript{\cH}}: \rcomod{\cH} \to \rmod{A}$ is  exact.  As it can be easily checked, the forgetful functor $\mathscr{U}_{\Sscript{\cH}}$ has a right adjoint functor $-\tensor{A}{}_{\Sscript{\Sf{s}}}\cH: \rmod{A} \to \rcomod{\cH}$. 

The full subcategory of right $\cH$-comodules whose underlying $A$-modules are finitely generated is denoted by $\frcomod{\cH}$. 
The category of left $\cH$-comodules is analogously defined, and it is isomorphic via the antipode to the category of right $\cH$-comodules.

\emph{A (right) $\cH$-comodule algebra} can be defined as a commutative monoid in the symmetric monoidal category $\rcomod{\cH}$. This is a commutative algebra extension $\sigma:A \to R$ where the associated $A$-module $R_{\Sscript{\sigma}}$ is also a (right) $\cH$-comodule whose coaction $\varrho_{\Sscript{R}}:R_{\Sscript{\sigma}} \to R_{\Sscript{\sigma}} \tensor{A} {}_{\Sscript{\Sf{s}}}\cH$ is an algebra map, which means that
$$
\varrho_{\Sscript{R}}(1_{\Sscript{R}})\,\,=\,\,  1_{\Sscript{R}}\tensor{A} 1_{\Sscript{\cH}},\quad \varrho_{\Sscript{R}}(rr')\,\,=\,\, r_{\Sscript{(0)}}r'_{\Sscript{(0)}}\tensor{A} r_{\Sscript{(1)}}r'_{\Sscript{(1)}}, \; \text{ for every }\, r, r' \in R.
$$ 

This of course induces a (right) $\mathscr{H}$-action on the presheaf of sets $\mathscr{R}$ associated to $R$. Precisely, given an algebra $C$, consider the map $\sigma^{\Sscript{*}}:R(C) \to A(C)$ sending $x \mapsto x \circ \sigma$, and set 
\begin{equation}\label{Eq:action}
R(C)\,  \due \times {\scriptscriptstyle{\sigma^*}} {\, \scriptscriptstyle{\Sf{s}^*}} \,\mathscr{H}_{\Sscript{1}}(C) \to R(C), (x,g) \mapsto xg,\,\text{ where } xg: R \to C, \, r \mapsto x(r_{\Sscript{(0)}})\, g(r_{\Sscript{(1)}}).
\end{equation}
Then this defines, in a natural way,  a  right action of the groupoid $\mathscr{H}(C)$ on the set $R(C)$, in the sense of Definition \ref{def:Gset}. Equivalently such an action  can be expressed as a pair morphism of presheaves $\rR \to \hH_{\Sscript{0}}$ and $\rR\,  \due \times {\Sscript{\sigma^*}} {\, \Sscript{\Sf{s}^*}} \, \hH_{\Sscript{1}} \to \rR$ satisfying pertinent compatibilities.  In this way,  an action of a presheaf of groupoids on a  presheaf of sets,  can be seen as a natural generalization to the groupoids framework of  the notion of an action of group scheme on a scheme \cite[n$^\text{o}$3, page 160]{DemGab:GATIGAGGC}, or more formally, as a generalization of the notion of  "objet \`a  groupe d'op\'erateurs \`a droite" \cite[Chapitre III, \S 1.1]{Giraud:1971}.

For a right $\cH$-comodule algebra $(R,\sigma)$, we have the \emph{subalgebra of coinvariants} defined by 
\begin{equation}\label{Eq:coinH}
R^{\Sscript{coinv_{\cH}}}:=\Big\{ r \in R|\,\, \varrho_{\Sscript{R}}(r)=r\tensor{A}1_{\Sscript{\cH}} \Big\}.
\end{equation}
Denote by $\rR^{\Sscript{\hH}}$ the presheaf of sets represented by the algebra  $R^{\Sscript{coinv_{\cH}}}$. On the other hand, we have the presheaf defining the orbit sets, which is given as follows:
Take an algebra $C$ and consider the action \eqref{Eq:action}, we then obtain the orbits set  $R(C)/\hH(C)$. Clearly this establishes a functor: $C \longrightarrow R(C)/\hH(C)$ yielding a presheaf $\oO_{\Sscript{\hH}}(\rR)$ with a canonical morphism of presheaves $ \oO_{\Sscript{\hH}}(\rR) \longrightarrow \rR^{\Sscript{\hH}}$. 

\begin{remark}\label{remak:Qlevoyhacer}
An important example of the previous construction is the case of the right $\cH$-comodule algebra $(A,\Sf{t})$. In this case  we have a commutative diagram of presheaves:
$$
\xymatrix@R=7pt{  \aA \ar@{->}^-{\Sscript{\tau}}[rr] \ar@{->}_-{\Sscript{\zeta}}[rd] & & \aA^{\Sscript{\hH}} \\ &  \oO_{\Sscript{\hH}}(\aA) \ar@{->}^-{}[ru] & }
$$
where as before $\aA$ is the presheaf represented by the algebra $A$ and  $\aA^{\Sscript{\hH}}$ is represented by $A^{\Sscript{coinv_{\cH}}}$. 

Notice that  the presheaf $\oO_{\Sscript{\hH}}(\aA)$ is  not necessarily represented by $A^{\Sscript{coinv_{\cH}}}$, thus, the right hand map in the previous diagram is not in general an isomorphism of presheaves, see \cite[page 54]{Powell:2008}.

In this direction, both $\aA \, \due \times {\Sscript{\tau}} {\, \Sscript{\tau}}\, \aA$ and $ \aA \, \due \times {\Sscript{\zeta}} {\, \Sscript{\zeta}}\, \aA$  enjoy a structure of presheaves of groupoids with fibres are groupoids of pairs described in Examples \ref{exam:X}. Nevertheless, $ \aA \, \due \times {\Sscript{\zeta}} {\, \Sscript{\zeta}}\, \aA$ is not necessarily representable. Furthermore,  there is a commutative diagram
$$
\xymatrix@R=7pt{  \hH \ar@{->}^-{}[rr] \ar@{->}^-{}[rd] & & \aA\, \due \times {\Sscript{\tau}} {\, \Sscript{\tau}}\, \aA \\ &  \aA \, \due \times {\Sscript{\zeta}} {\, \Sscript{\zeta}}\, \aA \ar@{->}^-{}[ru] & }
$$
of presheaves of groupoids.
\end{remark}

Given two Hopf algebroids $(A,\cH)$ and $(B,\cK)$,  \emph{the category of $(\cH,\cK)$-bicomodules} is defined as follows. An  object in this category is a  triple $(M, {\lambda}, {\varrho})$ consisting of left $\cH$-comodule $(M,{\lambda})$ and right $\cK$-comodule $(M,{\varrho})$ such that ${\lambda}$ is a morphism of right $\cK$-comodules, or equivalently ${\varrho}$ is a morphism of left $\cH$-comodules. Morphisms between bicomodules are simultaneously left and right comodules morphisms. On the other hand, the pair of tensor product $(A\tensor{}B, \cH^{\Sscript{o}}\tensor{}\cK)$ admits, in a canonical way, a structure of Hopf algebroid, where $(A,\cH^{\Sscript{o}})$ is the opposite Hopf algebroid (i.e., the source and the target are interchanged, or equivalently, the fibres of the associated presheaf are the opposite groupoids $\hH(C)^{\Sscript{op}}$).  This is \emph{the tensor product Hopf algebroid}, and its  category of right comodules is canonically identified with the category of $(\cH,\cK)$-bicomodules.  Thus bicomodules form also a symmetric monoidal $\Bbbk$-linear category. 

A \emph{bicomodule algebra} is a bicomodule which is simultaneously a left comodule algebra and right comodule algebra. As above, by using the actions of equation \eqref{Eq:action} a bicomodule algebra leads to a presheaf of groupoid bisets. That is, a presheaf with fibres groupoid-bisets, in the sense of Definition \ref{def:biset}.

\subsection{Weak equivalences and principal bundles between Hopf algebroids }\label{ssec:W}
Any morphism $\B{\phi}: (A,\cH) \to (B,\cK)$ of Hopf algebroids induces a symmetric monoidal $\Bbbk$-linear functor 
$$
\B{\phi}_*:= \mathscr{U}_{\Sscript{\cH}}(-)\tensor{A}B:  \rcomod{\cH} \longrightarrow \rcomod{\cK},
$$
where, for any $\cH$-comodule $(M, \varrho)$, the  $\cK$-comodule structure of $M\tensor{A}B$ is given by  $$M\tensor{A}B \longrightarrow (M\tensor{A}B)\tensor{B}\cK, \quad m\tensor{A}b \longmapsto (m_{\Sscript{(0)}}\tensor{A}1_{\Sscript{B}}) \tensor{B} \phi_{\Sscript{1}}(m_{\Sscript{(1)}})\Sf{t}(b).$$  
Following \cite[Definition 6.1]{HovStr:CALEHT}, $\B{\phi}$ is said to be a \emph{weak equivalence} whenever $\B{\phi}_*$ is an equivalence of categories. In this case, $\rcomod{\cH}$ and $\rcomod{\cK}$ are equivalent as symmetric monoidal $\Bbbk$-linear categories.

Notice that if $\B{\phi}$ is a weak equivalence, then so is the associated morphism between the tensor product Hopf algebroids $\B{\phi}^{\Sscript{o}}\tensor{}\B{\phi}: (A\tensor{}A, \cH^{\Sscript{o}}\tensor{}\cH) \to (B\tensor{}B, \cK^{\Sscript{o}}\tensor{}\cK)$, which induces then a symmetric monoidal equivalence between the categories of $\cH$-bicomodules and $\cK$-bicomodules.  

Two  Hopf algebroids $(A,\cH)$ and $(B,\cK)$ are said to be \emph{weakly equivalent} if there exists a diagram  
$$
\xymatrix@R=7pt{  &  (C, \cJ) & \\ (A, \cH) \ar@{->}[ru] & & \ar@{->}[lu] (B,\cK),}
$$
of weak equivalences.

As was shown in \cite{Kaoutit/Kowalzig:14} weak equivalences between flat Hopf algebroids are strongly related to  principal bi-bundles. Such a relation is in part a consequence of the analogue one for groupoids as was shown in Proposition \ref{prop:pb} (see Remark \ref{remak:tusojosverdes}).

Recall that, for two flat Hopf algebroids $(A,\cH)$ and $(B,\cK)$, a \emph{left principal $(\cH,\cK)$-bundle} is a three-tuple $(P,\alpha,\beta)$ which consists of   diagram of commutative algebras $\alpha: A \rightarrow P \leftarrow B: \beta$ where the $(A,B)$-bimodule ${}_{\Sscript{\alpha}}P_{\Sscript{\beta}}$ enjoys a structure of an $(\cH,\cK)$-bicomodule algebra such that 
\begin{enumerate}[({PB}1)]
\item $\beta: B \to P$ is a faithfully flat extension (the local triviality of the bundle in the fpqc topology);
\item the canonical map $\can{\Sscript{P,\,\cH}}: P\tensor{B}P \to \cH\tensor{A}P$ sending $p\tensor{B}p' \mapsto p_{\Sscript{(0)}}\tensor{A}p_{\Sscript{(1)}}p'$ is bijective.
\end{enumerate}
Observe that these two conditions, in conjunction with the faithfully flat descent theorem \cite[Theorem 5.9]{K/GT:2004}, show that $P^{\Sscript{coinv_{\cH}}}=B$, see \eqref{Eq:coinH} for the notation. 

The notion of principal bundles is a natural generalization of the notion of \emph{Torsor}, where the group object is replaced by groupoid object, see \cite[D\'efinition 1.4.1, page 117]{Giraud:1971} and \cite[Chapter III, \S4]{DemGab:GATIGAGGC}. In case of Hopf algebras over commutative rings, these objects are termed \emph{Hopf Galois extensions}, see \cite{Montgomery:1993, Schau:HGABGE}.

Right principal bundles and bi-bundles are clearly understood.  For instance to each left principal bundle $(P,\alpha,\beta)$, one can define a right principal bundle on the opposite bicomodule $P^{\Sscript{co}}$. As in the case of groupoids, see subsection \ref{ssec:biset}, a simpler example of left principal bundle is \emph{the unit bundle} $\mathscr{U}(\cH)$ which is $\cH$ with its canonical structure of $\cH$-bicomodule algebra. A \emph{trivial bundle} attached to a given morphism of Hopf algebroids $\B{\phi}: (A,\cH) \to (B,\cK)$, is the one whose underlying bicomodule algebra is of the form  $P:=\cH\tensor{A}B$, that is, the pull-back bundle $\B{\phi}^{*}(\uU(\cH))$ of the unit bundle $\uU(\cH)$.

Parallel to subsection \ref{ssec:biset}, for any bicomodule algebra, and thus for any left principal bundle, one can associate the so called \emph{the two-sided translation  Hopf algebroid}, which is denoted by $(P,\cH \lJoin P \rJoin \cK)$. The underlying pair of algebras is $(P,\cH_{\Sscript{\Sf{s}}}\tensor{}{}_{\Sscript{\alpha}}P_{\Sscript{\beta}}\tensor{}{}_{\Sscript{\Sf{s}}}\cK)$ and its structure of Hopf algebroid is given as follows:
\begin{itemize}
\item the source and target are given by 
$$
\Sf{s}(p)\,:=\, 1_{\scriptscriptstyle{\cH}} \tensor{A}p\tensor{B}1_{\scriptscriptstyle{\cK}}, \quad \Sf{t}(p)\,:=\, \mathscr{S}(p_{(-1)}) \tensor{A}p_{(0)}\tensor{B}p_{(1)}; 
$$
\item the comultiplication and counit are given by:
$$
\Delta(u\tensor{A}p\tensor{B}w) :=  \big( u_{(1)}\tensor{A}p\tensor{B}w_{(1)}\big)\tensor{P}\big( u_{(2)}\tensor{A}1_{\scriptscriptstyle{P}}\tensor{B}w_{(2)}\big) ,\;\, \varepsilon(u\tensor{A}p\tensor{B}w):= \alpha\big( \varepsilon(u)\big)p \beta\big( \varepsilon(w)\big);
$$
\item whereas the antipode is defined as:  
$$ 
\mathscr{S}\big( u\tensor{A}p\tensor{B}w\big)\,:=\, \mathscr{S}(u p_{(-1)})\tensor{A}p_{(0)}
\tensor{B}p_{(1)}\mathscr{S}(w).$$
\end{itemize}
Furthermore, there is a diagram
\begin{equation}\label{Eq:triangleI}
\xymatrix@R=8pt{ & (P,\cH \lJoin P \rJoin \cK) & \\ (A,\cH) \ar@{->}^-{{\B{\alpha}=(\alpha, \,\alpha_1)}}[ur] & & \ar@{->}_-{{\B{\beta}=(\beta, \,\beta_1)}}[ul] (B,\cK) }
\end{equation}
of Hopf algebroids,
where $\alpha_{1}$ and $\beta_1$ are, respectively,  the maps $u \mapsto u \tensor{A}1_{\Sscript{P}} \tensor{B}1_{\scriptscriptstyle{\cK}}$ and  
$w \mapsto 1_{\scriptscriptstyle{\cH}} \tensor{A} 1_{\Sscript{P}} \tensor{B} w $. It is easily checked that $(P,\cH \lJoin P \rJoin \cK)$ is a flat Hopf algebroid whenever $(A,\cH)$ and $(B,\cK)$ they are so.

\begin{remark}\label{remark:serastu}
It is noteworthy to mention that the fibres of the presheaf  associated  to a left principal bundle are not necessarily  principal bisets over the fibres groupoids, in  the sense of  Definition  \ref{def:pbset} (but possibly the entry presheaf is locally so in the fpqc topology sense). 
To be precise, let $\pP$ denote as before the presheaf of sets associated to the algebra $P$. This is a presheaf of $(\hH,\kK)$-bisets, that is,  using left and right actions of equation  \eqref{Eq:action},  for any algebra $C$, we have that the fibre $\pP(C)$ is actually an $(\hH(C), \kK(C))$-biset as in Definition  \ref{def:biset}. However,  $\pP(C)$ is not necessarily a left principal biset. 
Nevertheless, it is  easily seen  that the associated presheaf of two-sided translation groupoids is represented by the two-sided translation Hopf algebroid $(P,\cH \lJoin P \rJoin \cK)$.
Lastly, as in \cite{Kaoutit/Kowalzig:14}, two weakly equivalent flat Hopf algebroids are shown  to be connected by  a principal bibundle, for which the diagram \eqref{Eq:triangleI} becomes a diagram of weak equivalences, see \emph{op.~ cit.}\  for more characterizations of weak equivalences. 
\end{remark}

\subsection{Dualizable right comodules}\label{ssec:dual} Recall that a  (right) $\cH$-comodule $(M,\varrho_{\Sscript{M}})$ is said to be \emph{dualizable}, if there is another (right) $\cH$-comodule $(N,\varrho_{\Sscript{N}})$ and two morphisms of comodules 
\begin{equation}\label{Eq:evdb}
\Sf{ev}:(N\tensor{A}M, \varrho_{\Sscript{N\tensor{A}M}}) \to (A,\Sf{t}),\; \text{ and }\; \Sf{db}: (A,\Sf{t}) \to (M\tensor{A}N, \varrho_{\Sscript{M\tensor{A}N}})
\end{equation}
satisfying, up to natural isomorphisms, the usual triangle properties. Taking the underlying $A$-linear maps $(\Sf{ev}, \Sf{db})$ and using these triangle properties, one shows that $\mathscr{U}_{\Sscript{\cH}}(N) \cong M^*=\hom{A}{M}{A}$. Thus, the underlying $A$-module of any dualizable comodule is finitely generated and projective.  Moreover, the dual of $(M,\varrho_{\Sscript{M}})$ in the category $\rcomod{\cH}$ is, up to an isomorphism, the comodule $(M^*, \varrho_{\Sscript{M^*}})$ with the following coaction:
\begin{equation}\label{Eq:star}
\varrho_{\Sscript{M^*}}: M^* \to M^*\tensor{A}\cH, \quad \Big(\varphi \longmapsto e_{\Sscript{i}}^*\tensor{A}\Sf{t}(\varphi(e_{\Sscript{i,\,(0)}}))\mathscr{S}(e_{\Sscript{i,\,(1)}}) \Big),
\end{equation}
where $\{e_i,e_i^*\}$ is a dual basis for $M_A$, that is, $\Sf{db}(1_{\Sscript{A}}) =\sum_i e_{i}\tensor{A}e_i^*$.  The converse  holds true as well, that is, dualizable objects in $\rcomod{\cH}$ are, up to natural isomorphisms, precisely the objects of the subcategory $\frcomod{\cH}$ which are projective as $A$-modules.  This is a well known fact which will be implicitly used below:

\begin{lemma}\label{lema:dualizable}
Let $(M, \varrho_{\Sscript{M}})$ be a right $\cH$-comodule whose underlying $A$-module $M$ is finitely generated and projective, and consider $(M^*, \varrho_{\Sscript{M^*}})$ as right $\cH$-comodule with the coaction given by \eqref{Eq:star}. Then $(M,\varrho_{\Sscript{M}})$ is a dualizable object in $\rcomod{\cH}$ with dual object $(M^*,\varrho_{\Sscript{M^*}})$. In particular, the full subcategory of dualizable right $\cH$-comodules consists of those comodules with finitely generated and projective underlying $A$-modules.
\end{lemma}

The following lemma will be used in the sequel and the $\Bbbk$-algebras involved in it are not necessary commutative.
\begin{lemma}\label{lema:MN}
Let $(A,\mathfrak{H})$ and $(A',\mathfrak{H}')$ be two corings. Assume that there are two bimodules ${}_{\Sscript{B}}M_{\Sscript{A}}$ and  ${}_{\Sscript{B'}}M'_{\Sscript{A'}}$ such that $(M,\varrho_{\Sscript{M}})$ and $(M,'\varrho_{\Sscript{M'}})$ are, respectively, $(B,\mathfrak{H})$-bicomodule and $(B',\mathfrak{H}')$-bicomodule (here $B$ and $B'$ are considered as $B$-coring and $B'$-coring in a trivial way), and that $M_{\Sscript{A}}$, $M'_{\Sscript{A'}}$ are finitely generated and projective modules with dual bases, respectively, $\{m_i,m_i^*\}$ and $\{n_j,n_j^*\}$.  
\begin{enumerate}[(1)]
\item If the associated canonical map:
\begin{equation}\label{Eq:can}
\Sf{can}_{\Sscript{M}}: M^*\tensor{B}M \to \mathfrak{H}, \quad \Big( m^*\tensor{A}m \longmapsto m^*(m_{\Sscript{(0)}}) m_{\Sscript{(1)}} \Big)
\end{equation}
is injective, then ${\rm End}^{\Sscript{M^*\tensor{B}M}}(M)\,=\, {\rm End}^{\Sscript{\mathfrak{H}}}(M)$, where $M^*\tensor{B}M$ is the standard $A$-coring \cite{Bruguieres:1994}, or the comatrix $A$-coring \cite{Kaoutit/Gomez:2003a}.

\item  If both $\Sf{can}_{\Sscript{M}}$ and $\Sf{can}_{\Sscript{M'}}$ are injective,  ${}_{\Sscript{B}}M$ and ${}_{\Sscript{B'}}M'$ are faithfully flat modules, then 
$$
 {\rm End}^{\Sscript{\mathfrak{H}\tensor{}\mathfrak{H}'}}(M\tensor{}M') \,\,\cong \,\, B\tensor{}B' \,\, \cong\,\,   {\rm End}^{\Sscript{\mathfrak{H}}}(M) \tensor{}{\rm End}^{\Sscript{\mathfrak{H'}}}(M').
$$
\end{enumerate}
\end{lemma}
\begin{proof} $(1)$ is a routine computation. Part $(2)$ uses  part (1) and the result \cite[Theorem 3.10]{Kaoutit/Gomez:2003a}.  
\end{proof}

\subsection{Dualizable comodule whose endomorphism ring is a principal bundle}\label{sse:dpb}
This subsection is of independent interest. We give conditions under which the endomorphism ring (of linear maps)  of a dualizable right comodule is a left principal bundle. 

Let $(A,\cH)$ be a flat Hopf algebroid and $(M, \varrho)$ a dualizable right $\cH$-comodule. Denote by $B:=\CoEnd{\cH}{M}$ its endomorphism ring of $\cH$-colinear maps, and  consider the endomorphism ring of $A$-linear maps $\End{A}{M}$ as right $\cH$-comodule via the isomorphism $M^*\tensor{A}M \cong \End{A}{M}$ together with the following obvious algebra maps 
$\alpha: A \longrightarrow \End{A}{M}$ and $\beta: B \hookrightarrow  \End{A}{M}$. The proof of the following is left to the reader. 
\begin{proposition}
Assume that ${}_{\Sscript{A}}M, {}_{\Sscript{B}}M$ are faithfully flat modules and that the canonical map $\can{M}$ of equation \eqref{Eq:can} is bijective (e.g., when  $(M,\varrho)$ is a small generator in the category of right $\cH$-comodules). Then the triple $(\End{A}{M},\alpha, \beta)$  is a right principal $(B, \cH)$-bundle (where $(B,B)$ is considered as a trivial Hopf algebroid).
\end{proposition}

\section{Geometrically transitive Hopf algebroid: Definition, basic properties and the result}\label{sec:2}
In this section we recall the definition of geometrically transitive Hopf algebroids and prove some of their basic properties.  Most of the results presented here are in fact consequences of those stated in \cite{Bruguieres:1994}. For sake of completeness, we give below  slightly different  elementary proofs of some of these results.

\subsection{Definition and basic properties}\label{ssec:DBP} We start by proving the following result which will help us to well understand the forthcoming definition. 
\begin{proposition}[{\cite[Proposition 6.2, page 5845]{Bruguieres:1994}}]\label{prop:GT} Let $(A,\cH)$ be a flat Hopf algebroid. Assume  that $(A,\cH)$ satisfies the following condition  
\begin{center}{
\begin{enumerate}[({\textit{GT}}1)]
\item $\cH$ is  projective as an $(A\tensor{}A)$-module.
\end{enumerate}}
\end{center}
Then, we have 
\begin{enumerate}[({\textit{GT}1}1)]
\item Every $\cH$-comodule is projective as an $A$-module;
\item $\frcomod{\cH}$ is an abelian category and the functor $\mathscr{U}_{\Sscript{\cH}}: \frcomod{\cH} \to \rmod{A}$ is faithful and exact;
\item Every object in $\rcomod{\cH}$ is a filtrated limit of subobjects in $\frcomod{\cH}$.
\end{enumerate}
\end{proposition}
\begin{proof}
$(GT11)$. Let $M$ be a right $\cH$-comodule. Then, as a right $A$-module $M$ is a direct summand of $M\tensor{A}\cH$. Since ${}_{\Sscript{A}}\cH_{\Sscript{A}}$ is a direct summand of a free $(A\tensor{}A)$-module, $M$ is a direct summand of the right $A$-module  $M\tensor{A}(A\tensor{}A)  \cong M\tensor{}A$. Thus $M_{\Sscript{A}}$ is projective.  The same proof works for left $\cH$-comodules. 

$(GT12)$. 
The category $\frcomod{\cH}$ is additive  with finite product and cokernels. Let us check that $\frcomod{\cH}$ do have kernels. So, assume a morphism $f : N \to M$ in $\frcomod{\cH}$ is given. Then the  kernel ${\rm Ker}(f)$ is a right $\cH$-comodule, since we already know that $\rcomod{\cH}$ is a Grothendieck category. Thus we need to check that the underlying module of this kernel  is a finitely generated $A$-module.  However, this follows from the fact that  $f^{\Sscript{k}}: {\rm Ker}(f) \to N$ splits in $A$-modules, as  we know, by the isomorphism of right $\cH$-comodules $N/{\rm Ker}(f) \cong {\rm Im}(f)$ and condition $(GT11)$, that this quotient is projective as an $A$-modules. The last claim in $(GT12)$   is now clear.

$(GT13)$. Following \cite[\S20.1, \S20.2]{BrzWis:CAC}, since ${}_{\Sscript{\Sf{s}}}\cH$  is by condition $(GT11)$ a projective module, we have that the category of rational left ${}^*\cH$-modules is isomorphic to the category of right $\cH$-comodules, where ${}^*\cH=\hom{A}{{}_{\Sscript{\Sf{s}}}\cH}{A}$ is the left convolution $A$-algebra of $\cH$. Since any submodule of a rational module is also rational,  every rational module is then a filtrated limit of finitely generated submodules.  Therefore, any right $\cH$-comodule is a filtrated limit of  subcomodules in $\frcomod{\cH}$, as any finitely generated rational module is finitely generated as an $A$-module.
\end{proof}

Recall that a (locally small)  $\Bbbk$-linear category $\mathcal{C}$ is said to be \emph{locally of finite type}, if  any object in $\mathcal{C}$ is of finite length and each of the $\Bbbk$-vector spaces of morphisms  $C(c, c')$ is finite  dimensional.  

\begin{definition}\label{def:geo-tran}[Brugui\`eres]\label{def:GT} Let $(A,\cH)$ be a flat Hopf algebroid. We say that  $(A,\cH)$ is a \emph{geometrically transitive Hopf algebroid} (GT for short)  if it satisfies the following conditions:
\begin{center}
\begin{enumerate}[({\textit{GT}}1)]
\item $\cH$ is  projective as an $(A\tensor{}A)$-module.
\item The category $\frcomod{\cH}$ is locally of finite type. 
\item ${\rm End}^{\Sscript{\cH}}(A) \,\cong\, \Bbbk$.
\end{enumerate}
\end{center}
Here ${\rm End}^{\Sscript{\cH}}(A)$ denotes the endomorphisms ring of the right $\cH$-comodule $(A,\Sf{t})$ which is identified with the coinvariant subring $A^{\Sscript{coinv_{\cH}}}=\{ a \in A|\,\, \Sf{t}(a)\,=\, \Sf{s}(a)\}$. We are implicitly assuming that $A\neq 0$ as a  comodule. 
\end{definition}

The subsequent lemma gives others consequences of the  properties stated in Definition \ref{def:geo-tran}, which will be used later on. 
\begin{lemma}\label{lema:1-2}
Let $(A,\cH)$ be a flat Hopf algebroid. 
\begin{enumerate}[(a)]
\item If $(A,\cH)$ satisfies $(GT11)$ and $(GT3)$, then $A$ is a simple (right) $\cH$-comodule.
\item If $(A,\cH)$ satisfies $(GT1)$ and $(GT3)$, then every (right) $\cH$-comodule is faithfully flat as an $A$-module.
\end{enumerate}
\end{lemma}
\begin{proof}
$(a)$. First let us check that, under the assumption $(G11)$,  any subcomodule of $(A,\Sf{t})$ is a direct summand in $\rcomod{\cH}$.    So let $(I,\varrho_{\Sscript{I}})$ be an $\cH$-subcomodule of $(A, \Sf{t})$. Then $A/I$ is an $\cH$-comodule which is finitely generated and projective as an $A$-module, by assumption $(GT11)$. Therefore, $I$ is a direct summand of $A$ as an  $A$-submodule.
Denotes by $\pi :A \to I$ the canonical projection of $A$-modules, and let  $e^2=e$ be an idempotent element in $ A$ such that $I=eA$ and $\pi(a)=ea$, for every $a \in A$. 

Next we show that $\pi$ is a morphism of right $\cH$-comodules, which proves that $(I,\varrho_{\Sscript{I}})$ is a direct summand of $(A,\Sf{t})$. To this end, it suffices to check that $\Sf{s}(e)=\Sf{t}(e)$, since we know that ${\rm End}^{\Sscript{\cH}}(A)=A^{\Sscript{coinv_{\cH}}}=\{ a \in A|\,\, \Sf{t}(a)\,=\, \Sf{s}(a)\}$.  Clearly the coaction of $I$ is entirely defined by the image of $e$, and we can write   $\varrho_{\Sscript{I}}(e)=e \tensor{A}u$, for some element $u \in \cH$,  which satisfies  the following equalities
\begin{equation}\label{Eq:e}
\Sf{s}(e)\,u \,=\, \Sf{t}(e)1_{\Sscript{\cH}}, \quad e\tensor{A}u\tensor{A}u\,=\, e\tensor{A}u_{\Sscript{(1)}}\tensor{A}u_{\Sscript{(2)}} \in I\tensor{A}\cH\tensor{A}\cH,
\end{equation} where the first equation comes from the fact that the inclusion $I \hookrightarrow A$ is a morphism of $\cH$-comodules. 

On the other hand, we know by Lemma \ref{lema:dualizable} that $I$ is a dualizable right $\cH$-comodule. Up to canonical isomorphism, its dual comodule have for the underlying $A$-module, the module $I^*=eA$ with coaction  $\varrho_{\Sscript{I^*}}: eA \to eA \tensor{A} \cH$, sending $ea \mapsto e \tensor{A} \Sf{t}(ea) \mathscr{S}(u)$ given by equation \eqref{Eq:star}. The evaluation map $\Sf{ev}: I^*\tensor{A}I \to A$, $ea\tensor{A} ea' \mapsto eaa'$ of equation \eqref{Eq:evdb}, is then a morphism of right $\cH$-comodules. Therefore, we have the following equality 
\begin{equation}\label{Eq:AB}
1 \tensor{A} \Sf{t}(e) \,=\, e \tensor{A} \Sf{t}(e) \mathscr{S}(u) u \; \in\,  A\tensor{A} {}_{\Sscript{\Sf{s}}}\cH\,\cong \, {}_{\Sscript{\Sf{s}}}\cH. 
\end{equation}
Combining the first equality of equation \eqref{Eq:e} and equation \eqref{Eq:AB}, we get 
$\Sf{t}(e)\mathscr{S}(u) \overset{}{=} \Sf{t}(e)$. Hence $\Sf{s}(e) u = \Sf{s}(e)$, and so $\Sf{s}(e)=\Sf{t}(e)$, by the first equality in \eqref{Eq:e}.

We have then show that any $\cH$-subcomodule of the $\cH$-comodule  $A$ is a direct summand, since by $(GT3)$ the endomorphism ring is a field ${\rm End}^{\Sscript{\cH}}(A) \cong \Bbbk$, we conclude that $A$ is a simple $\cH$-comodule. 

$(b)$. By  Proposition \ref{prop:GT}, we know that $(A,\cH)$ satisfies conditions  $(GT11)$-$(GT13)$. Let us first show that any comodule in $\frcomod{\cH}$ is faithfully flat as an $A$-module. By condition $(GT11)$, we know that any comodule in this subcategory is  finitely generated and projective as $A$-module, so it is flat as an $A$-module.  Moreover, we know from Lemma \ref{lema:dualizable} that the subcategory $\frcomod{\cH}$ consists exactly of dualizable right $\cH$-comodules. Let us then pick a dualizable comodule $M \in \frcomod{\cH}$, and assume that $M \tensor{A}X=0$ for some $A$-module $X$. This  in particular implies that  $\Sf{ev}_{\Sscript{M}}\tensor{A}X\,=\,0$, from which we get that  $A\tensor{A}X\cong X=0$, as $\Sf{ev}_{\Sscript{M}}$ is surjective, since we already know  by item $(a)$ that $A$ is a  simple comodule.   This shows that every object in $\frcomod{\cH}$ is faithfully flat as an $A$-module.

For an arbitrary comodule,   we know by condition  $(GT13)$ stated in Proposition \ref{prop:GT}, that any right $\cH$-comodule is a filtrated limit of subcomodules in $\frcomod{\cH}$.  Therefore, any right $\cH$-comodule is a flat $A$-module. Given now a right $\cH$-comodule $M$ and assume that  $M \tensor{A}X=0$, for some $A$-module $X$. We have that $M={\varinjlim}(M_{\Sscript{i}})$
where  $\{\tau_{\Sscript{ij}}:M_{\Sscript{i}} \hookrightarrow M_j\}_{\Sscript{i \leq j\, \in\, \Lambda}}$  is a filtrated system in $\frcomod{\cH}$ with  structural morphisms $\tau_{\Sscript{ij}}$ which are split morphisms of $A$-modules. This limit is also a filtrated limit of $A$-modules, and so the equality $\varinjlim(M_{\Sscript{i}}\tensor{A}X)\cong M\tensor{A}X =0$ implies  that there exists some $j \in \Lambda$, such that $M_{\Sscript{j}}\tensor{A}X=0$. Hence $X=0$, since $M_{\Sscript{j}}$ is a faithfully flat $A$-module by the previous argumentation.
\end{proof}

\begin{lemma}\label{lema:etamono}
Let $(A,\cH)$ be a flat Hopf algebroid which satisfies conditions $(GT11)$ and $(GT3)$. Then the $\Bbbk$-algebra map $\etaup: A\tensor{}A \to \cH$ is injective. In particular, if $(A,\cH)$ is geometrically transitive, then $\etaup$ is injective. 
\end{lemma}
\begin{proof}
We know from Lemma \ref{lema:1-2}(a) that $A$ is a simple $\cH$-comodule. Therefore, by \cite[Theorem 3.1]{Brzezinski:2005}, the following map 
$$
\cohom{\cH}{A}{\cH} \tensor{} A \longrightarrow \cH, \qquad \Big( f \tensor{\Bbbk}a \longmapsto f(a) \Big)
$$
is a monomorphism, which is, up to the isomorphism $\cohom{\cH}{A}{\cH} \cong A$ derived from the adjunction between the forgetful functor $\mathscr{U}_{\Sscript{\cH}}$ and the functor $ -\tensor{A}{}_{\Sscript{\Sf{s}}}\cH$, is exactly the map $\etaup$. Hence $\etaup$ is injective.  The particular case is immediately obtained form Definition \ref{def:GT} and Proposition \ref{prop:GT}.
\end{proof}

\begin{remark}[Transitive Hopf algebroids]\label{rem:Transitif}
Recall from \cite[D\'efinitions pages 5838, 5850]{Bruguieres:1994} that a Hopf algebroid $(A, \cH)$ with $A\neq 0$, is said to be \emph{transitive} if  it satisfies conditions $(GT12)$, $(GT13)$ and  $(GT2)$, $(GT3)$ from Proposition \ref{ssec:DBP} and Definition \ref{def:GT}, respectively,  and every comodule in $\frcomod{\cH}$ is projective. 
 Thus  the geometrically transitive property implies the transitive one. The converse holds true if the the center of the  division ring of any simple comodule (left or right one) is a separable field extension of $\Bbbk$, that is, if $(A,\cH)$ is a separable Hopf algebroid over $\Bbbk$,  as introduced in \cite[D\'efinition page 5847]{Bruguieres:1994}. Obviously, over a perfect field $\Bbbk$ both notions coincide. 
It is noteworthy to mention that if we consider the associated presheaf $\hH$ of a transitive Hopf algebroid $(A,\cH)$, it is not clear, at least to us, how to express the transitivity of $(A,\cH)$  in terms of certain topology at the level of $\hH$. Lastly, let us mention that in general a Hopf algebroid $(A,\cH)$ is a geometrically transitive if and only if   $(\AL, \HL)$ is transitive, for any filed extension $L$ of $\Bbbk$ (see \cite[Proposition 7.3 page 5851]{Bruguieres:1994}), perhaps this justifies the terminology ``\emph{geometrically transitive}''.  
\end{remark}

We finish this section by characterizing  dualizable objects over GT Hopf algebroids and by making some useful remarks on these algebroids.  

\begin{proposition}\label{prop:dualizable}
Let $(A,\cH)$ be a flat Hopf algebroid. Assume that $(A,\cH)$ satisfies the following  condition:
\begin{enumerate}[ ]
\item $(GT11)'$ Every finitely generated right $\cH$-comodule is projective. 
\end{enumerate} 
Then the full subcategory of $\rcomod{\cH}$ of dualizable objects coincides with $\frcomod{\cH}$. In particular, if $(A,\cH)$ is geometrically transitive, then the category $\frcomod{\cH}$ consists of all dualizable right $\cH$-comodules.
\end{proposition}
\begin{proof}
By Lemma \ref{lema:dualizable}, every dualizable right $\cH$-comodule is finitely generated and projective as an $A$-module. This gives the direct inclusion.  Conversely,  any object in $\frcomod{\cH}$ is, by condition $(GT11)'$ and Lemma \ref{lema:dualizable}, a dualizable right  $\cH$-comodule, form which we obtain the other inclusion. The particular case of GT Hopf algebroids  follows directly from Proposition \ref{prop:GT}.
\end{proof}

\begin{remark}\label{rem:generators}
Let $(A,\cH)$ be a GT Hopf algebroid. Then, by condition \emph{(GT13)} of Proposition \ref{ssec:DBP} and Proposition  \ref{prop:dualizable}, the category of comodules $\rcomod{\cH}$ has a set of small projective generators, which we denote by $\cA$. Therefore, by applying \cite[Theorem 5.7]{K/GT:2004}, we obtain that the canonical map $\Sf{can}: \lL(\oO) \to \cH$ is an isomorphism of Hopf algebroids, where $\oO: \cA \to \mathsf{proj}(A)$ is the forgetful functor to the category of finitely generated and projective $A$-modules,  and  where $\lL(\oO)$ is the Hopf algebroid reconstructed from the pair $(\cA, \oO)$, see \cite{Bruguieres:1994} and also \cite{K/GT:2004} for the explicit description of the underlying $A$-coring of $\lL(\oO)$. 
\end{remark}

\subsection{Characterization by means of weak equivalences}\label{sec:3} 
This subsection contains our main result. We give several  new characterizations of geometrically transitive flat Hopf algebroids.  The most striking  one is  the characterization of these Hopf algebroids  by means of weak equivalences, which can be   seen as the geometric counterpart of the characterization of transitive groupoids, as we have shown in subsection \ref{ssec:TGrpd}, precisely in Proposition \ref{prop:grpd}.

\begin{theorem}\label{thm:A}
Let $(A,\cH)$ be a flat Hopf algebroid over a field $\Bbbk$ and  denote by $\hH$ its associated presheaf of groupoids.  Assume that $\hHo(\Bbbk) \neq \emptyset$. 
Then the following are equivalent:
\begin{enumerate}[(i)]
\item $\etaup: A\tensor{}A \to \cH$ is a  faithfully flat extension;
\item Any two objects of $\mathscr{H}$ are fpqc locally isomorphic (see Definition \ref{def:1});
\item For any extension $\phi: A \to B$, the  extension $\alpha: A \to \cH_{\Sscript{\Sf{t}}}\tensor{A}{}_{\Sscript{\phi}}B$, $a \mapsto \Sf{s}(a)\tensor{A}1_{\Sscript{B}}$ is  faithfully flat;
\item $(A,\cH)$ is geometrically transitive (Definition \ref{def:geo-tran}). 
\end{enumerate} 
\end{theorem}
By \cite[Proposition 5.1]{Kaoutit/Kowalzig:14}, condition $(iii)$ in Theorem \ref{thm:A} is also equivalent to the following ones:
\begin{enumerate}
\item[\emph{(v)}] \emph{ For any extension $\phi: A \to B$, the associated canonical morphism of ${\B \phi}:(A,\cH) \to (B,\cH_{\Sscript{\phi}})$ is a weak equivalence};
\item[\emph{(vi)}] \emph{The trivial principal left $(\cH, \cH_{\Sscript{\phi}})$-bundle $\cH\tensor{A}B$ is a principal bi-bundle.}
\end{enumerate}

Given a GT Hopf algebroid $(A,\cH)$ and  an extension $\phi:A\to B$. Since $\cH_{\Sscript{\phi}}$ is a  flat Hopf algebroid,  the forgetful functor $\rcomod{\cH_{\Sscript{\phi}}} \to \rmod{B}$ is exact. Therefore, condition $(v)$ implies that $B$ is \emph{Landweber exact} over $A$, in the sense that the  functor $\mathscr{U}_{\Sscript{\cH}}(-)\tensor{A}B: \rcomod{\cH} \to \rmod{B}$ is exact, see \cite[Definition 2.1]{HovStr:CALEHT}.

\begin{example}
The following Hopf algebroids $(A,A\tensor{}A)$ and $(A,(A\tensor{}A)[X,X^{-1}])$  described, respectively, in Examples  \ref{exam:HAlgd1} and \ref{exam:HAlgd3}, are clearly geometrically transitive.  This is also the case of $(A,A\tensor{}B\tensor{}A)$ for any Hopf algebra $B$.  On the other hand, if $A$ is a right $B$-comodule algebra whose canonical map  $A\tensor{}A \to A\tensor{}B$ is a faithfully flat extension, then the split Hopf algebroid $(A,A\tensor{}B)$ is obviously geometrically transitive. 

A more elaborade example of GT Hopf algebroid, by using principal bundles over Hopf algebras (i.e., Hopf Galois extensions),  is given in Proposition \ref{prop:GTPB} below. 
\end{example}

Next, we give the proof of Theorem \ref{thm:A}.\smallskip

\emph{The proof of $(i) \Rightarrow (ii)$}.  Let $C$ be an algebra and $x,y$ 
two objects in $A(C)$. Denote by $x \tensor{}y: A\tensor{}A \to C$ the associated algebra map and consider  the obvious algebra map $p: C \to C':= \cH\tensor{A\tensor{}A}C$. By assumption it is clear  that $p$ is a faithfully flat extension. Set the algebra map $g : \cH \to C'$ which sends $u \mapsto u \tensor{A\tensor{}A}1_{\Sscript{C}}$. We then have that $p \circ x = g \circ \Sf{s}$ and $p \circ y = g \circ \Sf{t}$, which shows that $x$ and $y$ are locally isomorphic. 

\emph{The proof of $(ii) \Rightarrow (iii)$}. We claim that under hypothesis $(ii)$ the underlying $A$-module of any (left or right) $\cH$-comodule is faithfully flat. In particular, this implies that the comodule $\cH\tensor{A}B$, with coaction  $\Delta\tensor{A}B$,  is faithfully flat for every $A$-algebra $B$, and this gives us condition $(iii)$. Since there is an isomorphism of categories between right $\cH$-comodules and left $\cH$-comodules, which commutes with forgetful functors, it suffices then to show the above claim for right $\cH$-comodules.

So let us fix a right $\cH$-comodule $M$ and take two objects in different fibres groupoids $x \in A(T)$ and $y \in A(S)$, where $T,S$ are algebras.  We claim that $M\tensor{A}{}_{\Sscript{x}}T$ is faithfully flat $T$-module if and only if $M\tensor{A}{}_{\Sscript{y}}S$ is faithfully flat $S$-module.
Clearly our first claim follows from this one since we know that $A(\Bbbk) \neq \emptyset$ and over a field any module is faithfully flat. 

Let us then check this second claim; we  first assume that $R=T=S$. In this case, we know that any pair of objects $x, y \in A(R)$ are fpqc locally isomorphic, thus there exists a faithfully flat extension $p: R \to R'$ and $g \in \cH(R)$ such that $\td{x}:=p \circ x=g \circ \Sf{s}$ and $\td{y}:=p\circ y = g \circ\Sf{t}$. On the other hand, the map 
$$
M\tensor{A}{}_{\Sscript{\td{x}}} R' \longrightarrow M\tensor{A}{}_{\Sscript{\td{y}}} R', \quad \Big(  m\tensor{A}r' \longmapsto m_{\Sscript{(0)}} \tensor{A} g^{-1}(m_{\Sscript{(1)}})r' \Big)
$$
is clearly an isomorphism of $R'$-modules. Therefore, $M\tensor{A}{}_{\Sscript{\td{x}}} R'$ is a faithfully flat $R'$-module if and only if $M\tensor{A}{}_{\Sscript{\td{y}}} R'$ it is. However, we know that $M\tensor{A}{}_{\Sscript{\td{x}}} R' \cong (M\tensor{A}{}_{\Sscript{{x}}}R) \tensor{R} {}_{\Sscript{p}}R'$ is faithfully flat  $R'$-module if and only if $M\tensor{A}{}_{\Sscript{{x}}}R$ is faithfully flat $R$-module, as $p$ is a faithfully flat extension. The same then holds true interchanging $x$ by $y$. Therefore,  $M\tensor{A}{}_{\Sscript{{x}}}R$ is faithfully flat $R$-module if and only if  $M\tensor{A}{}_{\Sscript{{y}}}R$ so is. 

For the general case, that is, when $T \neq S$ with  $x \in A(T)$ and  $y \in A(S)$, we take $R:=T\tensor{}S$ and consider the canonical faithfully flat extensions $T \to R \leftarrow S$. This leads to the following two objects  $\bara{x}: A \to T \to R$ and $\bara{y}: A \to S \to R$. Since $M\tensor{A}{}_{\Sscript{x}}T$ (resp.  $M\tensor{A}{}_{\Sscript{y}}S$)   is faithfully flat $T$-module (resp. $S$-module)  if and only if $M\tensor{A}{}_{\Sscript{\bara{x}}}R$ (resp. $M\tensor{A}{}_{\Sscript{\bara{y}}}R$) is faithfully flat $R$-module, we have, by the proof of the previous case,  that $M \tensor{A}{}_{\Sscript{x}}T$ is faithfully flat $T$-module if and only if  $M\tensor{A}{}_{\Sscript{y}}S$ is faithfully flat $S$-module, and this finishes the proof of this implication.

\emph{The proof of $(iii) \Rightarrow (iv)$}. Take an object $x$ in $A(\Bbbk)$ and denote by $\Bbbk_{\Sscript{x}}$ the base field endowed with its  $A$-algebra  structure via the algebra map  $x: A\to \Bbbk$. By assumption $A \to \cH\tensor{A}\Bbbk_{\Sscript{x}}$ is a faithfully flat extension. Therefore, by \cite[ Proposition 5.1]{Kaoutit/Kowalzig:14},  we know that the associated base change morphism $\Sf{x}: (A,\cH) \to (\Bbbk_{\Sscript{x}},\cH_{\Sscript{x}})$, where $(\Bbbk_{\Sscript{x}}, \cH_{\Sscript{x}})$ is the Hopf  $\Bbbk$-algebra $\cH_{\Sscript{x}}=\Bbbk_{\Sscript{x}}\tensor{A}\cH\tensor{A}\Bbbk_{\Sscript{x}}$, is actually a weak equivalence. This means that the induced functor $\Sf{x}_*:=\mathscr{U}_{\Sscript{\cH}}(-)\tensor{A}\Bbbk_{\Sscript{x}}: \rcomod{\cH} \to \rcomod{\cH_{\Sscript{x}}}$ is a symmetric monoidal equivalence of categories, and thus transforms, up to natural isomorphisms,  dualizable $\cH$-comodules into dualizable $\cH_{\Sscript{x}}$-comodules. Similar property hods true for its inverse functor.  In particular, taking an object $M \in\frcomod{\cH}$, it is clear that $\Sf{x}_{*}(M)=M\tensor{A}\Bbbk_{\Sscript{x}}$ is  finite dimensional $\Bbbk$-vector space and so a dualizable right $\cH_{\Sscript{x}}$-comodue, see for instance Lemma \ref{lema:dualizable}. Therefore, $M$ should be a dualizable right $\cH$-comodule.  The converse is obvious and then the full subcategory $\frcomod{\cH}$ coincides with the full subcategory  of dualizable right $\cH$-comodules, form which we have that $\frcomod{\cH}$ and $\frcomod{\cH_{\Sscript{x}}}$ are equivalent $\Bbbk$-linear categories. Hence    $\frcomod{\cH}$ is locally of finite type, and the endomorphism ring  ${\rm End}^{\Sscript{\cH}}(A) \cong \Bbbk$. This shows simultaneously  conditions $(GT2)$ and $(GT3)$.   

To check condition $(GT1)$ we use the morphism between the tensor product Hopf algebroids, that is,  $\Sf{x}^{\Sscript{o}}\tensor{}\Sf{x}: (A\tensor{}A, \cH^{\Sscript{o}}\tensor{}\cH) \to (\Bbbk_{\Sscript{x}}\tensor{}\Bbbk_{\Sscript{x}}\cong \Bbbk, \cH_{\Sscript{x}}\tensor{}\cH_{\Sscript{x}})$. As we have seen in subsection \ref{ssec:W}, this is also a weak equivalence. Thus the category of right $(\cH^{\Sscript{o}}\tensor{}\cH)$-comodules is equivalent, as a symmetric monoidal category,  to the category of right comodules over the Hopf $\Bbbk$-algebra $\cH_{\Sscript{x}}\tensor{}\cH_{\Sscript{x}}$, which as in the case of $\Sf{x}$ also implies that $\frcomod{\cH^{\Sscript{o}}\tensor{}\cH}$ and $\frcomod{\cH_{\Sscript{x}}\tensor{}\cH_{\Sscript{x}}}$ are equivalent. Therefore,  from one hand, we have  by the same reasoning as above  that any comodule in $\frcomod{\cH^{\Sscript{o}}\tensor{}\cH}$ is  projective as an $(A\tensor{}A)$-module since it is a dualizable comodule.  On the other hand,  we have that  every right $(\cH^{\Sscript{o}}\tensor{}\cH)$-comodule  is a filtrated inductive  limit of objects in $\frcomod{\cH^{\Sscript{o}}\tensor{}\cH}$ since  right  $(\cH_{\Sscript{x}}\tensor{}\cH_{\Sscript{x}})$-comodules satisfies the same property with respect to finite-dimensional right comodules $\frcomod{\cH_{\Sscript{x}}\tensor{}\cH_{\Sscript{x}}}$. Now, by apply \cite[Proposition 5.1(ii)]{Bruguieres:1994} to the $(A\tensor{}A)$-coring $\cH^{\Sscript{o}}\tensor{}\cH$, we then conclude that  every right $(\cH^{\Sscript{o}}\tensor{}\cH)$-comodule is projective as an $(A\tensor{}A)$-module. Thus, $\cH$ is projective as an  $(A\tensor{}A)$-module, which shows condition $(GT1)$.

\emph{The proof of $(iv) \Rightarrow (i)$}. Set $B:=A\tensor{}A$ and $\cK:=\cH^{\Sscript{o}}\tensor{}\cH$. We know that $(B,\cK)$ is a flat Hopf algebroid. Since $\cH$ is projective as $(A\tensor{}A)$-module, we have that $\cK$ is projective as $(B\tensor{}B)$-module. Now, since the map $\etaup$ is injective by Lemma \ref{lema:etamono}, we can apply Lemma \ref{lema:MN} by taking $M=A$ as right $\cH$-comodule and $M'=A$ as right $\cH^{\Sscript{o}}$-comodule, to obtain the following chain of isomorphism
$$
{\rm End}^{\Sscript{\cK}}(B)\,\,=\,\,  {\rm End}^{\Sscript{\cH^{\Sscript{o}}\tensor{}\cH}}(A\tensor{}A)\,\,\cong\,\, {\rm End}^{\Sscript{\cH}^{\Sscript{o}}}(A)\tensor{} {\rm End}^{\Sscript{\cH}}(A)\,\,\cong\,\, \Bbbk \tensor{}\Bbbk \,\, \cong \,\, \Bbbk. 
$$
This means that the Hopf algebroid $(B,\cK)$ satisfies the conditions of Lemma  \ref{lema:1-2}(b). Therefore, any right $\cK$-comodule is faithfully flat as a $B$-module, henceforth, $\cH$ is a faithfully flat  $(A\tensor{}A)$-module. This finishes the proof of Theorem \ref{thm:A}.

\begin{remark}\label{rem:Lk}
Let $(A,\cH)$ be a flat Hopf algebroid  over $\Bbbk$ with $A \neq 0$ and  $A(\Bbbk) =\emptyset$. Then the same proof of the implication $(ii) \Rightarrow (iii)$ in Theorem \ref{thm:A}, works for $(A, \cH)$ by using any field extension $L$ of $\Bbbk$ such that $A(L) \neq \emptyset$. 
Assume now that  $(A,\cH)$ satisfies condition $(i)$ of Theorem  \ref{thm:A} and take a field extension $L$ such that $A(L) \neq \emptyset$. Then  $(\AL, \HL)$ also satisfies this condition and by Theorem  \ref{thm:A} we have that $(\AL,\HL)$ is a GT Hopf algebroid, as we know that $\AL(L) =\Algl(\AL, L) \neq \emptyset$. Furthermore, if $L$ is a perfect field, then by applying \cite[Th\'eor$\mathrm{\grave{e}}$me 6.1, page 5845]{Bruguieres:1994} we can show that  $(A,\cH)$ is a GT Hopf algebroid as well. 
Summing up, given a Hopf algebroid $(A,\cH)$ as above, if its satisfies condition $(i)$ of Theorem  \ref{thm:A}   and $\Bbbk$ admits a perfect extension $L$ such that $A(L) \neq \emptyset$, then $(A,\cH)$ satisfies all the other conditions of this Theorem. 
\end{remark}

\section{More properties of geometrically transitive Hopf algebroids}\label{sec:more}
In this section we give more properties of GT Hopf algebroids.  First we set up an analogous property of transitive groupoids with respect to the conjugacy of their isotropy groups.  To this end we introduce here perhaps a  known notion  of \emph{isotropy Hopf algebra}. This is the affine group scheme which represents the presheaf of groups defined by the isotropy group at each fibre. Next we show that any  two isotropy Hopf algebras are weakly equivalent. The notion of conjugacy between two isotropy Hopf algebras, is not at all obvious, and the $2$-category of flat Hopf algebroids is employed in order to make it clearer. In this direction we show that two isotropy Hopf algebras are conjugated if and only if the character groupoid is transitive, and both conditions are fulfilled is the case of GT Hopf algebroids. 
Lastly, we give an elementary proof of the fact that any dualizable  comodule is locally free of constant rank, which in some sense bear out the same property enjoyed by finite dimensional $\Bbbk$-representations of a given transitive groupoid.  The case when the character groupoid of a GT Hopf algebroid is an empty groupoid, is also analyzed.

\subsection{The isotropy Hopf algebras are weakly equivalent}\label{ssec:isotropy} 
Let $(A,\cH)$ be a flat Hopf algebroid and $\hH$ its associated presheaf  of groupoids. Assume as before that the base algebra satisfies $A\neq 0$ and $A(\Bbbk)\neq \emptyset$, and consider $\mathscr{H}(\Bbbk)$ the character groupoid  of $(A,\cH)$, see Definition \ref{def:characters}. As before,  for each object $x \in A(\Bbbk)$, we denote by $\Bbbk_{\Sscript{x}}$ the $A$-algebra $\Bbbk$ via the extension $x$, and consider  the associated Hopf $\Bbbk$-algebra of a base ring extension (given  by the $\Bbbk$-algebra map $x:A\to \Bbbk_{\Sscript{x}}$), that is,  $\cH_{\Sscript{x}}:=\Bbbk_{\Sscript{x}} \tensor{A} \cH \tensor{A} \Bbbk_{\Sscript{x}}$.

\begin{definition}\label{def:isotropy}
Given an object $x \in A(\Bbbk)$. The  Hopf algebra $(\Bbbk_{\Sscript{x}},\cH_{\Sscript{x}})$  is called  \emph{the isotropy Hopf algebra of $(A,\cH)$ at the point $x$}.  
\end{definition}
It noteworthy to mention that the associated affine $\Bbbk$-group of $(\Bbbk_{\Sscript{x}},\cH_{\Sscript{x}})$ coincides with the one called  \emph{groupe d'inertie de $x$ relativement \`a $\hH$} as referred to in \cite[III, \S 2, n$^{\text{o}}$ 2; page 303]{DemGab:GATIGAGGC}.

The terminology used in Definition \ref{def:isotropy}  is, in relation with groupoids, justified  by the following lemma. 
Fix an object $x \in A(\Bbbk)$, we denoted by $1_{\Sscript{x}}$ the unit element of the $A$-algebra $\Bbbk_{\Sscript{x}}$.  Take $C$ to be an  algebra with unit map $1_{\Sscript{C}}: \Bbbk \to C$. Composing with $x$, we have then an object $x^*(1_{\Sscript{C}})= 1_{\Sscript{C}} \circ x \in A(C)$. Let us denote by $\mathscr{G}^{\Sscript{x}}(C):= \mathscr{H}(C)^{\Sscript{x^*(1_{\Sscript{C}})}}$
 the isotropy group of the object $x^*(1_{\Sscript{C}})$ in the groupoid $\mathscr{H}(C)$, see equation \eqref{Eq:isotropy}. 
This construction is clearly   funtorial and so leads to a presheaf of groups $\mathscr{G}^{\Sscript{x}}: \Algk \to \Sf{Grps}$, $C \to \mathscr{G}^{\Sscript{x}}(C)$.
\begin{lemma}\label{lem:iso}
For any $x \in A(\Bbbk)$, the presheaf of groups $\mathscr{G}^{\Sscript{x}}$ is affine, and up to a natural isomorphism, is  represented by the Hopf $\Bbbk$-algebra $\cH_{\Sscript{x}}$.
\end{lemma}
\begin{proof}
Given an element $g$ in the group $\mathscr{G}^{\Sscript{x}}(C)$, that is, an algebra map $g: \cH \to C$ such that $g \circ \Sf{t}= g \circ \Sf{s}= x^*(1{\Sscript{C}})$, we can define the following algebra map:
$$
\kappaup_{\Sscript{C}}(g): \cH_{\Sscript{x}} \longrightarrow C,\quad \big( k1_{\Sscript{x}}\tensor{A}u\tensor{A}k'1_{\Sscript{x}} \longmapsto kk'g(u)\big),
$$
which is clearly functorial in $C$. This leads  to a  natural transformation $\kappaup_{-}: \gG^{\Sscript{x}}(-) \longrightarrow \rm{Alg}_{\Sscript{\Bbbk}}(\cH_{\Sscript{x}}, -)$.

Conversely, to any algebra map $h: \cH_{\Sscript{x}} \to C$, one associate the algebra map 
$$
\nuup_{\Sscript{C}}(h):= h \circ \tau_{\Sscript{x}}:  \cH \longrightarrow \cH_{\Sscript{x}} \longrightarrow C, 
$$
where $\tau_{\Sscript{x}}:\cH \to \cH_{\Sscript{x}}$ sends $u \mapsto 1_{\Sscript{x}} \tensor{A}u \tensor{A}1_{\Sscript{x}}$. This construction  is also functorial in $C$, which  defines a natural transformation $\nuup_{-}: \rm{Alg}_{\Sscript{\Bbbk}}(\cH_{\Sscript{x}},-) \longrightarrow \gG^{\Sscript{x}}(-)$.
It is not difficult now to check that both natural transformations $\kappaup$ and $\nuup$,  are mutually inverse. 
\end{proof}

Recall that for groupoids the transitivity property is  interpreted by means of   conjugation between theirs isotropy groups, which means that any two of these groups are isomorphic. 
Next we show how this last property is reflected  at the level of the isotropy Hopf algebras. The conjugacy of the isotropy Hopf algebras, in relation with the transitivity of the character groupoid,  will be considered in the next subsection. 

\begin{proposition}\label{prop:weak-isotropy}
Let $(A,\cH)$ be a flat Hopf algebroid with $A\neq 0$ and $A(\Bbbk) \neq \emptyset$. Assume that $(A,\cH)$ is  geometrically transitive. Then any two isotropy Hopf algebras are weakly equivalent.
\end{proposition}
\begin{proof}
Take two objects $x, y \in A(\Bbbk)$ and consider as before the following diagram
\begin{equation}\label{Eq:triangle}
\xymatrix@R=7pt{ (\Bbbk_{\Sscript{x}},\cH_{\Sscript{x}}) & & (\Bbbk_{\Sscript{y}},\cH_{\Sscript{y}}) \\ & \ar@{->}^-{\Sf{x}}[lu] (A,\cH) \ar@{->}_-{\Sf{y}}[ru] &}
\end{equation}
of Hopf algebroids. By Theorem \ref{thm:A}, both $\Sf{x}$ and $\Sf{y}$ are weak equivalences, in particular, the Hopf algebras $(\Bbbk_{\Sscript{x}},\cH_{\Sscript{x}})$ and $(\Bbbk_{\Sscript{y}},\cH_{\Sscript{y}})$ are Morita equivalent, in the sense that their categories of comodules are equivalent as symmetric monoidal $\Bbbk$-linear categories. Therefore, $(\Bbbk_{\Sscript{x}},\cH_{\Sscript{x}})$ and $(\Bbbk_{\Sscript{y}},\cH_{\Sscript{y}})$ are weakly equivalent by applying \cite[Theorem A]{Kaoutit/Kowalzig:14}.
\end{proof}

\begin{remark}\label{rem:Schu}
In the terminology of \cite[Definition 3.2.3]{Schau:HGABGE}, the Hopf algebras $(\Bbbk_{\Sscript{x}},\cH_{\Sscript{x}})$ and $(\Bbbk_{\Sscript{y}},\cH_{\Sscript{y}})$ are said to be \emph{monoidally Morita-Takeuchi} equivalent. By applying \cite[Corollary 3.2.3]{Schau:HGABGE}, there is a Hopf bi-Galois object, or a principal bi-bundle as in subsection \ref{ssec:W}, connecting $\cH_{\Sscript{x}}$ and $\cH_{\Sscript{y}}$ (notice here that the side on comodules  is not relevant since the Hopf algebras are commutative). 
\end{remark}

Next, we compute explicitly, by using results from \cite{Kaoutit/Kowalzig:14}, the principal bi-bundle connecting $(\Bbbk_{\Sscript{x}},\cH_{\Sscript{x}})$ and $(\Bbbk_{\Sscript{y}},\cH_{\Sscript{y}})$, as was mentioned in the previous Remark.  Following \cite{Kaoutit/Kowalzig:14}, any two weakly equivalent flat Hopf algebroids are connected by a two-stage zig-zag of weak equivalences, and this is the case for the previous Hopf algebras. That is, in the situation of Proposition \ref{prop:weak-isotropy},  diagram \eqref{Eq:triangle} can be completed to a square by considering the two-sided translation Hopf algebroid built up by  using the principal bibundle connecting $(\Bbbk_{\Sscript{x}},\cH_{\Sscript{x}})$ and $(\Bbbk_{\Sscript{y}},\cH_{\Sscript{y}})$, see subsection \ref{ssec:W}. In more specific way,   we have the two trivial principal bibundles $P_{\Sscript{x}}:=\cH\tensor{A}\Bbbk_{\Sscript{x}}$ and $P_{\Sscript{y}}:=\cH\tensor{A}\Bbbk_{\Sscript{y}}$ which correspond, respectively,  to the weak equivalences $\Sf{x}$ and $\Sf{y}$. Notice that $P_{\Sscript {x}}$ is an $(\cH,\cH_{\Sscript{x}})$-bicomodule algebra with algebra maps 
\begin{equation}\label{Eq:alphaBetax}
\alpha_{\Sscript{x}}: A \to P_{\Sscript{x}},\quad \big(a \mapsto \Sf{s}(a)\tensor{A}1\big)\; \text{ and }\;\beta_{\Sscript{x}}: \Bbbk_{\Sscript{x}} \to P_{\Sscript{x}}, \quad\big(k \mapsto 1_{\Sscript{\cH}}\tensor{A}k1_{\Sscript{x}}\big).
\end{equation}
Similar notations are applied to  the 
$(\cH,\cH_{\Sscript{y}})$-bicomodule algebra $P_{\Sscript {y}}$.
The cotensor product of these two bibundles $P_{\Sscript{x}}{}^{\Sscript{co}}\cotensor{\cH}P_{\Sscript{y}}$ is again a principal $(\cH_{\Sscript{x}},\cH_{\Sscript{y}})$-bibundle (recall here that $P_{\Sscript{x}}{}^{\Sscript{co}}$ is the opposite bundle of $P_{\Sscript{x}}$). The algebra maps defining this structure are $\td{\beta_{\Sscript{x}}}: \Bbbk_{\Sscript{x}} \longrightarrow  P_{\Sscript{x}}{}^{\Sscript{co}}\cotensor{\cH}P_{\Sscript{y}} \longleftarrow \Bbbk_{\Sscript{y}} : \td{\beta_{\Sscript{y}}}$, given by 
$$
\td{\beta_{\Sscript{x}}}(k) \,=\, \beta_{\Sscript{x}}(k ) \cotensor{\cH} 1_{\Sscript{P_{\Sscript{y}}}} ,\qquad  \td{\beta_{\Sscript{y}}}(k) \,=\,  1_{\Sscript{P_{\Sscript{x}}}}  \cotensor{\cH} \beta_{\Sscript{y}}(k ),
$$
where the notation is the obvious one. 
The associated two-sided translation Hopf algebroid is described as follows. 
First we observe the following general fact in Hopf algebroids with source equal to the target, i.e., Hopf algebras over commutative algebras.
\begin{lemma}\label{lema:Comopesaesto}
Let $(R,L)$ and $(R',L')$ be two commutative Hopf algebras, and assume that there is  a diagram of Hopf algebroids:
$$
\xymatrix@R=8pt{ (R,L) & & (R',L')  \\ & \ar@{->}^-{\omega}[lu] (A,\cH) \ar@{->}_-{\omega'}[ru] &}
$$
Then the pair $\big(R\tensor{A}\cH\tensor{A}R',L\tensor{A}\cH\tensor{A}L'\big)$ of algebras, admits a structure of Hopf algebroid with maps:
\begin{enumerate}[$\bullet$]
\item the source and target:
$$
\Sf{s}(r\tensor{A}u\tensor{A}r')\,:=\, r1_{\scriptscriptstyle{L}} \tensor{A}u\tensor{A}r'1_{\scriptscriptstyle{L'}}, \quad \Sf{t}(r\tensor{A}u\tensor{A}r')\,:=\, r\omega(\mathscr{S}(u_{\Sscript{(1)}})) \tensor{A}u_{\Sscript{(2)}}\tensor{A}\omega'(u_{\Sscript{(3)}})r'; 
$$
\item comultiplication and counit:
$$
\Delta(l\tensor{A}u\tensor{A}l') :=  \big( l_{\Sscript{(1)}}\tensor{A}u\tensor{A}l'_{\Sscript{(1)}}\big)\tensor{C}\big( l_{\Sscript{(2)}}\tensor{A}1_{\scriptscriptstyle{\cH}}\tensor{A}l'_{\Sscript{(2)}}\big) ,\;\, \varepsilon(l\tensor{A}u\tensor{A}l'):=  \varepsilon_{\Sscript{L}}(l)\tensor{A}u\tensor{A}  \varepsilon_{\Sscript{L'}}(l');
$$
\item the antipode:
$$
\mathscr{S}(l\tensor{A}u\tensor{A}l')\,:=\, \mathcal{S}_{\Sscript{L}}\big(l\,\omega(u_{\Sscript{(1)}})\big) \tensor{A}u_{\Sscript{(2)}}\tensor{A}\omega'(u_{\Sscript{(3)}})
\mathcal{S}_{\Sscript{L'}}(l').
$$
\end{enumerate}
\end{lemma}
\begin{proof}
These are routine computations.
\end{proof}

Now we come back to the situation of Proposition \ref{prop:weak-isotropy}. Consider the following algebras:
$$P_{\Sscript{x,\,y}} := \Bbbk_{\Sscript{x}}\tensor{A} \cH \tensor{A}\Bbbk_{\Sscript{y}},\qquad \cH_{\Sscript{x,\, y}}:= \cH_{\Sscript{x}}\tensor{A}\cH\tensor{A}\cH_{\Sscript{y}},$$
with the structure of Hopf algebroid, as in Lemma \ref{lema:Comopesaesto}. 
Consider then the following obvious algebra maps
$$
\omega_{\Sscript{x}}:  \cH_{\Sscript{x}} \longrightarrow \cH_{\Sscript{x,\, y}},\,\, \Big( k1_{\Sscript{x}}\tensor{A}u \tensor{A}k' 1_{\Sscript{x}} \longmapsto k1_{\Sscript{\cH_{\Sscript{x}}}}\tensor{A}u \tensor{A}k' 1_{\Sscript{\cH_{\Sscript{y}}}} \Big);   
$$
and 
$$
\omega_{\Sscript{y}}:\cH_{\Sscript{y}} \longrightarrow \cH_{\Sscript{x,y}},\,\, \Big(k1_{\Sscript{y}}\tensor{A}u \tensor{A}k' 1_{\Sscript{y}} \longmapsto k1_{\Sscript{\cH_{\Sscript{x}}}} \tensor{A}u \tensor{A}k' 1_{\Sscript{\cH_{\Sscript{y}}}}  \Big).
$$ 
\begin{proposition}\label{prop:tst}
Let $(A,\cH)$ be as in Proposition \ref{prop:weak-isotropy},  consider $x, y \in A(\Bbbk)$ and their associated isotropy Hopf algebras $(\Bbbk_{\Sscript{x}},\cH_{\Sscript{x}})$ and $(\Bbbk_{\Sscript{y}},\cH_{\Sscript{y}})$. Assume that $(A,\cH)$ is geometrically transitive.  Then there is an isomorphism 
$$
\Big(P_{\Sscript{x}}{}^{\Sscript{co}}\cotensor{\cH}P_{\Sscript{y}}, \, \cH_{\Sscript{x}}\lJoin (P_{\Sscript{x}}{}^{\Sscript{co}}\cotensor{\cH}P_{\Sscript{y}}) \rJoin \cH_{\Sscript{y}} \Big)\,\,\cong\,\, \big(P_{\Sscript{x,\,y}}, \cH_{\Sscript{x,\, y}}\big)
$$ 
of Hopf algebroids with the following  diagram $$
\xymatrix@R=7pt{  & (P_{\Sscript{x,\,y}}, \cH_{\Sscript{x,\, y}}) &  \\ (\Bbbk_{\Sscript{x}},\cH_{\Sscript{x}}) \ar@{->}^-{{\bf \omega}_{\Sscript{x}}}[ru] & & (\Bbbk_{\Sscript{y}},\cH_{\Sscript{y}}) \ar@{->}_-{{\bf \omega}_{\Sscript{y}}}[lu] \\ & \ar@{->}^-{\Sf{x}}[lu] (A,\cH) \ar@{->}_-{\Sf{y}}[ru] &}
$$
of weak equivalences. 
\end{proposition}
\begin{proof}
The stated isomorphism follows directly by comparing the structure of the two-sided translation Hopf algebroid, as given in subsection \ref{ssec:W}, with that of $(P_{\Sscript{x,\, y}}, \cH_{\Sscript{x,\, y}})$ given in Lemma \ref{lema:Comopesaesto}. By Proposition \ref{prop:weak-isotropy},  we know that $\Sf{x}$ and $\Sf{y}$ are weak equivalences. Therefore,  ${\bf \omega}_{\Sscript {x}}$ and ${\bf \omega}_{\Sscript {y}}$ are weak equivalences by applying \cite[Proposition 6.3]{Kaoutit/Kowalzig:14} in conjunction with the previous isomorphism of Hopf algebroids.  
\end{proof}

\begin{remark}\label{rem:w}
The diagram stated in Proposition \ref{prop:tst}, is not necessarily strictly  commutative; however, it is commutative  up to a $2$-isomorphism in the $2$-category of  flat Hopf algebroids described in subsection \ref{ssec:H}. Precisely, one shows by applying \cite[Lemma 6.11]{Kaoutit/Kowalzig:14}  that there is a $2$-isomorphism ${\bf \omega}_{\Sscript {x}} \circ \Sf{x}\,\cong \, {\bf \omega}_{\Sscript {x}} \circ \Sf{y}$.
\end{remark}

\subsection{The transitivity of the character groupoid}\label{ssec:CH}
Let $(A,\cH)$ be a flat  Hopf algebroid  as in the previous subsection and consider its character groupoid $\mathscr{H}(\Bbbk)=(\cH(\Bbbk),A(\Bbbk))$.  We have seen in Theorem \ref{thm:A} that $(A,\cH)$ is geometrically transitive if and only if the attached presheaf of groupoids $\mathscr{H}$ is locally transitive, that is, satisfies condition $(ii)$ of that theorem. The aim of this subsection is to characterize the transitivity of the  groupoid $\mathscr{H}(\Bbbk)$, by means of  the conjugation between the isotropy Hopf algebras. First we introduce the notion of conjugacy. 

\begin{definition}
Let $x, y$ be two objects in $\mathscr{H}(\Bbbk)$. We say that \emph{the isotropy Hopf algebras $(\Bbbk_{\Sscript{x}},\cH_{\Sscript{x}})$ and $(\Bbbk_{\Sscript{y}},\cH_{\Sscript{y}})$ are conjugated}, provided there is an isomorphism of Hopf algebras $\Sf{g}:(\Bbbk_{\Sscript{x}},\cH_{\Sscript{x}}) \to (\Bbbk_{\Sscript{y}},\cH_{\Sscript{y}})$ such that the following diagram 
$$
\xymatrix@R=8pt{ (\Bbbk_{\Sscript{x}},\cH_{\Sscript{x}}) \ar@{->}^-{\Sf{g}}[rr] & & (\Bbbk_{\Sscript{y}},\cH_{\Sscript{y}}) \\ & \ar@{->}^-{\Sf{x}}[lu] (A,\cH) \ar@{->}_-{\Sf{y}}[ru] &}
$$
is commutative up to a $2$-isomorphism, where Hopf $\Bbbk$-algebras are considered as 0-cells in the 2-category of flat Hopf algebroids  described in subsection  \ref{ssec:H}.
\end{definition}
As in \cite[\S 6.4]{Kaoutit/Kowalzig:14}, this means that there is an algebra map $g: \cH\to \Bbbk$ such that  
\begin{equation}\label{Eq:Z}
g \circ \Sf{s}= x,\quad g \circ \Sf{t}=y, \; \text{ and }\;\; u_{\Sscript{(1)}}{}^{\Sscript{-}}\tensor{A} u_{\Sscript{(1)}}{}^{\Sscript{0}}\tensor{A} u_{\Sscript{(1)}}{}^{\Sscript{+}} g(u_{\Sscript{(2)}})\,=\, g(u_{\Sscript{(1)}})\tensor{A}u_{\Sscript{(2)}}\tensor{A}1_{\Sscript{y}}\,\in \cH_{\Sscript{y}}
\end{equation}
where, by denoting  the Hopf algebroids map $\Sf{z}:=\Sf{g} \circ \Sf{x}: (A,\cH) \to (\Bbbk_{\Sscript{x}}, \cH_{\Sscript{x}})$, we have 
$$
\Sf{z}_{\Sscript{0}}=x \quad \text{ and } \quad \Sf{z}_{\Sscript{1}}(u)=\Sf{g}(1_{\Sscript{x}} \tensor{A}u \tensor{A}1_{\Sscript{x}}):= u^{\Sscript{-}}\tensor{A}u^{\Sscript{0}} \tensor{A }u^{\Sscript{+}} \;(\text{summation understood}).
$$

\begin{proposition}\label{prop:conjugation}
Let $(A,\cH)$ be a flat Hopf algebroid with $A\neq 0$ and $A(\Bbbk) \neq \emptyset$. Assume that $(A,\cH)$ is  geometrically transitive. Then the following are equivalent: 
\begin{enumerate}[(i)]
\item the character groupoid $\mathscr{H}(\Bbbk)$ is transitive;
\item for any two objects $x, y$ in $\mathscr{H}(\Bbbk)$, the algebras  $\cH\tensor{A}\Bbbk_{\Sscript{x}}$ and $\cH\tensor{A}\Bbbk_{\Sscript{y}}$ are isomorphic as left $\cH$-comodules algebras;
\item any two isotropy Hopf algebras are conjugated.
\end{enumerate}
Furthermore, under the same assumptions, condition $(i)$ is always fulfilled. 
\end{proposition}
\begin{proof}
We first check the equivalences between these conditions. So, let $x \in A(\Bbbk)=\hH_{\Sscript{0}}(\Bbbk)$ and denote as before by $P_{\Sscript{x}}:=\cH\tensor{A}\Bbbk_{\Sscript{x}}$ the stated left $\cH$-comodule algebra.

$(i) \Rightarrow (ii)$. Given $x, y \in A(\Bbbk)$, by assumption there is an algebra map $h:\cH \to \Bbbk$ such that $h \circ \Sf{s}= x$ and $h \circ \Sf{t}=y$. So we can define the following map 
$$
F: P_{\Sscript{x}} \longrightarrow  P_{\Sscript{y}}, \qquad \Big( u\tensor{A}k 1_{\Sscript{x}} \longmapsto u_{\Sscript{(1)}}\tensor{A}h\big(\mathscr{S}(u_{\Sscript{(2)}})\big)k 1_{\Sscript{y}}\Big).
$$ 
Clearly $F$ is an $A$-algebra map, and so it is  left $A$-linear. The fact that $F$ is left $\cH$-colinear is also clear, and this shows condition  $(ii)$, since $F$ is obviously bijective.

$(ii) \Rightarrow (iii)$. Assume for a given $x, y \in A(\Bbbk)$, there is  a left $\cH$-comodule algebra isomorphism $F: P_{\Sscript{x}} \to  P_{\Sscript{y}}$. For any $u \in \cH$, we denote by  $F(u\tensor{A}1_{\Sscript{x}})=u^{\Sscript{-}}\tensor{A}u^{\Sscript{+}} $ (summation understood). Consider the $\Bbbk$-linear map $g: \cH \to \Bbbk$ which sends $u \mapsto y\big(\varepsilon(u^{\Sscript{-}})\big)u^{\Sscript{+}}$. This is a $\Bbbk$-algebra map since $F$ it is so. For any $a \in A$, we have 
$$
g\big(\Sf{s}(a)\big)\,=\, y\big(\varepsilon(\Sf{s}(a)^{\Sscript{-}})\big)\Sf{s}(a)^{\Sscript{+}}\,=\, y\big(\varepsilon(\Sf{s}(a))\big) 1\,=\, y(a)
$$
and 
$$
g\big(\Sf{t}(a)\big)\,=\, y\big(\varepsilon(\Sf{t}(a)^{\Sscript{-}})\big)\Sf{t}(a)^{\Sscript{+}}\,=\, y\big(\varepsilon(1_{\Sscript{\cH}})\big) x(a) 1\,=\, x(a),
$$
as $F$ is $\Bbbk$-linear. Define the map
$$
\Sf{g}: (\Bbbk_{\Sscript{x}}, \cH_{\Sscript{x}}) \longrightarrow  (\Bbbk_{\Sscript{y}}, \cH_{\Sscript{y}}),\quad  \Big( (k1_{\Sscript{x}}, 1_{\Sscript{x}}\tensor{A}u\tensor{A}1_{\Sscript{x}}) \longmapsto (k1_{\Sscript{y}}, g(\mathscr{S}(u_{\Sscript{(1)}}))1_{\Sscript{y}}\tensor{A}u_{\Sscript{(2)}}\tensor{A} g(u_{\Sscript{(3)}})1_{\Sscript{y}})\Big).
$$
By  using the characterization given in Lemma \ref{lem:iso}, or a direct computation, one can shows that this map is an isomorphism of Hopf algebras.  Furthermore, it is easily seen that  the pair $(g,\Sf{g})$ satisfies the equalities of equation \eqref{Eq:Z}. Thus,  $(\Bbbk_{\Sscript{x}}, \cH_{\Sscript{x}})$ and $(\Bbbk_{\Sscript{y}}, \cH_{\Sscript{y}})$ are conjugated, which means condition $(iii)$. 

$(iii) \Rightarrow (i)$. This implication follows immediately from equation \eqref{Eq:Z}. 

Let us check that condition $(i)$ is fulfilled under assumption. 
For a given element $u \in \cH$ there exists, by the isomorphism of Remark \ref{rem:generators},  a finite family $M_{1},\cdots,M_{k}$  of dualizable right $\cH$-comodules and finite set of elements $\{(p^{l},\varphi^{l})\}_{1\leq\, l\,\leq k}$, $p^{l} \in M_{l}$ and $\varphi^{l} \in M_{l}^{*}$, such that $u$ is uniquely written as $u = \sum_{l} \Sf{s}\big(\varphi^{l}(p^{l}_{\Sscript{(0)}})\big) p^{l}_{\Sscript{(1)}} $ (see \cite[Section 4]{K/GT:2004} for more details on the map $\Sf{can}$ quoted in Remark \ref{rem:generators}). Given now two objects $x, y \in A(\Bbbk)$, we define $g :\cH \to \Bbbk$ by $g(u):= y\Big( \Sf{s}\big(\varphi^{l}(p^{l}_{\Sscript{(0)}})\big)\Big) x\Big( \varepsilon(p^{l}_{\Sscript{(1)}})\Big)$. It turns out that $g$ is a well defined algebra map, which satisfies $g \circ \Sf{s}= y$ and $g \circ \Sf{t}= x$. This shows that $\hH(\Bbbk)$ is transitive and finishes the proof.
\end{proof}

\subsection{GT Hopf algebroids and principal bundles over Hopf algebras}\label{ssec:GTHpbH}
Parallel to subsection \ref{ssec:GPS} we study here the relationship between GT Hopf algebroids and principal bundles over Hopf algebras (i.e.,  commutative Hopf Galois extensions  \cite[\S 8]{Montgomery:1993}, \cite{Schau:HGABGE}, or $\Bbbk$-\emph{torsor}  as in  \cite{Giraud:1971} and \cite{DemGab:GATIGAGGC}). This is a restricted notion of principal bundle, as defined in subsection \ref{ssec:W}, to the case of Hopf algebras. 

To be precise, let $B$ be a commutative Hopf algebra over $\Bbbk$, a pair $(P,\alpha)$ consisting of  an algebra extension $\alpha : A\to P$ and a right $B$-comodule algebra $P$ with left $A$-linear coaction, is said to be a right \emph{principal $B$-bundle} provided
$\alpha$ is faithfully flat and the canonical map $\Sf{can}_{\Sscript{P}}: P\tensor{A}P \to P\tensor{}H$, $x\tensor{A}y \mapsto xy_{\Sscript{(0)}} \tensor{} y_{\Sscript{(1)}}$ is  bijective. Notice that if we  translate this definition to the associated affine $\Bbbk$-schemes, then the outcome characterizes in fact the  notion of torsors as it was shown in  \cite[Corollaire 1.7, page 362]{DemGab:GATIGAGGC}, see also \cite[D\'efinition 1.4.1, page 117]{Giraud:1971}. 

Let $(A,\cH)$ be a Hopf algebroid as in subsection \ref{ssec:isotropy} and $\hH$ its associated presheaf of groupoids.  Take an object $x \in A(\Bbbk)$ and consider as before  $P_{\Sscript{x}}=\cH\tensor{A}\Bbbk_{\Sscript{x}}$ the right comodule algebra  over the isotropy Hopf algebra $(\Bbbk_{\Sscript{x}},\cH_{\Sscript{x}})$ with the algebra extension $\alpha_{\Sscript{x}}: A \to P_{\Sscript{x}}$ of equation \eqref{Eq:alphaBetax}. On the other hand denote by $\pP_{\Sscript{x}}$ the presheaf of sets which associated to each algebra $C$ the set  $\pP_{\Sscript{x}}(C) :=\Sf{t}^{-1}\big( \{1_{\Sscript{C}} \circ x\}\big)$ where $\Sf{t}$ is the target of the groupoid $\hH(C)$. 

\begin{lemma}\label{lema:P}
For any $x \in A(\Bbbk)$, the presheaf of sets $\mathscr{P}_{\Sscript{x}}$ is affine, and up to a natural isomorphism, is  represented by the  algebra $P_{\Sscript{x}}$. Furthermore, if $(A,\cH)$ is geometrically transitive, then $(P_{\Sscript{x}}, \alpha_{\Sscript{x}})$ is a principal right $\cH_{\Sscript{x}}$-bundle. 
\end{lemma}
\begin{proof}
The first claim is an immediate verification. The last one is a consequence of Theorem \ref{thm:A}.
\end{proof}

In contrast with the case of transitive groupoids described in subsection \ref{ssec:GPS}, the converse in Lemma \ref{lema:P} is not obvious. Specifically,  it is not automatic to construct a  GT Hopf algebroid from a principal bundle over a Hopf algebra.  In more details, let $(P,\alpha)$ be a right principal $B$-bundle over a Hopf algebra $B$ with extension $\alpha:A \to P$, and 
consider $P\tensor{}P$ as a right $B$-comodule algebra via the diagonal coaction and set $$
\cH:=(P\tensor{}P)^{\Sscript{coinv_{B}}}\,=\, \Big\{ u \in P\tensor{}P|\,\, \varrho_{\Sscript{P\tensor{}P}}(u) = u\tensor{} 1_{\Sscript{B}} \Big\}
$$ 
its coinvariant subalgebra. 
The map $\alpha$ induces two maps $\Sf{s},\Sf{t}: A \to \cH$ which going to be the source and the target. The counity is induced by the multiplication of $P$. The comultiplication is derived from that of $(P,P\tensor{}P)$, however,  not in an immediate way, because  slightly technical assumptions are needed for this. 

Precisely,  consider $\cM:=(P\tensor{}P)\tensor{A}(P\tensor{}P)$ as a right $B$-comodule algebra with the coaction 
$$
\varrho: \cM \longrightarrow \cM \tensor{}B, \quad (x\tensor{}y)\tensor{A}(u\tensor{}v) \longmapsto (x_{\Sscript{(0)}}\tensor{}y_{\Sscript{(0)}}) \tensor{A} (u_{\Sscript{(0)}}\tensor{}v_{\Sscript{(0)}}) \tensor{} x_{\Sscript{(1)}}y_{\Sscript{(1)}} u_{\Sscript{(1)}}v_{\Sscript{(1)}}.
$$
This is a well defined coaction since we know that $P^{\Sscript{coinv_B}}\cong A$. Clearly we have that $\cH\tensor{A}\cH \subseteq \cM^{\Sscript{coinv_B}}$, and under the assumption of  equality we obtain:

\begin{proposition}\label{prop:GTPB}
Let $(P,\alpha)$ be a right principal $B$-bundle over a Hopf algebra $B$ with extension $\alpha: A \to P$. Denote by $\upsilonup: \cH:=(P\tensor{}P)^{\Sscript{coin_{B}}} \to P\tensor{}P$ the canonical injection where $P\tensor{}P$ is a right $B$-comodule algebra via the diagonal coaction. Assume that $\upsilonup$ is a faithfully flat extension and that $\cH\tensor{A}\cH = \cM^{\Sscript{coin_B}}$. 

Then $(A,\cH)$ admits a unique  structure of Hopf algebroid such that $(\alpha,\upsilonup):(A,\cH) \to (P,P\tensor{}P)$ is a morphism of GT Hopf algebroids. 
\end{proposition}
\begin{proof}
First observe that the map $\Sf{s}: A \to \cH$  is a flat extension (and so is $\Sf{t}$) since $\alpha$ and $\upsilonup$ are faithfully flat extension and we have a commutative diagram:
$$
\xymatrix@R=20pt{0 \ar@{->}^-{}[r] & \cH  \ar@{->}^-{\upsilonup}[r] & P\tensor{}P   \\ 0 \ar@{->}^-{}[r] & A \ar@{->}^-{\alpha}[r] \ar@<0.5ex>@{->}^-{\Sf{s}}[u] \ar@<-0.5ex>@{->}_-{\Sf{t}}[u]  & P  \ar@<0.5ex>@{->}^-{}[u] \ar@<-0.5ex>@{->}^-{}[u]  }
$$
of algebra maps. The fact that $(A,\cH)$ admits a coassociative comultiplication follows essentially form the second assumption. Indeed, let $\Delta': P\tensor{}P \to \cM$ be the map which sends $x\tensor{}y \mapsto (x\tensor{}1)\tensor{A}(1\tensor{}y)$, so  it is easily checked that, under the stated assumption, there is a map $\cH \to \cH\tensor{A} \cH$ which completes the diagram:
$$
\xymatrix@R=20pt{ 0 \ar@{->}^-{}[r] & \cH\tensor{A}\cH  \ar@{->}^-{}[r] & \cM \ar@<0.5ex>@{->}^-{\varrho}[rr] \ar@<-0.5ex>@{->}_-{-\tensor{}1}[rr]   & & \cM \tensor{}B \\ 0 \ar@{->}^-{}[r] & \cH \ar@{->}^-{\upsilonup}[r]  \ar@{.>}^-{\Delta}[u] & P\tensor{}P \ar@{->}^-{\Delta'}[u] & & }
$$
This gives a coassociative comultiplication on the $A$-bimodule $\cH$ using the structure of $A$-bimodule derived from the above source and the target $\Sf{s}$, $\Sf{t}$. To check that $\Delta$ is counital one uses the following equalities 
$$
(p\tensor{}1)\tensor{A}(1\tensor{}q)\,=\, \big(p_{\Sscript{(0)}}\tensor{}p_{\Sscript{(1)}}^{\Sscript{-}} q_{\Sscript{(1)}}^{\Sscript{-}}\big) \tensor{A} \big(  p_{\Sscript{(1)}}^{\Sscript{+}}q_{\Sscript{(1)}}^{\Sscript{+}}\tensor{} q_{\Sscript{(0)}}\big) \; \,\in \,(P\tensor{}P) \tensor{A}(P\tensor{}P),
$$
together with the properties of the translation map $\delta: B \to P\tensor{A}P$, $b \mapsto b^{\Sscript{-}}\tensor{A}b^{\Sscript{+}}$  given by the inverse of the canonical map $\Sf{can}_{\Sscript{P}}$. 

With the previous structure maps, $(A,\cH)$ is now a Hopf algebroid such that the pair of maps $(\alpha,\upsilonup): (A,\cH) \to (P,P\tensor{} P)$  is a morphism of Hopf algebroids with codomain a GT Hopf algebroid. 
Lastly, since $\alpha\tensor{}\alpha$ is a faithfully flat extension, $\Sf{s}\tensor{}\Sf{t}: A\tensor{}A \to \cH$ is also faithfully flat, and hence $(A,\cH)$ is  by Theorem \ref{thm:A} a GT Hopf algebroid as well.
\end{proof}

\subsection{GT Hopf algebroids with empty character groupoid}\label{ssec:AK}
Let $(A,\cH)$ be a flat Hopf algebroid over $\Bbbk$ with $A\neq 0$ and $A(\Bbbk) = \emptyset$. For example,  taking any non zero algebra $A$ with $A(\Bbbk) = \emptyset$  and consider the Hopf algebroids given in Examples \ref{exam:HAlgd1} and \ref{exam:HAlgd3}. Now, let $L$ be a field extension of $\Bbbk$ such that $A(L) \neq \emptyset$ and denote by $\oOL: \Algl \to \Algk$ the forgetful functor from the category of commutative $L$-algebras to $\Bbbk$-algebras. Next, we will use the notations of Example \ref{exam:Scalars}. So, fix an algebra map $q \in A(L)$ and denote by $\td{q}  \in \AL(L)=\Algl(\AL, L)$ its image, that is, the $L$-algebra map $\td{q}: \AL \to L$  sending $a\tensor{}l \mapsto q(a)l$. Consider the base extension Hopf algebroid $(\Lq, \Lq\tensor{A}\cH\tensor{A}\Lq):=(\Lq, \Hq)$ over $\Bbbk$, where $\Lq$ is considered as an algebra extension of $A$ via the map $q$.  The associated presheaf of groupoids is denoted by $\hHq$ and its composition with $\oOL$ by $\td{\hHq}:=\hHq \circ \oOL$. In this way, we get  a presheaf of groups
\begin{equation}\label{Eq:Hq}
\td{\hHq}^{\star}: \Algl \longrightarrow \Sf{Grps},\quad  \Big(  R \longrightarrow \td{\hHq}(R)^{\Sscript{z}} \Big)
\end{equation}
where $z : L\to R$ is the $\Bbbk$-algebra map defining $R$ as an object in $\Algl$ and where $\td{\hHq}(R)^{\Sscript{z}}$ is the isotropy group of  the groupoid $\td{\hHq}(R)$ attached to the object $z$. Thus,  for any pair $(R,z)$ as before,  we have  by Example \ref{exam:Base change}  that 
$$
\td{\hHq}(R)^{\Sscript{z}} :=\LR{  (z,g,z) |\,\, g \in \cH(R) \text{ such that }  g \Sf{s} = g \Sf{t} = z q  }
$$
where the multiplication is given as in Example \ref{exam:induced} and the unit is the element $(z,zq\varepsilon,z)$.  

On the other hand, following Lemma \ref{lem:iso}  we can define the presheaf of groups attached to  the Hopf algebroid  $(\AL, \HL)$ over $L$, at the point $\td{q}$. That is, we can consider the presheaf  $\hHL$ associated to $(\AL, \HL)$, and denote by 
\begin{equation}\label{Eq:HLq}
\gGL^{\Sscript{\td{q}}}: \Algl \longrightarrow \Sf{Grps},\quad  \Big(  R \longrightarrow \gGL^{\Sscript{\td{q}}}(R) \Big)
\end{equation}
where $\gGL^{\Sscript{\td{q}}} (R)$ is the isotropy group of the groupoid $\hHL(R)$ at the point $\td{q} \in \AL(L)$. Thus 
$$
\gGL^{\Sscript{\td{q}}} (R):=\LR{ h \in \HL(R) =\Algl(\HL, R)|\,\;  h \sL = h \tL = \td{q} },
$$ 
where the group structure comes from the groupoid $\hHL(R)$.

\begin{proposition}\label{prop:Lq}
Given $(R,z)$ as above, then the following morphisms of groups 
$$
\xymatrix@R=0pt{  \gGL^{\Sscript{\td{q}}} (R) \ar@{->}^-{\phiup_{\Sscript{q, \, R}}}[rr]  & &  \td{\hHq}(R)^{\Sscript{z}} 
\\ h \ar@{|->}[rr] & & (z,\hat{h},z), }
$$
where $\hat{h} \in \cH(R)$ is the $\Bbbk$-algebra map sending $u \mapsto h(u\tensor{}1)$, establish a natural isomorphism 
$$
\phiup_{\Sscript{q}} : \gGL^{\Sscript{\td{q}}} \longrightarrow \td{\hHq}^{\star}
$$ 
of presheaves of groups. In particular, up to a natural isomorphism,  $\td{\hHq}^{\star}$ is represented by the isotropy Hopf $L$-algebra $(L, \hH_{\Sscript{L,\, \td{q}}})$ of the Hopf algebroid $(\AL, \HL)$ at the point $\td{q}$.
\end{proposition}
\begin{proof}
Let us first check that $\phiup_{\Sscript{q, \, R}}$  is a well defined map. Take $h \in \gGL^{\Sscript{\td{q}}} $, then, for every $a \in A$,  we have 
$$
\hat{h} \circ \Sf{s}(a)= h(\Sf{s}(a)\tensor{}1) = h \circ \sL(a\tensor{}1) = h \circ \tL(a\tensor{}1)= h \circ (\Sf{t}(a)\tensor{}1) = \hat{h} \circ \Sf{t}(a) = \td{q}(a) = q(a). 1_{\Sscript{L}} = zq(a). 
$$
Hence $(z, \hat{h}, z) \in  \td{\hHq}(R)^{\Sscript{z}} $. The image by $\phiup_{\Sscript{q, \, R}}$ of the identity  element  $\td{q}\eL$ is $(z, \widehat{\td{q}\eL}, z)=(z, z q \varepsilon , z)$ which is the identity  element of the group $\td{\hHq}(R)^{\Sscript{z}} $. Now, given  $h, h' \in \gGL^{\Sscript{\td{q}}}(R) $ and  $u \in \cH$,  we have that 
$$
\widehat{h h'} (u)  = (h h') (u\tensor{}1) = h'\big((u\tensor{}1)_{\Sscript{(1)}}\big) \, h\big((u\tensor{}1)_{\Sscript{(2)}}\big) = h'\big(u_{\Sscript{(1)}} \tensor{}1\big) \, h\big(u_{\Sscript{(2)}} \tensor{}1\big) = \hat{h'}(u_{\Sscript{(1)}}) \,  \hat{h}(u_{\Sscript{(2)}})  = (\hat{h} \, \hat{h'}) (u)
$$
which implies that $\widehat{h h'} = \hat{h} \, \hat{h'}$.  Therefore, 
$$
(z, \hat{h}, z) \,  (z, \hat{h'}, z) = (z, \hat{h} \hat{h'}, z) = (z, \widehat{hh'}, z), 
$$ 
which shows that $\phiup_{\Sscript{q, \, R}}$ is a morphism of groups.  On the other hand, $\phiup_{\Sscript{q, \, R}}$ is clearly injective and if we take an element  $(z, g, z) \in \td{\hHq}(R)^{\Sscript{z}} $ and set $h = \hat{g}: \HL \to R$ sending $u\tensor{}l \mapsto g(u) \tensor{}l$, then we have that $(z,g, z) = \phiup_{\Sscript{q, \, R}}(h)$.  This shows that $\phiup_{\Sscript{q, \, R}}$ is also surjective, and thus an isomorphism of groups. Lastly, it is immediate to see that $\phiup_{\Sscript{q, \, -}}$ is a natural transformation and so a natural isomorphism as desired.  The particular statement follows directly from Lemma \ref{lem:iso}. 
\end{proof}

\begin{remark}\label{rem:Hq}
Let $(\Lq, \Lq\tensor{A}\cH\tensor{A}\Lq)$ be as above the  base change Hopf algebroid of  $(A,\cH)$ and denote by $(L,\bara{\Hq})$ its quotient Hopf $L$-algebra where $\bara{\Hq}: = \Lq\tensor{A}\cH\tensor{A}\Lq /\langle  \Sf{s} - \Sf{t}\rangle$ is the quotient modulo the Hopf ideal generated by the set 
$\LR{  \Sf{s}(l) - \Sf{t}(l) }_{l \, \in \, L}$. Then the following map of $L$-vector spaces
$$
\xymatrix@R=0pt{  \bara{\Hq}  \ar@{->}[rr]  & &  \hH_{\Sscript{L,\, \td{q}}}  \\ \bara{l\tensor{A}u \tensor{A} l'} \ar@{|->}[rr] && 1\tensor{\AL} (u\tensor{}ll') \tensor{\AL} 1   }
$$
is a surjective morphism of Hopf $L$-algebras.  On the other hand, the presheaf of set $\pP_{\Sscript{\td{q}}}: \Algl \to \Sf{Sets}$  defined as in Lemma \ref{lema:P} for the Hopf algebroid $(\AL, \HL)$ is, up to a natural isomorphisms, represented by the left $\cH$-comodule $L$-algebra $\cH \tensor{A}\Lq$ which under condition $(i)$ of Theorem \ref{thm:A}, becomes a principal $(\cH, \Hq)$-bibundle.
\end{remark}

\begin{proposition}\label{prop:Lk}
Let $(A,\cH)$ be a flat Hopf algebroid over $\Bbbk$ with $A\neq 0$ and $A(\Bbbk) = \emptyset$,  denote by $\hH$ its associated presheaf of groupoids. Consider $L$ a field extension of $\Bbbk$ such that $A(L) \neq \emptyset$.  Assume that  the unit map $\etaup=\Sf{s}\tensor{}\Sf{t}: A\tensor{}A \to \cH$ is a faithfully flat extension. Then
\begin{enumerate}[(1)]
\item $\hH(L)$ is a transitive groupoid;
\item For every $p, q \in A(L)$,  the base change Hopf algebroids $(L_{\Sscript{q}}, \Hq)$ and $(L_{\Sscript{p}},\Hp)$ are weakly equivalent. 
\end{enumerate}
\end{proposition}
\begin{proof}
As we have observed in Remark \ref{rem:Lk}, if $\etaup$ is a faithfully flat extension, then so is $\etaup_{\Sscript{L}}: \AL\tensor{L}\AL \to \HL$. Since $\AL(L) \neq \emptyset$, we have by Theorem \ref{thm:A}, that $(\AL, \HL)$ is geometrically transitive Hopf algebroid. Therefore,  by applying  Proposition \ref{prop:conjugation}, we know that $\HL(L)$ is a transitive groupoid. Now, given $p,q \in A(L)$ we obtain  two objects $\td{p}, \td{q} \in \AL(L)$ of this groupoid. Hence, there exists an $L$-algebra map $h \in \HL(L)$ such that $h \circ \sL = \td{p}$ and $h \circ \tL =\td{q}$. Consider the algebra map $g=\hat{h}: \cH \to L$ sending $u \mapsto h(u\tensor{}1)$, so we have that $g \circ \Sf{s} = p$ and $g \circ \Sf{t} = q$. This proves part $(1)$. As for part $(2)$, we know by Remark  \ref{rem:Lk} that $(A, \cH)$ satisfies condition $(iii)$ of Theorem \ref{thm:A}. Henceforth, the canonical base change maps  $(A, \cH) \to (L_{\Sscript{p}}, \Hp)$  and $(A, \cH) \to (L_{\Sscript{q}}, \Hq)$ are weak equivalences by \cite[Proposition 5.1]{Kaoutit/Kowalzig:14}. Therefore,  $(L_{\Sscript{p}}, \Hp)$ and $(L_{\Sscript{q}}, \Hq)$ are weakly equivalent by applying \cite[Theorem A]{Kaoutit/Kowalzig:14} and this finishes the proof.
\end{proof}

\subsection{Dualizable comodules over GT Hopf algebroids are locally free of constant rank}\label{ssec:VB}
The aim of this subsection is to apply Theorem \ref{thm:A} in order to give an elementary proof of the well know  fact sated in \cite[page 114]{Deligne:1990} which implicitly asserts  that over a GT Hopf algebroid with non empty character groupoid, any comodule which has a locally free fibre with rank $n$, then so are other fibres. An important consequence of this fact is that any dualizable comodule over such a Hopf algebroid is locally free with constant rank.  This is  an algebraic  interpretation of a well known property on representations of transitive groupoid in vector spaces. Namely, if a given representation over such a groupoid  has a finite dimensional fibre,  then so are all other fibres and all the fibres have  the same dimension.
We start by the following general lemma which will be needed below.

\begin{lemma}\label{lema: DFR}
Let $\varphi: R \to T$ be a faithfully flat extension of commutative algebras. Then, for any $R$-module $P$, the following conditions are equivalent.
\begin{enumerate}[(i)]
\item $P$ is locally free $R$-module of constant rank $n$;
\item $P_{\Sscript{\varphi}}:=P\tensor{R}T$ is locally free $T$-module of constant rank $n$.
\end{enumerate}
\end{lemma}
\begin{proof}
By \cite[Proposition 12, page 53, and Th\'eor\`eme 1, page 138]{Bou:AC12}, we only need to check that $P$ is of a constant rank $n$  if and only if so is $P_{\Sscript{\varphi}}$. So let us first denote by $\varphi_{\Sscript{*}}:\Spec{T} \to \Spec{R}$ the associated continuous map of $\varphi$. Denote by $\mathfrak{r}^{\Sscript{R}}_{\Sscript{P}}: \Spec{R} \to \mathbb{Z}$ and $\mathfrak{r}^{\Sscript{T}}_{\Sscript{P_{\varphi}}}: \Spec{T} \to \mathbb{Z}$, the rank functions corresponding, respectively, to $P$ and $P_{\Sscript{\varphi}}$.  

It suffices to check that $\mathfrak{r}^{\Sscript{R}}_{\Sscript{P}}$ is a constant function with value $n$ if and only if $\mathfrak{r}^{\Sscript{T}}_{\Sscript{P_{\varphi}}}$ is a constant function with the same value. 
Given a prime ideal $\fk{a} \in \Spec{T}$, consider the localising algebras $T_{\Sscript{\fk{a}}}$ and $R_{\Sscript{\varphi_{\Sscript{*}}(\fk{a})}}$ at the prime ideals  $\fk{a}$ and $\varphi_{\Sscript{*}}(\fk{a})$. It is clear that we have an isomorphism of $T_{\Sscript{\fk{a}}}$-modules $P\tensor{R}T_{\Sscript{\fk{a}}}\,\cong \, P_{\Sscript{\varphi_*(\fk{a})}} \tensor{R_{\Sscript{\varphi_*(\fk{a})}}}T_{\Sscript{\fk{a}}}$, where $\varphi_{\Sscript{\fk{a}}}: R_{\Sscript{\varphi_*(\fk{a})}} \to T_{\Sscript{\fk{a}}}$ is the associated localisation map of the extension $\varphi$. Therefore, the free modules $P\tensor{R}T_{\Sscript{\fk{a}}}$ and $P_{\Sscript{\varphi_*(\fk{a})}}$ have the same rank.    Hence, we have 
$\mathfrak{r}^{\Sscript{R}}_{\Sscript{P}} \big(\varphi_{\Sscript{*}}(\fk{a}) \big) \,=\,  \mathfrak{r}^{\Sscript{T}}_{\Sscript{P_{\varphi}}}(\fk{a})$, for any $\fk{a} \in \Spec{T}$, and so $\mathfrak{r}^{\Sscript{R}}_{\Sscript{P}} \circ \varphi_{\Sscript{*}} \,=\,  \mathfrak{r}^{\Sscript{T}}_{\Sscript{P_{\varphi}}}$.  
This shows that if $\mathfrak{r}^{\Sscript{R}}_{\Sscript{P}}$ is a constant function with value $n$, then so is $\mathfrak{r}^{\Sscript{T}}_{\Sscript{P_{\varphi}}}$. The converse also hods true since we know that $\varphi_{\Sscript{*}}$ is surjective, because of the faithfully flatness of $\varphi$, and this finishes the proof. 
\end{proof}

\begin{proposition}\label{prop:DFR}
Let $(A,\cH)$ be a flat Hopf algebroid with $A\neq 0$ and $A(\Bbbk) \neq \emptyset$. Assume that $(A,\cH)$ is  geometrically transitive, and let $M$ be a (right) $\cH$-comodule whose underlying $A$-module is finitely generated and projective. Given two objects $x \in A(S)$ and $y \in A(T)$, then the following are equivalent
\begin{enumerate}[(i)]
\item $M_{\Sscript{x}}:=M\tensor{A}S$ is locally free $S$-module of constant rank $n$;
\item $M_{\Sscript{y}}:=M\tensor{A}T$ is locally free $T$-module of constant rank $n$.
\end{enumerate}
\end{proposition}
\begin{proof}
Let us first show that the stated conditions are equivalent when $R=S=T$. In this case we know by Theorem \ref{thm:A}, that the objects $x, y \in A(R)$ are locally isomorphic. Therefore,  there exists a faithfully flat extension $p: R\to R'$  such that $M_{\Sscript{\td{x}}}=M\tensor{\td{x}}R'$ is isomorphic as $R'$-module to $M_{\Sscript{\td{y}}}=M\tensor{\td{y}}R'$ , where $\td{x}= p \circ x$ and $\td{y}= p \circ y$. Thus, $M_{\Sscript{\td{x}}}$ and $M_{\Sscript{\td{y}}}$ they have the same rank function. 

On the other hand,  by  applying  Lemma \ref{lema: DFR} to $M_{\Sscript{x}}$, we get that $M_{\Sscript{\td{x}}}$ is locally free $R'$-module of constant rank $n$ if and only if $M_{\Sscript{x}}$ is locally free $R$-module of constant rank $n$.  The same result hold true using $M_{\Sscript{y}}$ and $M_{\Sscript{\td{y}}}$.
Therefore, $M_{\Sscript{x}}$  is locally free $R$-module of constant rank $n$ if and only if so is $M_{\Sscript{y}}$.

For the general case $S \neq T$, consider $R:=T\tensor{}S$ and set the algebra maps $\bara{x}:= \iota_{\Sscript{S}} \circ x$, $\bara{y}:= \iota_{\Sscript{T}}\circ y$, where $\iota_{\Sscript{S}}: S \rightarrow R \leftarrow T: \iota_{\Sscript{T}}$ are the obvious maps. By the previous case, we know that $M_{\Sscript{\bara{x}}}$ is locally free $R$-module of constant rank $n$ if and only if so is $M_{\Sscript{\bara{y}}}$. 
Now by Lemma \ref{lema: DFR}, we have, from one hand, that $M_{\Sscript{\bara{x}}}$ is locally free $R$-module of constant rank $n$ if and only if $M_{\Sscript{x}}$ is locally free $S$-module of constant rank $n$, and from the other, we have that $M_{\Sscript{\bara{y}}}$ is locally free $R$-module of constant rank $n$ if and only if $M_{\Sscript{y}}$ is locally free $T$-module of constant rank $n$. Therefore, $M_{\Sscript{x}}$ is locally free $S$-module of constant rank $n$ if and only if $M_{\Sscript{y}}$ is so as $T$-module.
\end{proof}

As a corollary of Proposition \ref{prop:DFR}, we have:
\begin{corollary}\label{coro:DFR}
Let $(A,\cH)$ be a flat Hopf algebroid with $A\neq 0$ and $A(\Bbbk) \neq \emptyset$. Assume that $(A,\cH)$ is  geometrically transitive. Then every dualizable (right) $\cH$-comodule is a locally free $A$-module of constant rank.  In particular, given a dualizable right $\cH$-comodule $M$ and two distinct object $x\neq y \in A(\Bbbk)$, then $M_{\Sscript{x}}$ and $M_{\Sscript{y}}$ have the same dimension as $\Bbbk$-vector spaces. 
\end{corollary}

\noindent\textbf{Aknowledgements:} The author would like to thank the referee for  her/his thorough review and highly appreciate the comments and 
suggestions, which significantly contributed to improving the paper.

\end{document}